\documentclass[10.5pt,reqno]{amsart}
\usepackage{longtable}
\usepackage{hyperref}
\usepackage[T1]{fontenc}
\usepackage[utf8]{inputenc}
\usepackage[english]{babel}
\usepackage{textcomp}
\usepackage{dsfont}
\usepackage{latexsym}
\usepackage{amssymb}
\usepackage{amsthm}
\usepackage{amsmath}
\DeclareMathAlphabet{\mathpzc}{OT1}{pzc}{m}{en}
\usepackage{yfonts}
\usepackage{xfrac}
\usepackage{newlfont}
\usepackage{graphicx}
\usepackage{mathtools}
\usepackage{comment}
\usepackage{indentfirst}
\usepackage{braket}
\usepackage{mathrsfs}
\usepackage{xcolor}
\usepackage{fixmath}

\newcommand{\atomic}{(L)}
\newcommand{\atomics}{(L')}

\usepackage{scalerel}[2014/03/10]
\usepackage[usestackEOL]{stackengine}
\newcommand{\dashint}{\,\ThisStyle{\ensurestackMath{%
			\stackinset{c}{.2\LMpt}{c}{.5\LMpt}{\SavedStyle-}{\SavedStyle\phantom{\int}}}%
		\setbox0=\hbox{$\SavedStyle\int\,$}\kern-\wd0}\int}

\DeclareMathOperator{\card}{Card}
\DeclareMathOperator{\supp}{Supp}
\DeclareMathOperator{\tr}{Tr}
\DeclareMathOperator{\Hol}{Hol}

\renewcommand{\Re}{\mathrm{Re}\,}
\renewcommand{\Im}{\mathrm{Im}\,}
\newcommand{\Supp}[1]{\supp\left( #1\right) }
\newcommand{\ee}{\mathrm{e}}
\newcommand{\vect}[1]{\mathbf{{#1}}}
\newcommand{\dd}{\mathrm{d}}
\newcommand{\osc}{\mathrm{osc}}

\DeclarePairedDelimiter{\abs}{\lvert}{\rvert}

\DeclarePairedDelimiter{\norm}{\lVert}{\rVert}

\let\originalleft\left
\let\originalright\right
\renewcommand{\left}{\mathopen{}\mathclose\bgroup\originalleft}
\renewcommand{\right}{\aftergroup\egroup\originalright}

\newcommand{\N}{\mathds{N}}
\newcommand{\Z}{\mathds{Z}}
\newcommand{\C}{\mathds{C}}
\newcommand{\R}{\mathds{R}}

\newcommand{\Uf}{\mathfrak{U}}

\newcommand{\Ac}{\mathcal{A}}
\newcommand{\Bc}{\mathcal{B}}
\newcommand{\Ec}{\mathcal{E}}
\newcommand{\Fc}{\mathcal{F}}
\newcommand{\Hc}{\mathcal{H}}
\newcommand{\Lc}{\mathcal{L}}
\newcommand{\Nc}{\mathcal{N}}
\newcommand{\Pc}{\mathcal{P}}
\newcommand{\Rc}{\mathcal{R}}
\newcommand{\Sc}{\mathcal{S}}

\newcommand{\meg}{\leqslant}
\newcommand{\Meg}{\geqslant}
\newcommand{\eps}{\varepsilon}
\newcommand{\mi}{\mu}

\newcommand{\leftexp}[2]{{\vphantom{#2}}^{#1}{#2}} 
\newcommand{\trasp}{\leftexp{t}}
\newcommand{\Lin}{\mathscr{L}}

\title[Bergman projectors on Siegel domanis]{Boundedness of Bergman
  projectors on homogeneous Siegel domains}

\date{}
\begin{document}

\theoremstyle{definition}
\newtheorem{deff}{Definition}[section]

\newtheorem{oss}[deff]{Remark}
\newtheorem{ass}[deff]{Assumptions}
\newtheorem{nott}[deff]{Notation}
\newtheorem{exam}[deff]{Examples}

\theoremstyle{plain}
\newtheorem{teo}[deff]{Theorem}
\newtheorem{lem}[deff]{Lemma}
\newtheorem{prop}[deff]{Proposition}
\newtheorem{cor}[deff]{Corollary}
\author[M. Calzi, M. M. Peloso]{Mattia Calzi, 
Marco M. Peloso}

\address{Dipartimento di Matematica, Universit\`a degli Studi di
  Milano, Via C. Saldini 50, 20133 Milano, Italy}
\email{{\tt mattia.calzi@unimi.it}}
\email{{\tt marco.peloso@unimi.it}}

\thanks{{\em Math Subject Classification 2020} 32A15, 32A37, 32A50, 46E22.}
\thanks{The authors are members of the
  Gruppo Nazionale per l'Analisi Matematica, la Probabilit\`a e le
  loro Applicazioni (GNAMPA) of the Istituto Nazionale di Alta
  Matematica (INdAM)} 

 \begin{abstract}
   In this paper we study the boundedness of Bergman projectors on
   weighted   Bergman spaces on homogeneous Siegel domains of Type II.  As it appeared
   to be a natural approach in the special case of tube domains
   over irreducible symmetric  cones, we study such boundedness on the
   scale of mixed-norm weighted Lebesgue spaces.
   The sharp range for
   the boundedness of  such operators is essentially known only in the case of  tube  domains
   over   Lorentz cones.
   
In this paper we prove that the boundedness of such Bergman projectors
is equivalent to variuos notions of atomic decomposition, duality, and characterization of boundary values   of the mixed-norm weighted 
Bergman spaces,  
extending results moslty  known only in the case of tube domains over
irreducible symmetric  cones.   Some  of our results are new even in the latter simpler context.

We also study the simpler, but still quite interesting, case of the ``positive'' Bergman
projectors, the integral operator in which the Bergman kernel is
replaced by its absolute value.
We   provide a useful  characterization which was previously known for tube domains.
\end{abstract}
\maketitle

\section{Introduction}
In this paper we study the boundedness of the Bergman projectors
$P_{\vect s'}$ on the mixed-norm Lebesgue spaces $L^{p,q}_{\vect s}(D)$
on a Siegel domain $D$ of Type II, where $p,q \in [1,\infty]$ and
$\vect s,\vect s'\in \R^r$.  We refer to Section~\ref{Prelim:sec} for the relevant
definitions and the main basic properties.  Here we just mention that
$r\in\N$ is the {\em rank} of the homogeneous cone $\Omega$  that appears in the
defintion of the underlying  Siegel domain $D$.  One of the
main features of the present work is the generality of our setting, in
which the analysis and geometric properties of $\Omega$ are crucial
and, generally speaking, still quite elusive.

We shall now try to describe our results to a broader audience than just   that of   the
specialists and this will 
  lead to  a somewhat longer   introduction.
We hope that it will make it easier to understand the problems we address
and that it may serve as a general reference.
\medskip

The simplest case of the domains we consider is the upper half-plane
$\C_+ \coloneqq \Set{z\in\C\colon \Im z>0}$.  For $p,q\in (0,\infty]$
and $s\in\R$, consider the Lebesgue spaces
\[
L^{p,q}_s(\C_+)=\Set{f\text{ measurable in }\C_+\colon  \int_0^\infty 
  \left(\int_\R \abs{f(x+iy)}^p\,\dd x\right)^{q/p}y^{qs-1} \,\dd  y<\infty} ,
\] 
with the obvious modifications when $\max(p,q)=\infty$. The  mixed-norm
Bergman spaces are defined as
\[
A^{p,q}_s(\C_+) \coloneqq L^{p,q}_s(\C_+)\cap \Hol(\C_+),
\] 
where $\Hol(\C_+)$ denotes the space of holomorphic functions on $\C_+$.
For the time
being, we simply write $A^{p,q}_s$ in place of $A^{p,q}_s(\C_+)$. 
Then, $A^{p,q}_s=\{0\}$ if $s<0$, and also if $s=0$ and $q<\infty$.
In addition, $A^{p,\infty}_0$ is the classical Hardy space
$H^p(\C_+)$. In order to avoid trivialities, from now  on we shall
assume $s>0$. 
In addition, for simplicity, in this introductory section we shall   generally   limit
ourselves to the case $p,q\Meg 1$. 
Then, the   $A^{p,p}_s$ are the classical 
 weighted Bergman spaces, the unweighted case corresponding to the choice
$s=1/p$. The   $A^{p,q}_s$ 
embed continuously into $\Hol(\C_+)$, so that $A^{p,q}_s$ is a  closed subspace of
$L^{p,q}_s(\C_+)$. In particular, the $A^{2,2}_s$ are reproducing kernel Hilbert spaces.  Their
reproducing kernel, that is, the weighted Bergman kernel, is the
kernel function
\[
K_s(z,w) = c_s (z-\overline w)^{-1-2s} 
\]
where $c_s=s (2i)^{2s+1}/\pi$.   The orthogonal  projection of
$L^{2,2}_s(\C_+)$ onto $A^{2,2}_s$ is the Bergman projector $\Pc_s$
given by
\[
\Pc_s f(z) = \int_{\C_+} f(w) K_s(z,w) (\Im w)^{2s-1}\dd \Hc^2(w),
\]
where $\Hc^2$ denotes the (suitably normalized) $2$-dimensional Hausdorff measure on $\C_+$, that is, Lebesgue measure.
One of the main questions in the theory of holomorphic function spaces
is the boundedness of projection operators.   In this setting, it is
known the
operator $\Pc_{\tilde s}$ induces a continuous endomorphism of $
L^{p,q}_s(\C_+)$ if and only if $ 2 \tilde s>s$ see, e.g.,~\cite[Proposition
5.20 and Corollary 5.27]{CalziPeloso}.\footnote{When $q=1$, the
  cited result only proves that $2\tilde s\Meg s$. Nonetheless, if the
  assertion held for $2 \tilde s=s$, then the argument below would show that
  the Hardy space $H^{p'}(\C_+)$ is canonically antilinearly
  isomorphic to $(A^{p,1})'$; this latter fact is false, as one may
  see by inspection of the boundary values (cf.,
  e.g.,~\cite[Proposition 5.12]{CalziPeloso}). }  

We also define an associated `positive' operator  by
\[
\Pc_{s,+}f(z) = \int_{\C_+} f(w) |K_s(z,w)| (\Im w)^{2s-1}\dd\Hc^2(w) .
\]
Clearly, the boundedness of $\Pc_{s,+}$ implies the boundedness of
$\Pc_{s}$ on
the same spaces.  In the case of $\C_+$, it turns out that $\Pc_{\tilde s}$
is bounded on $L^{p,q}_s(\C_+)$ if and only if $\Pc_{\tilde s,+}$ is, but
this equivalence fails in more general settings, as we shall soon
discuss.   

 The boundedness of the Bergman projector $\Pc_{\tilde s}$ is tighly
related to the characterization of the dual space $(A^{p,q}_s)'$ as
the space $A^{p',q'}_{2\tilde s-s}$,   where $p'$ denotes the conjugate exponent of $p$, that is, $p'=p/(p-1)$:   on the one hand, it is readily seen that the
continuity of a Bergman projector implies a characterization of the
dual of $A^{p,q}_s$, at least when $p,q<\infty$, by standard arguments
(cf., e.g.,~\cite[Theorem 5.2]{Bekolleetal}); on the other hand, more
subtle arguments show that also che converse inplication holds, even in
greater generality\footnote{Of course, the equivalence is trivial in
  this case, as both properties are (separately) known to be true in
  the same cases.} (cf.~\cite[Theorem 1.6]{Bekolleetal2} for the case
$p=q$, and  Corollary~\ref{cor:1} below for the general case).  As a
matter of fact, even for $p,q\in (0,\infty)$, the sesquilinear form 
\[
A^{p,q}_s\times A^{p',q'}_{\tilde s} \ni(f,g) \to \int_{\C_+} f(z)\overline{g(z)}
\, (\Im z)^{s+\tilde s+(\frac1p-1)_+ -1} \, \dd \Hc^2(z)
\]
is continuous and induces an antilinear isomorphism of
$A^{p',q'}_{\tilde s}$ onto $(A^{p,q}_s)'$ for every $s,\tilde s>0$.
Notice that we set $p'=\max(1,p)'$ for every $p\in(0,\infty]$. 

The boundedness of   Bergman projectors is also strictly related to the
existence of suitable {\em atomic decompositions} for the Bergman spaces 
$A^{p,q}_s$.  For a fixed $R>0$,   consider the family $(Q^{(R)}_{j,k})_{ j,k\in\Z}$ of open boxes
in $\C_+$ defined by
\[
Q_{j,k}^{(R)} \coloneqq
\big(2^{kR}Rj,2^{kR}R(j+1)\big)\times\big(2^{kR},2^{(k+1)R}\big), 
\]
Define
\[
\ell^{p,q}(\Z)=\Big\{(\lambda_{j,k})\in\C^{\Z\times\Z}
  \colon \big\| \|(\lambda_{j,k})_j\|_{\ell^p}\big\|_{\ell^q}<\infty \Big\},
\]
and consider the mapping
\[
\Ac\colon \ell^{p,q}(\Z)\ni (\lambda_{j,k})\to \sum_{j,k}
\lambda_{j,k} K_{\tilde s}(\,\cdot\,, w_{j,k,R}) (\Im w_{j,k,R})^{2\tilde s+1-s-\frac1p}
 \in A^{p,q}_s ,
\]
where the $w_{j,k,R} \in Q^{(R)}_{j,k}$ are fixed, but otherwise
arbitrary.  When $p,q\in (0,\infty]$, it is known that this mapping
is continuous (and onto, for  $R$   small) if and only if $2
\tilde s>s+\left(\frac 1 p-1\right)_+ $ (cf.~\cite{RicciTaibleson} and
also~\cite[Lemma 3.29 and Thereom 3.34]{CalziPeloso}).\footnote{Again,
  when $q\meg 1$, the cited results only show that $2 \tilde s\Meg 
  s+\left(\frac 1 p-1\right)_+$, but a direct argument involving
  duality shows that the weak inequality cannot occur (cf.~also
  Corollary~\ref{cor:4} below). } 
Then, it turns out that (for $p,q\Meg 1$) $\Ac$ is bounded (and onto,
for  $R$  small) if and only if $P_{\tilde s}$ is bounded on
$A^{p,q}_s$.  

Observe that the   $A^{p,p}_{s/p}$  ($s>0$, $p\in (0,\infty)$) are all
weighted Bergman spaces corresponding to the same measure. In other
words, $A^{p,p}_{s/p}=L^p((\Im z)^{s-1}\,\dd \Hc^2(z))\cap
\Hol(\C_+)$. For this reason, it is somewhat more natural to
investigate the continuity of the Bergman projector $\Pc_{s/2}$ on
$A^{p,p}_{s/p}$, rather than that of the other projectors $\Pc_{\tilde
  s}$. Nonetheless, even though the general problem for the operators
$\Pc_{\tilde s}$ has been somewhat  more rarely considered, the
corresponding problem for atomic decomposition has usually been
considered in the general case (or, at least, with uncomparable
restrictions). Because of the deep connection between the continuity
of Bergman projectors and atomic decomposition (cf., e.g.,
Corollary~\ref{cor:1}), it is therefore quite natural to consider all
the Bergman projectors $\Pc_{\tilde s}$. Besides that, considering the
general problem has also some technical advantages.  

Analogously, considering the general mixed-norm family of spaces
$A^{p,q}_s$ instead of the `pure-norm' $A^{p,p}_s$ has some technical
advantages, as it both sheds light on some phenomena, such as the different
behaviour of the operators $\Pc_{\tilde s}$ and $\Pc_{\tilde s,+}$, since the continuity of the latter on
$L^{p,q}_s$ does \emph{not} depend on $p$
(cf.~Theorem~\ref{prop:7}), and allows to apply Fourier techniques
to the spaces $L^{2,q}_s$ by means of the Plancherel formula. 

 We now discuss the characterization of the boundary values of
functions in $A^{p,q}_s$, that is,  of the limits $f_0$ of $f_y$
for $y\to 0^+$ (in a suitable topology), where $f_y\coloneqq
f(\,\cdot\,+iy)$ for $y>0$. It turns out that the space of boundary 
values of the functions in $A^{p,q}_s$ is the closed subspace of
the  homogeneous Besov space $\dot{B}^{-s}_{p,q}(\R)$ whose elements have
Fourier transform supported in $[0,\infty)$.  The mapping
\[
\Bc \colon  A^{p,q}_s\ni f \mapsto f_0 \in \dot{B}^{-s}_{p,q}(\R)
\]
is an isomorphism onto its image.

Besides that, we observe that the operators $f \mapsto f^{(k)}$ induce
isomorphisms of $A^{p,q}_s$ onto $A^{p,q}_{s+k}$ for every $s>0$,
$p,q\in (0,\infty]$, and $k\in\N$. The continuity of the
aforementioned operators is an easy consequence of Cauchy's estimates
and extends to more general spaces without difficulty. The continuity of the inverse
operator then appears as a Hardy-type inequality (cf.,
also,~\cite[Section 1.6]{Bekolleetal}), and its validity in more
general contexts has been proved to be equivalent to the continuity of
Bergman projectors, under suitable circumstances (cf.~\cite[Theorem
5.2]{Bonami},~\cite[Theorem 1.3]{Bekolleetal2}, and~\cite[Corollaries
5.16 and 5.28]{CalziPeloso}). 
\medskip

A very natural extesion of the upper half-plane to several complex
variables are the {\em tube domains over symmetric cones}.  Let
$\Omega\subseteq\R^m$ be an open convex cone not containing affine
lines. The cone $\Omega$ is said to be homogeneous if the subgroup of
$GL(m,\R)$ that preserves $\Omega$ acts transitively on $\Omega$
itself. 
Then, $\Omega$ is said to be symmetric if it coincides
with its dual $\Omega'$, where
\[
\Omega' \coloneqq \Set{ \lambda\in \R^m\colon \forall
	h\in\overline\Omega\setminus\{0\} \:\:\langle h,\lambda\rangle>0},
\]
 for a suitable scalar product on $\R^m$. 
The tube domain over the cone $\Omega$ is then
\[
T_\Omega \coloneqq \Set{ z\in \C^m\colon \Im z\in \Omega}.
\]
It is possible to define a natural polynomial function $\Delta$ on $\Omega$  such that the Bergman kernel for the unweighted Bergman space
$A^{2,2}(T_\Omega)$ is given by  $c\Delta^{-2m/r}\Big(\frac{z-\overline
  w}{i}\Big)$ for a suitable $c>0$; cf.~Section~\ref{Prelim:sec} for more information. 
A natural family of mixed-norm Lebesgue spaces may  then be defined as
\[
L^{p,q}_s(T_\Omega)=\Set{f\text{ measurable in }T_\Omega\colon  \int_\Omega
  \left(\int_{\R^m} \abs{f(x+iy)}^p\,\dd\Hc^m
    (x)\right)^{q/p}\Delta^{qs-\frac mr}(y) \,\dd\Hc^m (y)<\infty} .
\]
Notice that the invariant measure $\dd y/y$ on $(0,\infty)$ is now
replaced by the measure $\Delta^{-\frac mr}(y)\,\dd\Hc^m (y)$, which is
invariant under the group of linear tranformations that preserve
$\Omega$. 

  To the cone $\Omega$ is associated a positive integer $r$, which is
called the {\em rank} of  $\Omega$, see
Section~\ref{Prelim:sec}.
The
only homogeneous cone of rank 1 is the half-line $\R_+^*\coloneqq(0,\infty)$, and thus it
is obvious that $\C_+=T_{(0,\infty)}$ is the only tube domain over a
symmetric cone in one dimension.
Apart from the reducible cone $(\R_+^*)^2$, the only (homogeneous, and actually) symmetric cones of rank $2$ are, up to isomorphisms, the Lorentz cones
\[
\Omega_L\coloneqq\Set{h\in\R^m\colon h_1>0, h_1^2-h_2^2-\cdots -h_m^2>0 }.
\]
These cones constitute perhaps the simplest, yet typical, examples of symmetric cones. The first typical examples of (non-symmetric) homogeneous cones arise in rank $3$.

The  problem of determining  the   boundedness of the Bergman
projectors on Bergman spaces
on the tube domains over   Lorentz cones
was undertaken by D.\ Bekoll\'e and A.\ Bonami~\cite{BekolleBonami},
where they obtained the sharp range of boundedness of the unweighted positive
projector $\Pc_+$.

  It was later shown in~\cite{Bekolleetal3}, for tube domains over Lorentz cones, and then in~\cite{BekolleBonamiGarrigosRicci}, for tube domains over irriducible
symmetric cones,  that the  Bergman
projectors $\Pc_s$ may be bounded even when $\Pc_{s,+}$ is unbounded, thus
showing that a finer analysis of the Bergman projectors must take cancellations into account.  
The projectors on this class of domains provide the only known
examples of such phenomenon, making  their analysis of great interest. 

In order to exploit the cancellations of the weighted Bergman kernels,
the authors of~\cite{Bekolleetal3} considered the mixed-normed spaces $L^{2,q}_s(T_\Omega)$
and used the Fourier transform in the ``horizontal'' variable
$x\in\R^m$.  Then, interpolating with the range of boundedness
of the operator $\Pc_{s,+}$, they obtained a larger range of boundedness for
$\Pc_s$.   See also~\cite{Nana} for the case of homogeneous Siegel domains of type II.
More precise arguments, based on suitable decoupling-type inequalities on the relevant cone, were later developed in~\cite{BekolleBonamiGarrigosRicci}. These methods, combined with the sharp
$\ell^2$-decoupling inequality of Bourgain and
Demeter~\cite{BourgainDemeter}, led to a precise characterization of the boundedness of Bergman projectors on the spaces $L^{p,q}_s$ (cf.~\cite{BekolleGonessaNana}).

The preceding problems have also been considered for mixed norm
weighted Bergman spaces on homogeneous Siegel domains of type II
(notice, though, that two non-equivalent classes of mixed  norm
weighted Bergman spaces   on Siegel domains which are \emph{not} of
tube type  have been considered). See Section~\ref{Prelim:sec} for
more information.  

Various relationships between the aforementioned properties have been
studied at different levels of generality. As a non-exhaustive review
of the literature, we mention the following works. In
\cite{RicciTaibleson} the authors studied the mixed-normed Bergman
spaces in $\C_+$ and provided a characterization of the dual   as a
consequence of atomic decomposition. In~\cite{Bonami}  some
connections between the boundedness of Bergman projectors, generalized
Hardy's inequalities, and the characterization of the boundary values
were highlighted. In~\cite{Bekolleetal,BekolleIshiNana} the atomic
decomposition is seen as a consequence of the characterization of the
dual, hence of the continuity of Bergman projectors, while
in~\cite{Bekolleetal3} the continuity of the Bergman projectors
is obtained as a consequence of the characterization of boundary
values; \cite{Bekolleetal2}, where generalized Hardy's inequality, the
continuity of Bergman projectors, and the characterization of the
duals were proved to be equivalent in some situations. Finally,
in~\cite{CalziPeloso}
the above properties are proved to be
equivalent (in complete generality, but) in a suitable weak sense.
We also mention~\cite{Garrigos-Nana} where the authors provide a necessary and
a sufficient condition for the boundedness of the positive Bergman
projectors on Siegel domains of Type II, and \cite{CP1,CP2} for some
related recent works on Bergman spaces on homogeneous Siegel domains
of Type II. 
\medskip

The paper is organized as follows.  In Section~\ref{Prelim:sec} we 
provide some preliminary material on homogeneous Siegel domains
of Type II.
In Section~\ref{sec:3}, we extend some previous
results (cf.~\cite[Theorem 4.10]{Bekolleetal}) on the continuity of
$\Pc_{s,+}$. 
In Section~\ref{sec:4}, we present a rather general and complete treatment of the equivalence
between the various problems presented above (except for what concerns
 generalized Hardy's inequalities) for $p,q\in [1,\infty]$. We
shall also indicate some results for the general case $p,q\in
(0,\infty]$, even though the lack of duality arguments does not allow
a complete treatment in this more general situation. 

In Section~\ref{sec:5}, we shall consider how complex interpolation
interacts with the preceding properties, thus extending some results
of~\cite{BekolleGonessaNana3}. In Section~\ref{sec:6} we
shall prove some transference results between weighted Bergman spaces
on a Siegel domain of type II and the corresponding tube domain, thus
extending some results of~\cite{BekolleGonessaNana}. 

Finally, Section~\ref{Proof:sec} is devoted to the proof of Proposition~\ref{prop:13}, which extends~\cite[Proposition 4.34]{BekolleBonamiGarrigosRicci} (cf., also,~\cite{BekolleGonessaNana}) to the present setting.

\section{Preliminaries}\label{Prelim:sec}

\subsection{Homogeneous Siegel Domains}

We fix a complex Hilbert space $E$ of dimension $n$, a real Hilbert
space of dimension $m$, an open convex cone $\Omega\subseteq F$ not
containing affine lines, and an $\overline \Omega$-positive
non-degenerate hermitian form $\Phi\colon E\times E\to F_\C$, where
$F_\C$ denotes the complexification of $F$ (  i.e., $ \Phi(\zeta,\zeta)\in \overline{\Omega}\setminus\{0\}$ for every non-zero $\zeta\in E$ ). We shall assume that there
is a triangular subgroup\footnote{This means that there is a basis of
  $F$ with respect to which all elements of $T_+$ are represented by
  an upper triangular matrix. Equivalently, all eigenvalues of every
  element of $T_+$ are real (cf.~\cite{Vinberg2}).} $T_+$ of $GL(F)$
which acts simply transitively on $\Omega$, and that for every $t\in
T_+$ there is $g\in GL(E)$ such that $t\Phi=\Phi(g\times g)$. 
Observe that $T_+$ is uniquely determined up to the conjugation by a linear automorphism of $F$ which preserves $\Omega$ (cf~\cite{Vinberg2,Vinberg}).
We define  $\Phi(\zeta)\coloneqq \Phi(\zeta,\zeta)$ for every $\zeta\in E$, 
\[
\rho\colon E\times F_\C\ni (\zeta,z)\mapsto \Im z-\Phi(\zeta)\in F
\]
and $D\coloneqq \rho^{-1}(\Omega)$, so that $D$ is a homogeneous Siegel domain of type II.

We denote by $\Omega'\coloneqq \Set{\lambda\in F'\colon \forall h\in
  \overline \Omega \setminus \Set{0}\:\: \langle \lambda,h\rangle>0}$
the dual cone of $\Omega$, so that $T_+$ acts simply transitively on
$\Omega'$ by tranposition (cf.~\cite[Theorem 1]{Vinberg}). We denote
by $t\cdot h$ and $\lambda \cdot t$ the actions of $t\in T_+$ on $h\in
\Omega$ and $\lambda\in \Omega'$. We use the same notation for the
extensions of the actions of $T_+$ to $F_\C$ and $F'_\C$. 

Then, there are $r\in \N$ and a homomorphism $\Delta$ of $T_+$ onto
$(\R_+^*)^r$ with kernel $[T_+,T_+]$ (cf.~\cite[\S\
2.1]{CalziPeloso}), so that the characters of $T_+$ are of the form
$\Delta^{\vect s}=\prod_{j=1}^r \Delta_j^{s_j}$, $\vect s\in
\C^r$. Once we fix two base points $e_\Omega\in \Omega$ and
$e_{\Omega'}\in \Omega'$, we may then define `generalized power
functions' $\Delta_\Omega^{\vect s}$ and $\Delta_{\Omega'}^{\vect s}$
on $\Omega$ and $\Omega'$, respectively, so that 
\[
\Delta_\Omega^{\vect s}(t \cdot e_\Omega)=\Delta^{\vect s}(t)=\Delta^{\vect s}_{\Omega'}(e_{\Omega'}\cdot t)
\]
for every $t\in T_+$. Then, $\Delta^{\vect s}_\Omega$ and
$\Delta^{\vect s}_{\Omega'}$ extend to holomorphic functions on
$\Omega+i F$ and $\Omega'+i F'$, respectively, for every $\vect s\in
\C^r$ (cf.~\cite[Corollary 2.25]{CalziPeloso}). 
In order to simplify the notation, we shall introduce two orderings on $\R^r$:
\[
\vect{s}\meg \vect{s'} \iff \forall j=1,\dots,r\:\: s_j\meg s_j' \iff \vect {s'}-\vect s\in \R_+^r
\]
and
\[
\vect{s}\preceq \vect{s'} \iff \vect s=\vect{s'} \lor\forall
j=1,\dots,r\:\: s_j< s_j' \iff \vect {s'}-\vect s\in \Set{\vect 0}\cup
(\R_+^*)^r, 
\]
which coincide only when $r=1$.

We shall then assume that $\Delta$ is chosen so that
$\Delta_\Omega^{\vect s}$ (or, equivalently, $\Delta^{\vect
  s}_{\Omega'}$) is bounded on the bounded subsets of $\Omega$ if and
only if $\vect s\Meg \vect 0$, that is, $\vect s\in \R_+^r$
(cf.~\cite[Lemma 2.34]{CalziPeloso}). We shall further assume that
$\Delta^{\vect s}(t_r)=r^{\sum_j s_j}  $
for every $r>0$ and for every $\vect s\in \C^r$, where $t_r$ is the unique element of $T_+$ such that $t_r \cdot x=r x$ for every $x\in F$ (notice that $T_+$
necessarily contains elements which act as homotheties). Notice that
this does not determine $\Delta$ completely; nonetheless, if $\Delta'$
is another homomorphism with the same properties as $\Delta$, then
there is a permutation $\sigma$ on $\Set{1,\dots, r}$ such that
$\Delta'_j=\Delta_{\sigma(j)}$ for every $j=1,\dots,r$.  This allows
us to use the results of~\cite{CalziPeloso},   where specific choices of $T_+$ and $\Delta$ are made,   with only minor
modifications.    Besides that, we hope that this axiomatic approach may help comparing the following formulae and results with those appearing in the literature, where a number of different conventions appear.

Observe that there are $\vect d\prec \vect 0$, $\vect b\meg\vect0$, and $\vect m,\vect {m'}\Meg \vect0$ such that the following hold:
\begin{itemize}
	\item the measures 
	\begin{equation}\label{eq:4}
		\nu_\Omega\coloneqq\Delta^{\vect
			d}_\Omega\cdot \Hc^m, \qquad  \nu_{\Omega'}\coloneqq\Delta^{\vect d}_{\Omega'}\cdot\Hc^m, \qquad \text{and}\qquad \nu_D\coloneqq(\Delta^{\vect b+2 \vect
			d}_\Omega\circ \rho) \cdot \Hc^{2n+2m},
	\end{equation}
	 where $\Hc^k$
          denotes the $k$-th dimensional Haudorff measure (that is, Lebesgue measure), are invariant under all linear
          automophisms of $F$ and $F'$, and under all biholomorphisms
          of $D$, respectively (cf.~\cite[Propositions 2.19 and
          2.44]{CalziPeloso} and~\cite[Proposition
          I.3.1]{FarautKoranyi}); 
	
	\item the measures $\Delta^{\vect s}_\Omega\cdot \nu_\Omega$ and $\Delta^{\vect s}_{\Omega'}\cdot \nu_{\Omega'}$ induce Radon measures on $F$ and $F'$, respectively, if and only if $\Re\vect s\succ \vect m$ and $\Re \vect s\succ \vect{m'}$, respectively, in which case
	\[
	\Lc(\Delta^{\vect s}_\Omega\cdot \nu_\Omega)=\Gamma_\Omega(\vect s) \Delta^{-\vect s}_{\Omega'} \qquad \text{and} \qquad \Lc(\Delta^{\vect s}_{\Omega'}\cdot \nu_{\Omega'})=\Gamma_{\Omega'}(\vect s) \Delta^{-\vect s}_{\Omega},
	\]
	for suitable $\Gamma_\Omega(\vect s),\Gamma_{\Omega'}(\vect
        s)\in \C$, where $\Lc$ denotes the Laplace transform
        (cf.~\cite[Proposition 2.19]{CalziPeloso}). 
\end{itemize}

Then, $\vect d=-(\vect 1_r+\vect m/2+\vect m'/2)$ (cf.~\cite[Definition 2.8, Lemma 2.18, and Proposition 2.19]{CalziPeloso}).

In addition, there are unique holomorphic families $(I^{\vect
  s}_{\Omega})_{\vect s\in \C^r}$ and $(I^{\vect{s}}_{\Omega'})_{\vect
  s\in \C^r}$ of tempered distributions on $F$ and $F'$, respectively,
such that $\Lc I^{\vect s}_{\Omega}= \Delta^{-\vect s}_{\Omega'}$ and
$\Lc I^{\vect s}_{\Omega'}=\Delta^{-\vect s}_\Omega$ on $\Omega'$ and
$\Omega$, respectively (cf.~\cite[Lemma 2.26 and Proposition
2.28]{CalziPeloso}). In particular, $I^{\vect
  s}_\Omega=\frac{1}{\Gamma_\Omega(\vect s)}\Delta_\Omega^{\vect
  s}\cdot \nu_\Omega$ for $\vect s \succ \vect m$, while $I^{\vect
  s}_{\Omega'}=\frac{1}{\Gamma_{\Omega'}(\vect s)}\Delta_\Omega^{\vect
  s}\cdot \nu_{\Omega'}$ for $\vect s \succ \vect{m'}$
(cf.~\cite[Proposition 2.28]{CalziPeloso}).  
We denote by $\N_\Omega$ and $\N_{\Omega'}$ the sets of $\vect
s\in\R^r$ such that $\Delta^{\vect s}_\Omega$ and $\Delta^{\vect
  s}_{\Omega'}$ are polynomials, respectively; equivalently,
$I^{-\vect s}_{\Omega'}$ and $I^{-\vect s}_{\Omega}$ are supported in
$\Set{0}$, respectively.

\subsection{Fourier Analysis on the \v Silov Boundary}

We endow $E\times F_\C$ with the product
\[
(\zeta,z)\cdot(\zeta',z')\coloneqq(\zeta+\zeta',z+z'+2 i\Phi(\zeta',\zeta))
\]
for every $(\zeta,z),(\zeta',z')\in E\times F_\C$, so that $E\times F_\C$ becomes a $2$-step nilpotent Lie group,  and the \v Silov boundary $b D=\rho^{-1}(0)$ a closed subgroup such that $b D \cdot D\subseteq D$.
We identify $b D$ with $\Nc \coloneqq E\times F$ by means of the mapping $(\zeta,x)\mapsto (\zeta,x+i\Phi(\zeta))$, so that $\Nc$ inherits the product
\[
(\zeta,x)(\zeta',x')=(\zeta+\zeta', x+x'+2 \Im \Phi(\zeta,\zeta')),
\]
for every $(\zeta,x),(\zeta',x')\in \Nc$. In addition, $\Nc$ inherits the structure of a CR manifold (cf.~\cite{Boggess}), with holomorphic tangent bundle generated by the left-invariant vector field $Z_v$ which induce the Wirtinger derivatives $\partial_{E,v}\coloneqq\frac 1 2 ( \partial_v-i \partial_{i v})$ at $(0,0)$, for $v\in E$. A distribution $u$ on $\Nc$ is then CR if $\overline{Z_v} u=0$ for every $v\in E$.
  We endow $\Nc$ with the dilations defined by $R\cdot (\zeta,x)\coloneqq (R^{1/2}\zeta,Rx)$ for every $R>0$ and for every $(\zeta,x)\in \Nc$. Even though these dilations may seem unnatural, they are well adapted to the Fourier transform we are about to describe, as well as the functional calculi we shall need in the sequel. 

We now recall some basic facts on Fourier analysis on $\Nc$. We limit ourselves to describing Plancherel formula for CR elements of $L^2(\Nc)$.
Define
\[
\Lambda_+\coloneqq \Set{\lambda\in F'\colon \forall \zeta\in E\setminus\Set{0}\:\: \langle \lambda, \Phi(\zeta)\rangle>0},
\]
so that $\Lambda_+$ is the interior of the polar of $\Phi(E)$ (cf.~\cite[Proposition 2.5]{PWS}), contains $\Omega'$, and coincides with $\Omega'$ if and only if $\vect b\prec \vect 0$ (cf.~\cite[Corollary 2.58]{CalziPeloso}). Then, for every $\lambda\in \Lambda_+$ the quotient $\Nc/\ker \lambda$ is isomorphism to a Heisenberg group (to $\R$, if $n=0$), so that the Stone--von Neumann theorem (cf., e.g.,~\cite[Theorem 1.50]{Folland}) implies that there is (up to unitary equivalence) a unique irreducible (continuous) unitary representation $\pi_\lambda$ of $\Nc$ in a Hilbert space $H_\lambda$ such that $\pi_\lambda(0,x)=\ee^{-i\langle \lambda,x\rangle}$ for every $x\in F$. One may choose $H_\lambda=\Hol(E)\cap L^2(\ee^{-2\langle \lambda, \Phi\rangle}\cdot \Hc^{2n})$ and 
\[
\pi_\lambda(\zeta,x)\psi(\omega)\coloneqq \ee^{\langle \lambda_\C, -i x+2 \Phi(\omega,\zeta)-\Phi(\zeta)\rangle}\psi(\omega-\zeta)
\]
for every $(\zeta,x)\in \Nc$, for every $\psi\in H_\lambda$, and for every $\omega\in E$.
If we define $P_{\lambda,0}$ as the orthogonal projector of $H_\lambda$ onto the space of constant functions, then (cf.~\cite[Proposition 1.15]{CalziPeloso})
\[
\tr(P_{\lambda,0}\pi_\lambda(\zeta,x))=\ee^{-\langle \lambda_\C,i x+\Phi(\zeta)\rangle }
\]
for every $(\zeta,x)\in \Nc$ and the mapping
\[
L^1_{\mathrm{CR}}(\Nc)\ni f \mapsto (\pi_\lambda(f))\in \prod_{\lambda\in \Lambda_+}\Lin(H_\lambda)
\]
induces an isometric isomorphism
\[
L^2_{\mathrm{CR}}(\Nc)\to c \int_{\Lambda_+}^\oplus \Lin^2(H_\lambda) P_{\lambda,0} \Delta^{-\vect b}_{\Omega'}(\lambda)\,\dd \lambda,
\]
where $L^1_{\mathrm{CR}}(\Nc)$ and $L^2_{\mathrm{CR}}(\Nc)$ denote the spaces of CR elements of $L^1(\Nc)$ and $L^2(\Nc)$, respectively, $\Lin(H_\lambda)$ and $\Lin^2(H_\lambda)$ denote the spaces of endomorphisms and Hilbert--Schmidt endomorphisms of $H_\lambda$, respectively, and $c>0$ is a suitable constant (cf.~\cite[Corollary 1.17 and Propositions 1.19 and 2.30]{CalziPeloso} or~\cite[Propositions 2.4 and 2.6]{PWS}).

\subsection{Lattices}
We endow $D$ with a complete Riemannian metric which is invariant under the group of affine automorphisms of $D$ (e.g., the Bergman metric, cf.~\cite[\S 2.5]{CalziPeloso}), and denote by $d$ the corresponding distance.  Observe that we may identify $\rho \colon D\to \Omega$ with the projection of $D$ onto its quotient by the action of $b D$ by isometries, so that $\Omega$ may be endowed with the quotient (complete) Riemannian metric, which is then $T_+$-invariant (but not necessarily invariant under all linear automorphisms of $\Omega$). We endow $\Omega'$ with the Riemannian metric 
induced by means of the diffeomorphism $\Omega\ni (t\cdot e_\Omega)\to (e_{\Omega'}\cdot t)\in \Omega'$, and denote by $d_\Omega$ and $d_{\Omega'}$ the induced distances on $\Omega$ and $\Omega'$, respectively. We denote by $B$, $B_\Omega$, and $B_{\Omega'}$ the balls on $D$, $\Omega$, and $\Omega'$, respectively.

By an $(\delta,R)$-lattice on $\Omega$, for $\delta>0$ and $R>1$, we mean a family $(h_k)_{k\in K}$ of elements of $\Omega$ such that the balls $B_\Omega(h_k,\delta)$ are pairwise disjoint while the balls $B_{\Omega}(h_k,R\delta)$ cover $\Omega$. For example, the $(\delta,2)$-lattices are the maximal families satisfying $d_\Omega(h_k,h_{k'})\Meg 2\delta$ for every $k\neq k'$.
We define $(\delta,R)$-lattices on $\Omega'$ analogously.

By a $(\delta,R)$-lattice on $D$ we mean a family $(\zeta_{j,k},z_{j,k})_{j\in J,k\in K}$ such that the balls $B((\zeta_{j,k},z_{j,k}),\delta)$ are pairwise disjoint, the balls $B((\zeta_{j,k},z_{j,k}),R\delta)$ cover $D$, and there is a $(\delta,R)$-lattice $(h_k)_{k\in K}$ on $\Omega$ such that $h_k=\rho(\zeta_{j,k},z_{j,k})$ for every $j\in J$ and for every $k\in K$.
Arguing as in~\cite[Lemma 2.55]{CalziPeloso} (where the Bergman metric on $D$ is considered), one may prove that $(\delta,4)$-lattices  exist for every $\delta>0$.

\subsection{Bergman Spaces}

\begin{deff}
	Take $\vect s\in \R^r$ and $p,q\in (0,\infty]$. We define   (cf.~\eqref{eq:4}) 
	\[
	L^{p,q}_{\vect s}(D)\coloneqq \Set{f\colon f \text{ is measurable, } \norm{h\mapsto \Delta^{\vect s}_\Omega(h) \norm{f_h}_{L^p(\Nc)}}_{L^q(\nu_\Omega)}<\infty},
	\]
	modulo the space of negligible functions, where $f_h\colon \Nc\ni(\zeta,x)\mapsto f(\zeta,x+i\Phi(\zeta)+i h)$ for every $h\in \Omega$ and for every $f\coloneqq D\to \C$. We define $L^{p,q}_{\vect s,0}(D)$ as the closure of $C_c(D)$ in $L^{p,q}_{\vect s}(D)$.
	
	We define 
	\[
	A^{p,q}_{\vect s}(D)\coloneqq \Hol(D)\cap L^{p,q}_{\vect s}(D) \qquad \text{and} \qquad A^{p,q}_{\vect s,0}(D)\coloneqq \Hol(D)\cap L^{p,q}_{\vect s,0}(D),
	\]
	endowed with the corresdponding topology.
\end{deff}

We observe explicitly that $A^{p,q}_{\vect s,0}(D)\neq \Set{0}$ (resp.\ $A^{p,q}_{\vect s}(D)\neq \Set{0}$) if and only if $\vect s \succ \frac{1}{2 q}\vect m$ (resp.\ $\vect s \Meg 0$ if $q=\infty$), and that $A^{p,\infty}_{\vect 0}(D)$ is the Hardy space $H^p(D)$ (cf.~\cite[Proposition 3.5]{CalziPeloso}).

By~\cite[Proposition 3.11]{CalziPeloso}, $A^{2,2}_{\vect s}(D)$  is a reproducing kernel Hilbert space for every $\vect s \succ \frac{1}{4}\vect m$, with reproducing kernel
\[
((\zeta,z),(\zeta',z'))\mapsto c_{\vect s}' B^{\vect b+\vect d-2\vect s}_{(\zeta',z')}(\zeta,z),
\]
where $c'_{\vect s}$ is a suitable (non-zero) constant and
\[
 B^{\vect{s'}}_{(\zeta',z')}(\zeta,z)\coloneqq \Delta^{\vect{s'}}_{\Omega}\left(\frac{z-\overline{z'}}{2 i}-\Phi(\zeta,\zeta')  \right)
\]
for every $(\zeta,z),(\zeta',z')\in D$ and for every $\vect{s'}\in \R^r$. We then denote by $P_{\vect b+\vect d-2 \vect s}$ the orthogonal projector of $L^{2,2}_{\vect s}(D)$ onto $A^{2,2}_{\vect s}(D)$, so that
\[
P_{\vect{s'}}f(\zeta,z)=c_{\vect {s'}}\int_D f \overline{B^{\vect {s'}}_{(\zeta,z)}}(\Delta^{-\vect {s'}}_\Omega\circ \rho)\,\dd \nu_D
\]
for every $f \in C_c(D)$, for every $(\zeta,z)\in D$, and for every $\vect{s'}\prec \vect b+\vect d-\frac 1 2 \vect m$, where $c_{\vect{s'}}\coloneqq c'_{(\vect b+\vect d-\vect{s'})/2}$.

  Notice that, even though this choice may seem unnatural, it has the advantage of simplifying the notation. 

\subsection{Besov Spaces}

Denote by  $\Sc(\Nc)$   the Schwartz space on $\Nc$ and define, for every
compact subset $K$ of $\Omega'$,
\[
\Sc_\Omega(\Nc,K)\coloneqq \Set{\varphi\in \Sc(\Nc)\colon \varphi\text{ is CR, } \forall \lambda\in \Lambda_+\:\: \pi_\lambda(\varphi)=\chi_{K}(\lambda)P_{\lambda,0} \pi_\lambda(\varphi)P_{\lambda,0}}
\]
and
\[
\Sc_{\Omega,L}(\Nc,K)\coloneqq \Sc(\Nc)*\Sc_\Omega(\Nc,K),
\]
endowed with the topology induced by $\Sc(\Nc)$, and
\[
\Sc_\Omega(\Nc) =\varinjlim_K \Sc_{\Omega}(\Nc,K) \qquad \text{and} \qquad\Sc_{\Omega,L}(\Nc) =\varinjlim_K \Sc_{\Omega,L}(\Nc,K),
\]
endowed with their locally convex topologies.

We denote by $\Sc'_{\Omega,L}(\Nc)$ the dual of the conjugate of $\Sc_{\Omega,L}(\Nc)$.
Then, the mapping
\[
\Fc_\Nc\colon \Sc_\Omega(\Nc)\ni \varphi\mapsto [\lambda\mapsto \tr(\pi_\lambda(\varphi))]\in C^\infty_c(\Omega')
\]
is an isomorphism, and there is $c>0$ such that
\begin{equation}\label{eq:2}
(\Fc_\Nc^{-1}\psi)(\zeta,x)=c\int_{\Omega'} \psi(\lambda) \ee^{\langle \lambda_\C, i x-\Phi(\zeta)\rangle} \Delta^{-\vect b}_{\Omega'}(\lambda)\,\dd \lambda
\end{equation}
for every $\psi\in C^\infty_c(\Omega')$ and for every $(\zeta,x)\in \Nc$ (cf.~\cite[Proposition 4.2]{CalziPeloso}, but also~\cite[\S\ 5]{PWS}).

\begin{deff}\label{def:1}
	Take $p,q\in (0,\infty]$ and $\vect s\in \R^r$. Given a $(\delta,R)$-lattice $(\lambda_k)_{k\in K}$ on $\Omega'$, and a bounded family $(\varphi_k)_{k\in K}$ of positive elements of $C^\infty_c(\Omega')$ such that
	\[
	\sum_{k\in K} \varphi_k(\,\cdot \,t_k^{-1})\Meg 1
	\]
	on $\Omega'$, where $t_k\in T_+$ and $\lambda_k=e_{\Omega'}\cdot t_k$ for every $k\in K$, we define $B^{\vect s}_{p,q}(\Nc,\Omega)$ (resp.\ $\mathring B^{\vect s}_{p,q}(\Nc,\Omega)$) as the space of $u\in \Sc'_{\Omega,L}(\Nc,\Omega)$ such that
	\[
	(\Delta^{\vect s}_{\Omega'}(\lambda_k)u*\psi_k)\in \ell^q(K;L^p(\Nc)) \qquad \text{(resp.\ $(\Delta^{\vect s}_{\Omega'}(\lambda_k)u*\psi_k)\in \ell^q_0(K;L^p_0(\Nc))$)},
	\]
	where $\psi_k\coloneqq \Fc_\Nc^{-1}(\varphi_k(\,\cdot\, t_k^{-1}))$ for every $k\in K$.\footnote{As before, $L^p_0(\Nc)$ is the closure of $C_c(\Nc)$ in $L^p(\Nc)$, while $\ell^q_0(K;L^p_0(\Nc))$ is the closure of $L^p_0(\Nc)^{(K)}$ in $\ell^q(K;L^p_0(\Nc))$. In addition, $u*\psi_k\in \Sc'(\Nc)$ is defined so that $\langle u*\psi_k\vert \tau \rangle= \langle u \vert \tau*\psi_k^*\rangle$ for every $\tau\in \Sc(\Nc)$. The definition is well posed since $\psi_k^*\in \Sc_\Omega(\Nc)$, so that $\tau*\psi_k^*\in \Sc_{\Omega,L}(\Nc)$. It is then readily seen that $u*\psi_k$ is actually a function of class $C^\infty$.  }
\end{deff}

Notice that $\mathring B^{p,q}_{\vect s}(\Nc,\Omega)$ is the closure of (the canonical image of) $\Sc_{\Omega,L}(\Nc)$ in $B^{p,q}_{\vect s}(\Nc,\Omega)$, and that the definition of both spaces does not depend on the choice of $(\lambda_k)$ and $(\varphi_k)$ by~\cite[Lemma 4.14 and Theorem 4.23]{CalziPeloso}. In addition, the closure of $\Sc_{\Omega,L}(\Nc)$ in $\Sc(\Nc)$ is  (cf.~\cite[Theorem 5.13]{PWS})
\[
\widetilde \Sc_\Omega(\Nc)\coloneqq \Set{\varphi\in \Sc(\Nc)\colon \varphi\text{ is CR, } \forall \lambda\in \Lambda_+\setminus \Omega'\:\: \pi_\lambda(\varphi)=0}
\]
and $B^{\vect s}_{p,q}(\Nc,\Omega)$ embeds canonically into the dual of the conjugate of $\widetilde \Sc_{\Omega}(\Nc)$ (which is a quotient of $\Sc'(\Nc)$) for every $\vect s\in \R^r$ and for every $p,q\in(0,\infty]$ (cf.~\cite[Proposition 7.12]{Besov}).

One may then define a canonical sesquiliear form
\[
B^{\vect s}_{p,q}(\Nc,\Omega)\times B^{-\vect s-(1/p-1)_+(\vect b+\vect d)}_{p',q'}(\Nc,\Omega) \to \C
\]
so that
\[
\langle u \vert u'\rangle\coloneqq \sum_k \langle u*\psi_k\vert u'*\psi_k\rangle,
\]
where $(\psi_k)$ is as in Definition~\ref{def:1} and $\sum_k (\Fc_\Nc \psi_k)^2=1$ on $\Omega'$. This definition does \emph{not} depend on the choice of $(\psi_k)$ (cf.~\cite[Proposition 4.20]{CalziPeloso}), and identifies $ B^{-\vect s-(1/p-1)_+(\vect b+\vect d)}_{p',q'}(\Nc,\Omega)$ with the dual of $\mathring B^{\vect s}_{p,q}(\Nc,\Omega)$ (cf.~\cite[Theorem 4.23]{CalziPeloso}).
By analogy, we define the weak topology $\sigma_{p,q}^{\vect s}$ as the weak topology $\sigma(B^{\vect s}_{p,q}(\Nc,\Omega),\mathring B^{-\vect s-(1/p-1)_+(\vect b+\vect d)}_{p',q'}(\Nc,\Omega)  ) $ with respect to this sesquilinear form.

\subsection{An Extension Operator}

Observe that $H^2(D)=A^{2,\infty}_{\vect 0}(D)$ is a reproducing kernel Hilbert space, with reproducing (Cauchy--Szeg\H o) kernel
\[
((\zeta,z),(\zeta',z'))\mapsto c_0 B^{\vect b+\vect d}_{(\zeta',z')}(\zeta,z).
\]
We then define a continuous linear mapping (cf.~\cite[Lemma 2.51, Theorem 5.2]{CalziPeloso})
\[
\Ec\colon B^{-\vect s}_{p,q}(\Nc,\Omega) \ni u \mapsto [(\zeta,z)\mapsto c_0 \langle u\vert (B^{\vect b+\vect d}_{(\zeta,z)})_0\rangle]\in A^{\infty,\infty}_{\vect s-(\vect b+\vect d)/p}(D)
\]
for $\vect s\succ \frac 1 p (\vect b+\vect d)+\frac{1}{2 q'}\vect{m'}$, so that $(B^{\vect b+\vect d}_{(\zeta,z)})_0\in \mathring B^{-\vect s-(1/p-1)_+(\vect b+\vect d)}_{p',q'}(\Nc,\Omega)$ for every $(\zeta,z)\in D$. In addition, $(\Ec u)_h$ converges to $u$ in $B^{-\vect s}_{p,q}(\Nc,\Omega)$ if $u\in \mathring B^{-\vect s}_{p,q}(\Nc,\Omega)$, in the weak topology $\sigma^{\vect s}_{p,q}$ if $u\in B^{-\vect s}_{p,q}(\Nc,\Omega)$. In particular, $\Ec$ is one-to-one.

We define 
\[
\widetilde A^{p,q}_{\vect s}(D)\coloneqq \Ec(B^{\vect s}_{p,q}(\Nc,\Omega)) \qquad \text{and} \qquad \widetilde A^{p,q}_{\vect s,0}(D)\coloneqq \Ec(\mathring B^{\vect s}_{p,q}(\Nc,\Omega)),
\]
endowed with the corresponding topology.   We denote by $\widetilde\sigma_{\vect s}^{p,q}$ the topology on $\widetilde A^{p,q}_{\vect s}(D)$ and $\widetilde A^{p,q}_{\vect s,0}(D)$ induced by the weak topology $\sigma^{-\vect s}_{p,q}$ on $B_{p,q}^{-\vect s}(\Nc,\Omega)$.  

The following result summerizes~\cite[Proposition 5.4 and Corollary 5.11]{CalziPeloso}.

\begin{prop}\label{prop:1}
	Take $p,q\in (0,\infty]$ and $\vect s\succ \frac 1 p (\vect b+\vect d)+\frac{1}{2 q'}\vect{m'}$ such that $\vect s\succ \frac{1}{2 q}\vect m$ (resp.\ $\vect s \Meg \vect 0$ if $q=\infty$). Then, there are continuous inclusions
	\[
	\Ec(\Sc_{\Omega,L}(\Nc)) \subseteq A^{p,q}_{\vect s,0}(D)\subseteq \widetilde  A^{p,q}_{\vect s,0}(D) \qquad \text{(resp.\ $\Ec(\Sc_{\Omega,L}(\Nc)) \subseteq A^{p,q}_{\vect s}(D)\subseteq \widetilde  A^{p,q}_{\vect s}(D) $)}.
	\]
	If, in addition,
	\[
	\vect s \succ \frac{1}{2 q}\vect m+\left( \frac{1}{2 \min(p,p')}-\frac{1}{2 q} \right)_+\vect{m'},
	\]
	then $A^{p,q}_{\vect s,0}(D)=\widetilde  A^{p,q}_{\vect s,0}(D) $  and $A^{p,q}_{\vect s}(D)= \widetilde  A^{p,q}_{\vect s}(D) $.
\end{prop}

 We are able to improve the first half of~\cite[Proposition
5.18]{CalziPeloso} following~\cite[Proposition
4.34]{BekolleBonamiGarrigosRicci} and~\cite{BekolleGonessaNana}.  For
clarity of presentation, we
postpone the proof until Section~\ref{Proof:sec}. 

\begin{prop}\label{prop:13}
	Take $p,q\in(0,\infty]$ and $\vect s\succ \frac{1}{p}(\vect b+\vect d)+\frac{1}{2 q'}\vect m'$. If $\widetilde A^{p,q}_{\vect s,0}(D)=A^{p,q}_{\vect s,0}(D)$, then $\vect s \succ \left(\frac{1}{2 p}-\frac{1}{2 q}\right)_+\vect m'$. 
	If, in addition, $n=0$, then  $\vect s \succ \left(\frac{1}{2 \min(2,p)}-\frac{1}{2 q}\right)_+\vect m'$.
\end{prop}

When $n=0$ and $r=2$, so that $\Omega$ is either (isomorphic to) $(\R_+^*)^2$ or a Lorentz cone, then combining~\cite[Theorems 6.6 and 6.8]{BekolleGonessaNana} (the latter being a consequence of~\cite[Theorem 1.2]{BourgainDemeter}) with~\cite[Theorem 5.10]{CalziPeloso}, we see that $A^{p,q}_{\vect s}(D)=\widetilde  A^{p,q}_{\vect s}(D)$  and $A^{p,q}_{\vect s,0}(D)=\widetilde  A^{p,q}_{\vect s,0}(D)$, provided that 
\[
\vect s \succ \frac{1}{2q}\vect m+\frac 1 2\left(\frac{1}{\min(2,p)} -\frac{1}{ q}   \right)_+\vect m', \frac{\vect d}{p}+\frac{1}{2 q'}\vect m'.
\]
In addition, since in this case $\alpha \vect m+\beta \vect m'=\sup(\alpha \vect m,\beta \vect m')$ for every $\alpha,\beta \Meg 0$ (cf.~\cite[Definition 2.8]{CalziPeloso}), the preceding sufficient conditions are also necessary, thanks to Proposition~\ref{prop:13}.

We now recall another result concerning the dual of $\widetilde A^{p,q}_{\vect s,0}(D)$, namely~\cite[Proposition 5.12]{CalziPeloso}.\footnote{We observe explicitly that there is a mistake in the statement of the cited result, which we correct here. The proof is unchanged, besides the corresponding correction.} 

\begin{prop}\label{prop:10}
	Take $p,q\in (0,\infty]$, $\vect s\succ \frac 1p (\vect b+\vect d)+\frac{1}{q'}\vect{m'}$ and $\vect{s'}\succ \frac{1}{p'}(\vect b+\vect d)+\frac{1}{2 q''}\vect{m'}$. Define $\vect{s''}\coloneqq \vect s+\vect{s'}-\left(\frac 1 p-1\right)_+(\vect b+\vect d)$, and assume that $\vect{s''}\succ \frac 1 2 \vect m$. Then, the sesquilinear form
	\[
	\Ec(\Sc_{\Omega,L}(\Nc))\times \Ec(\Sc_{\Omega,L}(\Nc))\ni (f,g)\mapsto  \int_D f\overline g (\Delta^{\vect{s''-\vect b-\vect d}}_\Omega\circ \rho)\,\dd \nu_D
	\]
	is well  defined and extends to a unique continuous sesquilinear form on $\widetilde A^{p,q}_{\vect s,0}(D)\times \widetilde A^{p',q'}_{\vect{s'}}(D)$ which is continuous on the second factor with respect to  $\widetilde\sigma^{p',q'}_{\vect s'}$ through $\Ec$. The extended sesquilinear form induces an antilinear isomorphism of $\widetilde A^{p',q'}_{\vect{s'}}(D)$ onto $\widetilde A^{p,q}_{\vect s,0}(D)'$, and there is $c\neq 0$ such that
	\[
	\langle \Ec u\vert \Ec (u'*I^{-\vect{s''}}_\Omega)\rangle=c\langle u\vert u'\rangle
	\]
	for every $u,u'\in \Sc_{\Omega,L}(\Nc)$.
\end{prop}

\section{Continuity of $P_{\vect{s'},+}$}\label{sec:3}

In this section we extend to homogeneous Siegel domains of type II an interesting characterization of the boundedness of the operator 
\[
P_{\vect{s'},+}\colon f \mapsto \int_D f(\zeta,z) \abs{B^{\vect{s'}}_{(\zeta,z)}} \Delta^{-\vect{s'}}_\Omega(\rho(\zeta,z))\,\dd \nu_D(\zeta,z)
\]
provided in~\cite[Theorem 4.10]{Bekolleetal} for irreducible symmetric tube domains ($\vect s\in \R\vect d$ and $\vect{s'}=\vect d-q\vect s$).

\begin{teo}\label{prop:7} 
	Take $\vect s, \vect{s'}\in \R^r$ and $p,q\in [1,\infty]$.
	Then, $P_{\vect{s'},+}$ induces an endomorphism of $L^{p,q}_{\vect s,0}(D)$ (resp.\ $L^{p,q}_{\vect s}(D)$) if and only if 
	\[
	T\colon C_c(\Omega)\ni f \mapsto  \Delta_\Omega^{\vect s} \int_\Omega f(h) \Delta^{\vect{s'}-(\vect b+\vect d)}_\Omega(\,\cdot\,+h)\Delta^{\vect b+\vect d-\vect s-\vect{s'}}_\Omega(h)\,\dd \nu_\Omega(h)
	\]
	induces an endomorphism of $L^q_0(\nu_\Omega)$ (resp.\ $L^q(\nu_\Omega)$).
\end{teo}

  For the continuity of operators such as $T$, see~\cite[Lemma 3.35]{CalziPeloso}, but also~\cite{Garrigos-Nana} for a much more general treatment.  

Before we pass to the proof, we need a lemma, which extends~\cite[Lemma 4.11]{Bekolleetal}. Cf.~also~\cite{Sehba}.

\begin{lem}\label{lem:1}
	Take $\vect{s}\in \C^r$, and let $U$ be a compact neighbourhood of $(0,0)$ in $\Nc$. Then, there are two constants $C,c>0$ such that
	\[
	\int_{U^2} \abs*{\left(B^{\vect{s'}}_{(\zeta,z)}\right)_{h'}(\zeta',x')}\,\dd (\zeta',x')\Meg C \Delta_\Omega^{\vect{s'}-(\vect b+\vect d)}(h+h')
	\]
	for every $(\zeta,z)\in D$ and for every $h'\in \Omega$ such that $(\zeta,\Re z)\in U$ and such that $\abs{\rho(\zeta,z)},\abs{h'}\meg c$.
\end{lem}

\begin{proof}
	Take an open neighbourhood $U'$ of $0$ in $E$ and an open neighbourhood $V'$ of $0$ in $F$ such that $U'\times V'\subseteq U$. Since $\Phi$ is proper, we may assume that $U'=\Phi^{-1}(\Phi(U'))$. Then,~\cite[Proposition 2.30]{CalziPeloso} implies that there is a constant $C_1>0$ such that
	\[
	\begin{split}
	\int_{U^2} \abs*{\left(B^{\vect{s'}}_{(\zeta,x+i\Phi(\zeta)+i h)}\right)_{h'}(\zeta',x')}\,\dd (\zeta',x')& \Meg \int_{(\zeta,x)U} \abs*{\left(B^{\vect{s'}}_{(0,i h)}\right)_{h'}((\zeta,x)^{-1}(\zeta',x'))}\,\dd (\zeta',x')\\ 
		&=\int_{U} \abs*{\left(B^{\vect{s'}}_{(0,i h)}\right)_{h'}(\zeta',x')}\,\dd (\zeta',x')\\
		&= 2^{-\vect{s'}}\int_{U} \abs*{\Delta_\Omega^{\vect{s'}}(h+h'+\Phi(\zeta')-i x' )}\,\dd (\zeta',x')\\
		&\Meg C_1\int_{\Phi(U')\times V'} \abs*{\Delta_\Omega^{\vect{s'}}(h+h'+h''-i x' )}\,\dd (I^{-\vect b}_\Omega\otimes \Hc^m) (h'',x')
	\end{split}
	\]
	for every $(\zeta,x)\in U$ and for every $h,h'\in \Omega$.
	In addition, take $R>0$ and observe that by~\cite[Corollary 2.51]{CalziPeloso} there is a constant $C_2\Meg 1$ such that
	\[
	\frac{1}{C_2} \meg \frac{\abs{h}}{\abs{h'}}\meg C_2
	\]
	for every $h,h'\in \Omega$ with $d_\Omega(h,h')\meg R$. If we take $c>0$ so that $\overline B_F(0,3 c C_2)\subseteq V'$, then
	\[
	B_\Omega(h+h'+h'',R)\subseteq V'
	\]
	for every $h,h',h''\in \Omega$ with $\abs{h}, \abs{h'},\abs{h''}\meg c$. Since we may assume that $\Phi(U')\subseteq B_F(0,c)$, this implies that
	\[
	\int_{U^2} \abs*{\left(B^{\vect{s'}}_{(\zeta,x+i\Phi(\zeta)+i h)}\right)_{h'}(\zeta',x')}\,\dd (\zeta',x')\Meg C_1\int_{\Phi(U')}\int_{B_\Omega(h+h'+h'',R)} \abs*{\Delta_\Omega^{\vect{s'}}(h+h'+h''-i x' )}\,\dd x'\,\dd I^{-\vect b}_\Omega (h'')
	\]
	for every $(\zeta,x)\in U$ and for every  $h,h'\in \Omega\cap \overline B_F(0,c)$.
	Now, by homogeneity,
	\[
	\int_{B_\Omega(h''',R)} \abs*{\Delta_\Omega^{\vect{s'}}(h'''-i x' )}\,\dd x'= C_3 \Delta_\Omega^{\Re\vect{s'}-\vect d}(h''') 
	\]
	for every  $h'''\in \Omega$, where $C_3\coloneqq \int_{B_\Omega(e_\Omega,R)} \abs*{\Delta_\Omega^{\vect{s'}}(e_\Omega-i x' )}\,\dd x'$. In addition,
	\[
	\int_{\Phi(U')}\Delta_\Omega^{\Re\vect{s'}-\vect d}(h+h'+h'') \,\dd I^{-\vect b}_\Omega (h'')=J(h+h')\Delta_\Omega^{\Re \vect{s'}-\vect b-\vect d}(h+h'),
	\]
	where
	\[
	J(t\cdot e_\Omega)\coloneqq \int_{t^{-1}\cdot \Phi(U')} \Delta_\Omega^{\Re\vect{s'}-\vect d}(e_\Omega+h'') \,\dd I^{-\vect b}_\Omega (h'')
	\]
	for every $t\in T_+$. Now, define $Q\coloneqq \Set{t\in T_+\colon t\cdot e_\Omega\in (e_\Omega-\overline\Omega)}$, and observe that 
	\[
	t\cdot [\Omega\cap(e_\Omega-\overline\Omega)]\subseteq \Omega \cap (e_\Omega-\overline\Omega-\overline\Omega)=\Omega \cap (e_\Omega-\overline\Omega)
	\]
	for every $t\in Q$, so that $Q$ is bounded in $\Lc(F)$ and $QQ\subseteq Q$.
	Now, by~\cite[\S\ 2.6]{CalziPeloso}, there is $e'\in \overline \Omega$ such that  $I^{-\vect b}_\Omega$ is concentrated on $T_+\cdot e'$ and has support $\overline{T_+\cdot e'}$ so that, by homogeneity, $\Phi(E)\supseteq T_+\cdot e'$. Observe that $\Phi(U')$ is a neighbourhood of $0$ in $\Phi(E)$ since $\Phi$ is proper, thanks to ~\cite[Proposition 10 of Chapter I, \S\ 5, No.\ 4]{BourbakiGT1}. 
	Since $Q\cdot e'$ is relatively compact, this implies that there is $R'>0$ such that $R' (Q\cdot e')\subseteq \Phi(U')$, so that
	\[
	t^{-1}\cdot \Phi(U')\supseteq  R' [t^{-1}\cdot (Q\cdot e')] \supseteq R' [t^{-1}\cdot(t Q\cdot e')]=R' (Q\cdot e')
	\]
	for every $t\in Q$, so that
	\[
	J(h)\Meg  \int_{R' Q\cdot e'} \Delta_\Omega^{\Re\vect{s'}-\vect d}(e_\Omega+h'') \,\dd I^{-\vect b}_\Omega (h'')>0
	\]
	for every $h\in Q\cdot e_\Omega=\Omega\cap (e_\Omega-\overline\Omega)$. The conclusion follows, since we may take $c$ so small that $\overline B_F(0,2c)\cap \Omega \subseteq \Omega\cap (e_\Omega-\overline\Omega)$.
\end{proof}

\begin{proof}[Proof of Theorem~\ref{prop:7}.]
	Assume first that $T$ induces an endomorphism of $L^q_0(\nu_\Omega)$ (resp.\ $L^q(\nu_\Omega)$). Observe that, since $T$ is an integral operator with a \emph{positive} kernel, the endomorphism of $L^q_0(\nu_\Omega)$ (resp.\ $L^q(\nu_\Omega)$) induced by $T$ is still an integral operator with the same kernel.	   Observe that~\cite[Theorem 1.4]{Garrigos-Nana} implies that $\vect s'\prec \vect b+\vect d -\frac 1 2 \vect m'$.  
	Then, for every $f\in L^{p,q}_{\vect s,0}(D)$ (resp.\ $f\in L^{p,q}_{\vect s}(D)$)
	\[
	\begin{split}
		\abs{P_{\vect{s'},+} f(\zeta,z)}&\meg \int_\Omega \int_\Nc \abs{f_{h'}(\zeta',x')} \abs*{\left(B^{\vect{s'}}_{(0, i h')}\right)_{h}((\zeta',x')^{-1}(\zeta,x))}\,\dd (\zeta',x') \Delta_\Omega^{\vect b+\vect d-\vect{s'}}(h') \,\dd \nu_\Omega(h')\\
			& = \int_\Omega \left( \abs{f_{h'}}* \abs*{\left(B^{\vect{s'}}_{(0, i h')}\right)_{h}}\right) (\zeta,x) \Delta_\Omega^{\vect b+\vect d-\vect{s'}}(h') \,\dd \nu_\Omega(h')
	\end{split}
	\]
	for every $(\zeta,z)\in D$, where $h\coloneqq \rho(\zeta,z)$ and $x\coloneqq \Re z$,
	so that, by Young's inequality
	\[
	\begin{split}
		\Delta_\Omega^{\vect s}(h)\norm{(P_{\vect{s'},+} f)_h}_{L^p(\Nc)}&\meg C_1\Delta_\Omega^{\vect s}(h)\int_\Omega \norm{f_{h'}}_{L^p(\Nc)}\Delta_\Omega^{\vect{s'-(\vect b+\vect d)}}(h+h') \Delta_\Omega^{\vect b+\vect d-\vect{s'}}(h') \,\dd \nu_\Omega(h')\\
		&=C_1 T(\Delta_\Omega^{\vect s} \norm{f_{\,\cdot\,}}_{L^p(\Nc)})(h)
	\end{split}
	\]
	for every $h\in \Omega$, where $C_1>0$ is a suitable constant (cf.~\cite[Lemma 2.39]{CalziPeloso}). Therefore, $P_{\vect{s'},+}$ maps $L^{p,q}_{\vect s,0}(D)$ (resp.\ $L^{p,q}_{\vect s}(D)$) in $L^{p,q}_{\vect s}(D)$. To conclude the proof of this implication, take $f\in C_c(D)$, and observe that $(P_{\vect{s'},+} f)_h\in L^p_0(\Nc)$ for every $h\in \Omega$, thanks to the preceding remarks. In addition, the mapping $h \mapsto \Delta_\Omega^{\vect s}(h) \norm{f_h}_{L^p(\Nc)}\in \R$ clearly belongs to $C_c(\Omega)$, so that $P_{\vect{s'},+} f\in L^{p,q}_{\vect s,0}(D)$.
	
	Conversely, assume that $P_{\vect{s'},+}$ induces an endomorphism of $L^{p,q}_{\vect s,0}(D)$ (resp.\ $L^{p,q}_{\vect s}(D)$). As for $T$, the endomorphism of $L^{p,q}_{\vect s,0}(D)$ (resp.\ $L^{p,q}_{\vect s}(D)$) induced by  $P_{\vect{s'}}$ is an integral operator with the same kernel. 
	Fix a compact neighoburhood $U$ of $(0,0)$ in $\Nc$ and take $C,c>0$ as in Lemma~\ref{lem:1}. Fix  two positive function $\tau_1,\tau_1'\in C_c(F)$ and a positive function $\tau_2\in C_c(\Nc)$ so that $\tau_1=1$ on $B_F(0,c)$, $\tau_1'$ is non-zero and supported in $B_F(0,c)$, and $\tau_2=1$ on $U^2$.
	
	Take a positive $f\in L^q_0(\nu_\Omega)$ (resp.\ $f\in L^q(\nu_\Omega)$), and define $g\in L^{p,q}_{\vect s,0}(D)$ (resp.\ $g\in L^{p,q}_{\vect s}(D)$) so that
	\[
	g_h=\tau_1(h) \Delta_\Omega^{-\vect s}(h) f(h)\tau_2
	\]
	for every $h\in \Omega$. 
	Then, 
	\[
	\begin{split}
	(P_{\vect{s'},+} g)_h(\zeta,x)&=  \int_\Omega \int_\Nc g_{h'}(\zeta',x') \abs*{\left(B^{\vect{s'}}_{(\zeta,x+i\Phi(\zeta)+i h)}\right)_{h}(\zeta',x')}\,\dd (\zeta',x') \Delta_\Omega^{\vect b+\vect d-\vect{s'}}(h') \,\dd \nu_\Omega(h')\\
		&\Meg\int_{\Omega\cap B_F(0,c)}  \int_{U^2}  \abs*{\left(B^{\vect{s'}}_{(\zeta,x+i\Phi(\zeta)+i h)}\right)_{h}(\zeta',x')}\,\dd (\zeta',x') f(h')  \Delta_\Omega^{\vect b+\vect d-\vect s-\vect{s'}}(h') \,\dd \nu_\Omega(h')\\
		&\Meg C \int_{\Omega\cap B_F(0,c)}   f(h') \Delta_\Omega^{\vect {s'}-(\vect b+\vect d)}(h+h')  \Delta_\Omega^{\vect b+\vect d-\vect s-\vect{s'}}(h') \,\dd \nu_\Omega(h')\\
		&=C\Delta_\Omega^{-\vect s}(h)T(\chi_{\Omega \cap B_F(0,c) } f)
	\end{split}
	\]
	for every $(\zeta,x)\in U$ and for every $h\in \Omega \cap B_F(0,c)$. Therefore,
	\[
	\Delta_\Omega^{\vect s}(h)\norm{(P_{\vect{s'},+} g)_h}_{L^p(\Nc)}\Meg C \Hc^{2n+m}(U) T(\chi_{B_F(0,c)} f)(h)
	\]
	for every $h\in  \Omega \cap B_F(0,c)$. Therefore, there is a constant $C_2>0$ such that
	\[
	\norm{T(\tau_1'f)}_{L^q(\nu_\Omega)}\meg C_2 \norm{\tau_1 f}_{L^q(\nu_\Omega)}
	\]
	for every (positive) $f\in L^q(\nu_\Omega)$. By homogeneity, this implies that
	\[
	\norm{T(\tau_1'(R\,\cdot\,) f)}_{L^q(\nu_\Omega)}\meg C_2 \norm{\tau_1(R\,\cdot\,)f}_{L^q(\nu_\Omega)}
	\]
	for every $f\in L^q(\nu_\Omega)$ and for every $R>0$. In addition, the preceding remarks also imply that $T(\tau_1'(R\,\cdot\,)f)\in L^q_0(\nu_\Omega) $ if $f\in L^q_0(\nu_\Omega)$.  
	We have thus proved that $T$ induces an endomorphism of $L^q_0(\nu_\Omega)$ (resp.\ a continuous linear mapping $L^q_0(\nu_\Omega)\to L^q(\nu_\Omega)$). To conclude, it suffice to observe that, since $T$ is an integral operator with a positive kernel, if $T$ induces a continuous linear mapping $L^q_0(\nu_\Omega)\to L^q(\nu_\Omega)$, then it also induces an endomorphism of $L^q(\nu_\Omega)$ with the same expression and the same norm.
\end{proof}

\begin{cor}
	Take $\vect s,\vect{s'}\in \R^r$ and $p,q\in [1,\infty]$ such that $P_{\vect{s'},+}$ induces an endomorphism of $L^{p,q}_{\vect s,0}(D)$ (resp.\ $L^{p,q}_{\vect s}(D)$). Then, the following hold:
	\begin{itemize}
		\item $\vect s\succ  \frac{1}{2 q} \vect m, \frac{1}{2 q'}\vect{m'}$;
			
		\item $\vect b+\vect d-(\vect s+\vect{s'})\succ \frac{1}{2 q'} \vect m, \frac{1}{2 q}\vect{m'}$.
	\end{itemize}
\end{cor}

Observe that, when $r\meg 2$, the above necessary conditions are also sufficient, since in this case $\alpha \vect m+\beta \vect m'=\sup(\alpha \vect m,\beta \vect m')$ for every $\alpha,\beta \Meg 0$ (cf.~\cite[Definition 2.8 and Corollary 5.23]{CalziPeloso}, or~\cite[Theorem 1.4]{Garrigos-Nana}). The same holds also if $D$ is irreducible and symmetric, and $\vect s$ and $\vect s'$ are parallel to $\vect d$. In this case, also $\vect b$ and $\vect m+\vect m'$ are parallel to $\vect d$, so that 
\[
\vect s \succ \frac{1}{2 q}\vect m+\frac{1}{2 q'}\vect m'\iff \vect s \succ\frac{1}{2 q}\vect m,\frac{1}{2 q'}\vect m'
\]
and
\[
\vect b+\vect d-(\vect s+\vect s')\succ \frac{1}{2 q'}\vect m+\frac{1}{2 q}\vect m'\iff \vect b+\vect d-(\vect s+\vect s')\succ \frac{1}{2 q'}\vect m,\frac{1}{2 q}\vect m'
\]
since there are $j\neq k$ such that $m_j=m'_k=0$ (cf.~\cite[Definition 2.8 and Corollary 5.23]{CalziPeloso}, or~\cite[Theorem 1.4]{Garrigos-Nana}).

\begin{proof}
	This is a consequence of Theorem~\ref{prop:7} and either~\cite[Proposition 5.20]{CalziPeloso} or~\cite[Theorem 1.4]{Garrigos-Nana}.
\end{proof}

\section{Equivalences}\label{sec:4}

In this section we prove the equivalence of various notions of atomic decompositions, the continuity of Bergman projectors, and the determination of boundary values.

\begin{prop}\label{prop:8}
	Take $p,q\in (0,\infty]$ and $\vect s\succ \frac{\vect b+\vect d}{p}+\frac{1}{2 q'}\vect{m'}$. Then, $\widetilde A^{p,q}_{\vect s}(D)= A^{p,q}_{\vect s}(D)$ if and only if $\widetilde A^{p,q}_{\vect s,0}(D)\subseteq A^{p,q}_{\vect s}(D)$.
	In particular,	the following conditions are equivalent:
	\begin{enumerate}
		\item[\textnormal{(1)}] $\widetilde A^{p,q}_{\vect s,0}(D)= A^{p,q}_{\vect s,0}(D)$;

		\item[\textnormal{(2)}] $\vect s\succ \vect 0$ and $\widetilde A^{p,q}_{\vect s}(D)= A^{p,q}_{\vect s}(D)$.
	\end{enumerate}
\end{prop}

We first need a technical lemma.

\begin{lem}\label{lem:3}
	There is a sequence $(\Psi_j)$ of linear mappings $\Sc_{\Omega,L}'(\Nc)\to \Sc_{\Omega,L}(\Nc)$ such that, for every $p,q\in(0,\infty]$ and for every $\vect s\in \R^r$, $(\Psi_j)$ induces an equicontinuous sequence of endomorphisms of $B^{\vect s}_{p,q}(\Nc,\Omega)$, and $\Psi_j u$ converges to $u$ in $\sigma^{\vect s}_{p,q}$ for every $u\in B^{\vect s}_{p,q}(\Nc,\Omega)$.
\end{lem}

\begin{proof}
	Fix a $(\delta,R)$-lattice $(\lambda_k)_{k\in K}$ on $\Omega'$ for some $\delta>0$ and some $R>1$, and a bounded family $(\varphi_k)$ of positive elements of $C^\infty_c(B_{\Omega'}(e_{\Omega'},R\delta))$ such that
	\[
	\sum_{k\in K} \varphi_k(\,\cdot\, t_k^{-1})=1
	\]
	on $\Omega'$, where $t_k\in T_+$ and $\lambda_k=e_{\Omega'}\cdot t_k$ for every $k\in K$. Define $\psi_k\coloneqq \Fc_\Nc^{-1}(\varphi_k(\,\cdot\, t_k^{-1}))$ for every $k\in K$. In addition, fix $\tau\in \Sc_\Omega(\Nc)$ such that $\tau(0,0)=1$, and define $\tau_j\coloneqq \tau(2^{-j}\,\cdot\,)$ for every $j\in \N$ (so that $\tau_j(\zeta,x)=\tau(2^{-j/2}\zeta,2^{-j}x)$ for every $(\zeta,x)\in \Nc$ and $\Fc_\Nc \tau_j=(\Fc_\Nc \tau)(2^j\,\cdot\,)$). 
	Observe that, for every $k\in K$, there is $j\in\N$ such that $B_{\Omega'}(\lambda_k,R\delta)+2^{-j'}\Supp{\Fc_\Nc \tau}\subseteq B_{\Omega'}(\lambda_k,2R\delta)$ for every $j'\Meg j$.	
	Therefore, there is  an increasing sequence $(K_j)$  of finite subsets of $K$, whose union is $K$, such that 
	\[
	B_{\Omega'}(\lambda_k,R\delta)+\Supp{\Fc_\Nc \tau_j}\subseteq B_{\Omega'}(\lambda_k,2R\delta)
	\]
	for every $k\in K_j$ and for every $j\in \N$. Define
	\[
	\Psi_j u\coloneqq \sum_{k\in K_j} (u*\psi_k) \tau_j
	\]
	for every $u\in \Sc'_{\Omega,L}(\Nc)$, and observe that $\Psi_j u\in \Sc_{\Omega,L}(\Nc)$ by~\cite[Proposition 4.19 and Corollary 4.6]{CalziPeloso}. 
	For every  $k\in K$, define
	\[
	K'_{k}\coloneqq \Set{k'\in K\colon d_{\Omega'}(\lambda_k,\lambda_{k'})\meg 3 R\delta},
	\]
	and observe that there is $N\in \N$ such that $\card(K'_k)\meg N$ for every $k\in K$, thanks to~\cite[Proposition 2.56]{CalziPeloso}.
	Therefore, 
	\[
	(\Psi_j u)*\psi_k=\sum_{k'\in K_j\cap K_k} [(u*\psi_{k'}) \tau_j]*\psi_k
	\]
	for every $j\in \N$, for every $k\in K$, and for every $u\in B^{\vect s}_{p,q}(\Nc,\Omega)$. Fix $\varphi'\in C^\infty_c(\Omega')$ such that $\varphi'=1$ on $B_{\Omega'}(e_{\Omega'},5 R\delta)$ and observe that, if we define $\psi'_k\coloneqq \Fc_\Nc^{-1}(\varphi'(\,\cdot\, t_k^{-1}))$, then
	\[
	[(u*\psi_{k'}) \tau_j]*\psi'_k=(u*\psi_{k'}) \tau_j
	\]
	for every $k'\in K_j\cap K_k$, for every $j\in \N$, for every $k\in K$, and for every $u\in \Sc'_{\Omega,L}(\Nc)$, thanks to~\cite[Corollary 4.6]{CalziPeloso}.	
	Now,~\cite[Corollary 4.10]{CalziPeloso} implies that there is a constant $C_1>0$ such that
	\[
	\norm{u'*\psi_k}_{L^p(\Nc)}\meg C_1 \norm{u'}_{L^p(\Nc)}
	\]
	for every $u'\in \Sc'_{\Omega,L}(\Nc)$ such that $u'=u'*\psi'_k$, for every $k\in K$. Set $C_2\coloneqq \norm{\tau}_{L^\infty(\Nc)}$. Then,
	\[
	\begin{split}
	\norm{(\Psi_j u)*\psi_k}_{L^p(\Nc)}&\meg N^{(\frac 1 p-1)_+} C_1 \sum_{k'\in K_j\cap K_k} \norm{(u*\psi_{k'}) \tau_j}_{L^p(\Nc)}\\
		&\meg N^{(\frac 1 p-1)_+} C_1C_2 \sum_{k'\in K_j\cap K_k} \norm{u*\psi_{k'}}_{L^p(\Nc)}
	\end{split}
	\]
	for every $u\in B^{\vect s}_{p,q}(\Nc,\Omega)$, for every $k\in K$, and for every $j\in \N$. 
	Now, by~\cite[Corollary 2.49]{CalziPeloso}, there is a constant $C_3>0$ such that
	\[
	\frac{1}{C_3}\Delta^{\vect s}_{\Omega'}(\lambda_k)\meg \Delta^{\vect s}_{\Omega'}(\lambda_{k'})\meg C_3\Delta^{\vect s}_{\Omega'}(\lambda_k)
	\]
	for every $k\in K$ and for every $k'\in K_k$, so that
	\[
	\norm*{\Delta^{\vect s}_{\Omega'}(\lambda_k) \norm{(\Psi_j u)*\psi_k}_{L^p(\Nc)}}_{\ell^q(K)}\meg N^{(\frac 1 p-1)_++\max(1/q,1)} C_1 C_2 C_3 \norm*{\Delta^{\vect s}_{\Omega'}(\lambda_k) \norm{ u*\psi_k}_{L^p(\Nc)}}_{\ell^q(K_j)}
	\]
	for every $j\in \N$ and for every $u\in B^{\vect s}_{p,q}(\Nc,\Omega)$. Thus, the $\Psi_j$ are equicontinuous on $B^{\vect s}_{p,q}(\Nc,\Omega)$. 
	Then, take $u\in B^{\vect s}_{p,q}(\Nc,\Omega)$ and let us prove that $(\Psi_j u)$ converges to $u$ in $\sigma^{\vect s}_{p,q}$ for $j\to \infty$. By~\cite[Corollary 4.25]{CalziPeloso} it will suffice to prove convergence in $\Sc'_{\Omega,L}(\Nc)$. Then, we are reduced to proving that $\sum_{k\in K_j}(\eta\tau_j)*\psi_k$ converges to $\eta$ for $j\to \infty$, for every $\eta\in \Sc_{\Omega,L}(\Nc)$. Nonetheless, it is clear that 
	\[
	\sum_{k\in K_j}(\eta\tau_j)*\psi_k=\eta \tau_j
	\]
	if $j$ is sufficiently large (cf.~\cite[Corollary 4.6]{CalziPeloso}), so that the proof is complete.
\end{proof}

\begin{proof}[Proof of Proposition~\ref{prop:8}.]
	(1) $\implies$ (2). Observe that $A^{p,q}_{\vect s,0}(D)=\widetilde A^{p,q}_{\vect s,0}(D)\neq \Set{0}$, so that $\vect s\succ\vect 0$ by~\cite[Proposition 3.5]{CalziPeloso}.
	Then, it will suffice to prove that, if $\widetilde A^{p,q}_{\vect s,0}(D)\subseteq A^{p,q}_{\vect s}(D)$, then  $\widetilde A^{p,q}_{\vect s}(D)= A^{p,q}_{\vect s}(D)$, which will prove also the first assertion.	
	Take $(\Psi_j)$ as in Lemma~\ref{lem:3}. Then, 
	\[
	\lim_{j\to \infty}\Ec \Psi_j u=\Ec u
	\]
	pointwise on $D$, for every $u\in B^{-\vect s}_{p,q}(\Nc,\Omega)$, thanks to~\cite[Lemma 5.1]{CalziPeloso} and Lemma~\ref{lem:3}.  In addition, the sequence $(\Ec \Psi_j u)$ is bounded in $\widetilde A^{p,q}_{\vect s,0}(D)=A^{p,q}_{\vect s,0}(D)$, so that, by lower semi-continuity,
	\[
	\norm{\Ec u}_{A^{p,q}_{\vect s}(D)}\meg \liminf_{j\to \infty} \norm{\Ec \Psi_j u}_{A^{p,q}_{\vect s}(D)}<\infty.
	\] 
	Hence, $\widetilde A^{p,q}_{\vect s}(D) \subseteq A^{p,q}_{\vect s}(D)$, whence $\widetilde A^{p,q}_{\vect s}(D) = A^{p,q}_{\vect s}(D) $ by Proposition~\ref{prop:1}. 
	
	(2) $\implies$ (1).  Observe first that, since $ A^{p,q}_{\vect s}(D)=\widetilde A^{p,q}_{\vect s}(D)\neq \Set{0}$ and $\vect s\succ\vect 0$,~\cite[Proposition 3.5]{CalziPeloso} implies that $\vect s\succ \frac{1}{2 q}\vect m$ so that, in particular, $A^{p,q}_{\vect s,0}(D)\neq \Set{0}$.	
	Then, take $f\in \widetilde A^{p,q}_{\vect s,0}(D)$, and observe that there is a sequence $(f_j)$ of elements of $\Ec(\Sc_{\Omega,L}(\Nc))$ which converges to $f$ in $\widetilde A^{p,q}_{\vect s,0}(D)$, hence in $\widetilde A^{p,q}_{\vect s}(D)=A^{p,q}_{\vect s}(D)$, thanks to~\cite[Theorem 4.23]{CalziPeloso}. Since $f_j\in A^{p,q}_{\vect s,0}(D)$ for every $j\in \N$, thanks to Proposition~\ref{prop:1}, this implies that $f\in A^{p,q}_{\vect s,0}(D)$. The assertion follows by Proposition~\ref{prop:1} as before.
\end{proof}

We now recall (and extend) various notions of atomic decomposition introduced in~\cite{CalziPeloso}.

\begin{deff}\label{def:2}
	Take $p,q\in(0,\infty]$ and  $\vect s,\vect{s'}\in \R^r$. We say that (weak) property $\atomic^{p,q}_{\vect s,\vect {s'}}$ holds if for every $\delta_0>0$ there are a $(\delta,4)$-lattice $(\zeta_{j,k},z_{j,k})_{j\in J,k\in K}$ on $D$, with $\delta\in (0,\delta_0]$, such that the mapping
	\[
	\Psi\colon \ell^{p,q}(J,K)\ni \lambda \mapsto  \sum_{j,k} \lambda_{j,k} B^{\vect{s'}}_{(\zeta_{j,k},z_{j,k})} \Delta_\Omega^{(\vect b+\vect d)/p-\vect s-\vect{s'}}(h_k)\in \Hol(D),
	\]
	with $h_k\coloneqq \rho(\zeta_{j,k},z_{j,k})$ for every $(j,k)\in J\times K$, is well defined (with locally uniform convergence of the sum) and maps $\ell^{p,q}(J,K)$ continuously into $A^{p,q}_{\vect s}(D)$.
	
	We say that strong property $\atomic^{p,q}_{\vect s,\vect{s'}}$ holds if $\Psi$ has the preceding properties for every $(\delta,R)$-lattice on $D$, with $\delta>0$ and $R>1$.
	
	We say that strong property $\atomics^{p,q}_{\vect s,\vect{s'}}$ holds if strong property $\atomic^{p,q}_{\vect s,\vect{s'}}$ holds and for every $R_0>1$ there is $\delta_0>0$ such that $\Psi(\ell^{p,q}(J,K))=A^{p,q}_{\vect s}(D)$  whenever $(\zeta_{j,k},z_{j,k})$ is a $(\delta,R)$-lattice with $\delta\in (0,\delta_0]$ and $R\in (1,R_0]$.
	
	We define weak  and  strong properties $\atomic^{p,q}_{\vect s,\vect{s'},0}$ and $\atomics^{p,q}_{\vect s,\vect{s'},0}$ analogously.
	
	Finally, if $\vect s\succ \frac 1 p (\vect b+\vect d)+\frac{1}{2 q'}\vect{m'}$, then we define (weak or strong) properties $(\widetilde L)^{p,q}_{\vect s,\vect{s'},0}$, $(\widetilde L')^{p,q}_{\vect s,\vect{s'},0}$, $(\widetilde L)^{p,q}_{\vect s,\vect{s'}}$, and $(\widetilde L')^{p,q}_{\vect s,\vect{s'}}$, replacing the spaces $A^{p,q}_{\vect s,0}(D)$ and $A^{p,q}_{\vect s}(D)$ with $\widetilde A^{p,q}_{\vect s,0}(D)$ and $\widetilde A^{p,q}_{\vect s}(D)$, respectively.
\end{deff}

\begin{prop}\label{prop:11}
	Take $p,q\in (0,\infty]$ and $\vect s, \vect{s'}\in \R^r$. 
	If $(\zeta_{j,k},z_{j,k})_{j\in J,k\in K}$ is a $(\delta,R)$-lattice on $D$ for some $\delta>0$ and some $R>1$ and the mapping
	\[
	\Psi\colon\C^{(J\times K)}\ni\lambda \mapsto \sum_{j,k} \lambda_{j,k} B^{\vect s'}_{(\zeta_{j,k},z_{j,k})} \Delta_\Omega^{(\vect b+\vect d)/p-\vect s-\vect s'}(h_k)\in \Hol(D)
	\]
	induces a continuous linear mapping $\ell^{p,q}_0(J,K)\to A^{p,q}_{\vect s}(D)$, then it also induces a continuous linear mapping $\ell^{p,q}(J,K)\to A^{p,q}_{\vect s}(D)$, defined in the same way with locally uniform convergence of the sum.
	
	In particular, the following conditions are equivalent:
	\begin{enumerate}
		\item[\textnormal{(1)}] weak (resp.\ strong) property $\atomic^{p,q}_{\vect s,\vect{s'},0}$ holds;
		
		\item[\textnormal{(2)}] $\vect s\succ \vect0$, and  weak (resp.\ strong) property $\atomic^{p,q}_{\vect s,\vect{s'}}$ holds;
	\end{enumerate}
\end{prop}

\begin{proof}
	The equivalence of (1) and (2) follow easily from the first assertion and~\cite[Proposition 2.41 and Lemma 3.29]{CalziPeloso}.
	
	Then, let us prove the first assertion. By~\cite[Lemma 3.29]{CalziPeloso}, $\widetilde A^{p'',q''}_{\vect s-(1/p-1)_+(\vect b+\vect d)}(D)$ and $\widetilde A^{p',q'}_{ (\vect b+\vect d)/p-\vect s-\vect s',0}(D)$ are well defined, while $A^{p'',q''}_{\vect s-(1/p-1)_+(\vect b+\vect d)}(D), A^{p',q'}_{ (\vect b+\vect d)/p-\vect s-\vect s'}(D)\neq \Set{0}$   (notice that $p''=\max(1,p)$ and $q''=\max(1,q)$).   
	Define $V\coloneqq A^{p',q'}_{ (\vect b+\vect d)/p-\vect s-\vect s'}(D)\cap \widetilde A^{p',q'}_{ (\vect b+\vect d)/p-\vect s-\vect s',0}(D)$, endowed with the topology induced by $A^{p',q'}_{ (\vect b+\vect d)/p-\vect s-\vect s'}(D)$, and observe that  $V$ is a closed subspace of $A^{p',q'}_{ (\vect b+\vect d)/p-\vect s-\vect s'}(D)$ by the continuity of the inclusion $A^{p',q'}_{ (\vect b+\vect d)/p-\vect s-\vect s'}(D)\subseteq \widetilde A^{p',q'}_{ (\vect b+\vect d)/p-\vect s-\vect s'}(D)$. Denote by $\iota_1\colon A^{p'',q''}_{\vect s-(1/p-1)_+(\vect b+\vect d)}(D)\to V'$ and $\iota_2\colon \widetilde A^{p'',q''}_{\vect s-(1/p-1)_+(\vect b+\vect d)}(D)\to \widetilde A^{p',q'}_{ (\vect b+\vect d)/p-\vect s-\vect s',0}(D)'$ the continuous antilinear mapping and the antilinear isomorphism induced by the sesquilinear form (cf.~Proposition~\ref{prop:10})
		\[
		(f,g)\mapsto \int_D f \overline g (\Delta^{-\vect s'}_\Omega\circ \rho)\,\dd \nu_D.
		\]
		Denote by $\iota_3\colon V\to\widetilde A^{p',q'}_{ (\vect b+\vect d)/p-\vect s-\vect s',0}(D)  $ and $\iota_4\colon A^{p'',q''}_{\vect s-(1/p-1)_+(\vect b+\vect d)}(D)\to \widetilde A^{p'',q''}_{\vect s-(1/p-1)_+(\vect b+\vect d)}(D)$ the canonical (continuous linear) mappings, so that $\iota_3$ is one-to-one and has a dense image (since $V$ contains $\Ec(\Sc_{\Omega,L}(\Nc))$, cf.~\cite[Theorem 4.23]{CalziPeloso} and Proposition~\ref{prop:1}). Then, $\trasp\iota_3$ is one-to-one and
		\[
		\iota_1= \trasp \iota_3\circ\iota_2\circ \iota_4.
		\]
		In addition, using~\cite[Theorem 3.22]{CalziPeloso} and the canonical mappings $\ell^{p,q}(J,K)\to\ell^{p'',q''}(J,K)\to \ell^{p',q'}(J,K)'$, we see that $\trasp \iota_3\circ\iota_2\circ \Psi$ extends to a continuous linear mapping
		\[
		\ell^{p,q}(J,K)\ni \lambda \mapsto \sum_{j,k} \lambda_{j,k}( \trasp \iota_3\circ \iota_2)\big(B^{\vect{s'}}_{(\zeta_{j,k},z_{j,k})}\big) \Delta_\Omega^{(\vect b+\vect d)/p-\vect s-\vect{s'}}(h_k)\in V',
		\]
		with the sum converging in the weak topology $\sigma(V',V)$ (cf.~the proof of~\cite[Proposition 3.39]{CalziPeloso}). Observe that, by compactness, 
		\[
		\lim_{H, \Uf} \sum_{(j,k)\in H} \lambda_{j,k} B^{\vect{s'}}_{(\zeta_{j,k},z_{j,k})} \Delta_\Omega^{(\vect b+\vect d)/p-\vect s-\vect{s'}}(h_k)
		\]
		converges in  $\widetilde \sigma^{p'',q''}_{\vect s-(1/p-1)_+(\vect b+\vect d)}$  for every ultrafilter $\Uf$ on the set of finite subsets of $J\times K$ which is finer than the section filter associated with $\subseteq$. Since $\trasp \iota_3\circ \iota_2$ is one-to-one, this proves that the sum
		\[
		\sum_{j,k} \lambda_{j,k} B^{\vect{s'}}_{(\zeta_{j,k},z_{j,k})} \Delta_\Omega^{(\vect b+\vect d)/p-\vect s-\vect{s'}}(h_k),
		\]
		converges in $\widetilde\sigma^{p'',q''}_{\vect s-(1/p-1)_+(\vect b+\vect d)}$ for every $\lambda\in \ell^{p,q}(J,K)$ (cf.~\cite[Proposition 2 of Chapter I, \S\ 7, No.\ 1]{BourbakiGT1}). 	
		Observe that, since the $B^{\vect s'}_{(\zeta,z)}$ stay in a compact subset of $\widetilde A^{p'',q''}_{\vect s-(1/p-1)_+(\vect b+\vect d),0}(D)$ as $(\zeta,z)$ stays in a compact subset of $D$ (cf.~\cite[Lemmas 3.29 and 5.15]{CalziPeloso}), by means of~\cite[Proposition 3.13 and Lemma 3.29]{CalziPeloso}  we see that convergence in  $\widetilde\sigma^{p'',q''}_{\vect s-(1/p-1)_+(\vect b+\vect d)}$ implies convergence in $\Hol(D)$. Thus, the sum
		\[
		\sum_{j,k} \lambda_{j,k} B^{\vect{s'}}_{(\zeta_{j,k},z_{j,k})} \Delta^{(\vect b+\vect d)/p-\vect s-\vect{s'}}_\Omega(h_k)
		\]
		converges in $\Hol(D)$ for every $\lambda\in \ell^{p,q}(J,K)$, and defines a continuous linear function (of $\lambda$) from $\ell^{p,q}(J,K)$ into $A^{p,q}_{\vect s}(D)$.	
\end{proof}

We are now ready to prove the main result of this section.

\begin{teo}\label{teo:1bis}
	Take $p,q\in [1,\infty]$ and $\vect s,\vect{s'}\in \R^r$ such that the following conditions hold:
	\begin{itemize}
		\item $\vect s\succ \frac 1 p (\vect b+\vect d)+\frac{1}{2 q'}\vect{m'} $;
				
		\item $\vect{s'}\prec \vect b+\vect d-\frac 1 2 \vect m$;
		
		\item $\vect s+\vect{s'}\prec\frac 1 p (\vect b+\vect d)- \frac{1}{2 q}\vect{m'}$.
	\end{itemize}		
	Then, the following conditions are equivalent:
	\begin{enumerate}
		\item[\textnormal{(1)}] $A^{p',q'}_{\vect b+\vect d-\vect s-\vect{s'}}(D)=\widetilde A^{p',q'}_{\vect b+\vect d-\vect s-\vect{s'}}(D)$;

		\item[\textnormal{(2)}] $P_{\vect{s'}}$ induces a continuous linear mapping of $L^{p,q}_{\vect s,0}(D)$ \emph{onto} $\widetilde A^{p,q}_{\vect s,0}(D)$;

		\item[\textnormal{(3)}] $P_{\vect{s'}}$ induces a continuous linear mapping of $L^{p,q}_{\vect s}(D)$ \emph{onto} $\widetilde A^{p,q}_{\vect s}(D)$ such that 
		\[
		P_{\vect{s'}}f(\zeta,z)=c_{\vect{s'}} \int_D f \overline{B_{(\zeta,z)}^{\vect{s'}}}(\Delta^{-\vect{s'}}_\Omega\circ \rho)\,\dd \nu_D
		\]
		for every $f\in L^{p,q}_{\vect s}(D)$ and for every $(\zeta,z)\in D$;
		
		\item[\textnormal{(3$'$)}] $P_{\vect{s'}}$ induces a continuous linear mapping of $L^{p,q}_{\vect s,0}(D)$ into $\widetilde A^{p,q}_{\vect s}(D)$;
		
		\item[\textnormal{(4)}] the sesquilinear mapping
		\[
		\Ec(\Sc_{\Omega,L}(\Nc))\times \Ec(\Sc_{\Omega,L}(\Nc))\ni (f,g)\mapsto \int_D f \overline g (\Delta^{-\vect{s'}}_\Omega\circ\rho)\,\dd \nu_D,
		\]
		extended to $\widetilde A^{p,q}_{\vect s,0}(D)\times\widetilde A^{p',q'}_{\vect b+\vect d-\vect s-\vect{s'}}(D)$ as in Proposition~\ref{prop:10}, induces an antilinear isomorphism of $A^{p',q'}_{\vect b+\vect d-\vect s-\vect{s'}}(D)$ onto $\widetilde A^{p,q}_{\vect s,0}(D)'$;
		
		\item[\textnormal{(5)}] strong property $(\widetilde L')^{p,q}_{\vect s,\vect{s'},0}$ holds;
		
		\item[\textnormal{(5$'$)}] weak property $(\widetilde L)^{p,q}_{\vect s,\vect{s'},0}$ holds;
		
		\item[\textnormal{(6)}] strong property $(\widetilde L')^{p,q}_{\vect s,\vect{s'}}$ holds;
		
		\item[\textnormal{(6$'$)}] weak property $(\widetilde L)^{p,q}_{\vect s,\vect{s'}}$ holds.
	\end{enumerate}
\end{teo}

In the proof of Theorem~\ref{teo:1bis} we shall make use of several duality arguments which are analogous to those employed in the proof of~\cite[Proposition 3.39]{CalziPeloso}.\footnote{We remark explicitly that  the assumption of this latter result are mistakenly stated, and should be: $\vect s \succ \frac{1}{2 q}\vect m, \frac 1 p (\vect b+\vect d)+\frac{1}{2 q'}\vect m'$, $\vect s'\prec \frac{1}{p'}(\vect b+\vect d)-\frac{1}{2 p'}\vect m'$, and $\vect s+\vect s'\prec \frac{1}{ \min(1,p)}(\vect b+\vect d)-\frac{1}{2 q'}\vect m'$ or $\vect s+\vect s'\meg \frac{1}{ \min(1,p)}(\vect b+\vect d)$ if $q'=\infty$. In addition, for conclusion (3) in the cited result to hold, one has to assume further that $p,q\Meg 1$. Nonetheless,~\cite[Corollary 3.40]{CalziPeloso} still holds, thanks to Corollaries~\ref{cor:1} and~\ref{cor:4}.} We shall present them once in the following extension of the implication (5$'$) $\implies$ (1).

\begin{prop}\label{prop:12}
	Take $p,q\in(0,\infty]$,    $\vect s\succ \frac 1 p (\vect b+\vect d)+\frac{1}{2 q'}\vect{m'} $, and $\vect s'\prec \vect b+\vect d-\frac 1 2 \vect m$.  
	Then, weak (resp.\ strong) property $(\widetilde L)^{p,q}_{\vect s,\vect{s'},0}$ holds if and only if weak  (resp.\ strong)   property $(\widetilde L)^{p,q}_{\vect s,\vect{s'}}$ holds, in which case   $(\vect b+\vect d)/\min(1,p)-\vect s-\vect{s'}\succ \frac{1}{p'} (\vect b+\vect d)+\frac{1}{2 q''}\vect m' $    and $A^{p',q'}_{(\vect b+\vect d)/\min(1,p)-\vect s-\vect{s'}}(D)=\widetilde A^{p',q'}_{(\vect b+\vect d)/\min(1,p)-\vect s-\vect{s'}}(D)$.
\end{prop}

\begin{proof}
	  Since $B^{\vect{s'}}_{(\zeta,z)}\in \widetilde A^{p,q}_{\vect s,0}(D)$ if and only if $B^{\vect{s'}}_{(\zeta,z)}\in \widetilde A^{p,q}_{\vect s,0}(D) $ for every $(\zeta,z)\in D$ by~\cite[Lemma 5.15]{CalziPeloso}, weak  (resp.\ strong)  property $(\widetilde L)^{p,q}_{\vect s,\vect{s'}}$ implies weak (resp.\ strong)   property $(\widetilde L)^{p,q}_{\vect s,\vect{s'},0}$. In addition, if either one of the preceding conditions hold, then $\vect s+\vect s'\prec \frac 1 p(\vect b+\vect d)-\frac{1}{2 q}\vect m'$, so that $(\vect b+\vect d)/\min(1,p)-\vect s-\vect{s'}\succ \frac{1}{p'} (\vect b+\vect d)+\frac{1}{2 q''}\vect m' $.  
	Then, assume that weak property $(\widetilde L)^{p,q}_{\vect s,\vect{s'},0}$ holds, and let us prove that $A^{p',q'}_{(\vect b+\vect d)/\min(1,p)-\vect s-\vect{s'}}(D)=\widetilde A^{p',q'}_{(\vect b+\vect d)/\min(1,p)-\vect s-\vect{s'}}(D)$. By~\cite[Theorem 3.22]{CalziPeloso}, there is $\delta_0>0$ such that, if $ (\zeta_{j,k},z_{j,k})_{j\in j,k\in K}$ is a $(\delta,4)$-lattice on $D$ with $\delta\in (0,\delta_0]$, and we define $h_k\coloneqq \rho(\zeta_{j,k},z_{j,k})$ for every $j\in J$ and for every $k\in K$,
	\[
	S\colon \Hol(D)\ni f \mapsto (\Delta_\Omega^{(\vect b+\vect d)/p-\vect s-\vect{s'}}(h_k)f(\zeta_{j,k},z_{j,k}))\in \C^{J\times K},
	\]
	then $S$ induces an isomorphism of $A^{p',q'}_{(\vect b+\vect d)/\min(1,p)-\vect s-\vect s'}(D)$ onto a closed subspace of $\ell^{p',q'}(J,K)$, and $S^{-1}(\ell^{p',q'}(J,K))\cap A^{\infty,\infty}_{(\vect b+\vect d)/p-\vect s-\vect{s'}}(D)=A^{p',q'}_{(\vect b+\vect d)/\min(1,p)-\vect s-\vect s'}(D)$. Since weak property $(\widetilde L)^{p,q}_{\vect s,\vect{s'},0}$ holds, we may assume that the mapping
	\[
	\Psi\colon \ell^{p,q}_0(J,K)\ni \lambda \mapsto \sum_{j,k} \lambda_{j,k} B^{\vect s'}_{(\zeta_{j,k},z_{j,k})} \Delta_\Omega^{(\vect b+\vect d)/p-\vect s-\vect s'}(h_k)\in \widetilde A^{p,q}_{\vect s,0}(D)
	\]
	is well defined and continuous. Now, denote by $\langle \,\cdot\,\vert \,\cdot\,\rangle$ the sesquilinear form on $\widetilde A^{p,q}_{\vect s,0}(D)\times \widetilde A^{p',q'}_{(\vect b+\vect d)/\min(1,p)-\vect s-\vect s'}(D)$ defined in Proposition~\ref{prop:10}. Then, 
	\[
	\langle \Psi(\lambda)\vert f\rangle=\sum_{j,k} \lambda_{j,k}\Delta_\Omega^{(\vect b+\vect d)/p-\vect s-\vect{s'}}(h_k) \int_D  B^{\vect{s'}}_{(\zeta_{j,k},z_{j,k})} \overline f (\Delta_\Omega^{-\vect s'}\circ \rho)\,\dd \nu_D=\frac{1}{c_{\vect{s'}}} \langle \lambda\vert S f\rangle
	\]
	for every $\lambda\in \C^{(J\times K)}$ and for every $f\in \Ec(\Sc_{\Omega,L}(\Nc))\subseteq A^{2,2}_{(\vect b+\vect d-\vect{s'})/2}(D)$, thanks to~\cite[Proposition 3.13]{CalziPeloso}, where $c_{\vect s'}\neq 0$ is a suitable constant. 
	Since $\Ec(\Sc_{\Omega,L}(\Nc))$ is dense in $\widetilde A^{p',q'}_{(\vect b+\vect d)/\min(1,p)-\vect s-\vect{s'}}(D)$ for the topology $\widetilde\sigma^{p',q'}_{(\vect b+\vect d)/\min(1,p)-\vect s-\vect{s'}}$, the preceding equality (excluding the middle term) may be extended to every $f\in \widetilde A^{p',q'}_{(\vect b+\vect d)/\min(1,p)-\vect s-\vect{s'}}(D)$ by continuity. Hence, $S f\in \ell^{p',q'}(J,K)$ for every $f\in \widetilde A^{p',q'}_{(\vect b+\vect d)/\min(1,p)-\vect s-\vect{s'}}(D)$. Since $\widetilde A^{p',q'}_{(\vect b+\vect d)/\min(1,p)-\vect s-\vect{s'}}(D)\subseteq A^{\infty,\infty}_{(\vect b+\vect d)/p-\vect s-\vect{s'}}(D)$ by Proposition~\ref{prop:1}, this implies that $\widetilde A^{p',q'}_{(\vect b+\vect d)/\min(1,p)-\vect s-\vect{s'}}(D)\subseteq A^{p',q'}_{(\vect b+\vect d)/\min(1,p)-\vect s-\vect{s'}}(D)$, whence our assertion.
	
	To conclude, we only need to show that weak  (resp.\ strong)  property $(\widetilde L)^{p,q}_{\vect s,\vect{s'},0}$ implies weak (resp.\ strong)   property $(\widetilde L)^{p,q}_{\vect s,\vect{s'}}$. The proof is now similar to that of Proposition~\ref{prop:11}, and proceeds by means of the duality between the closed subspace $\widetilde A^{p',q'}_{(\vect b+\vect d)/\min(1,p)-\vect s-\vect s',0}(D)$ of $A^{p',q'}_{(\vect b+\vect d)/\min(1,p)-\vect s-\vect s'}(D)$ and the space $\widetilde A^{p'',q''}_{\vect s-(1/p-1)_+(\vect b+\vect d)}(D)$.
\end{proof}

\begin{proof}[Proof of Theorem~\ref{teo:1bis}.]
	(1) $\implies$ (3$'$). Take  $f\in L^{p,q}_{\vect s}(D)\cap L^{2,2}_{(\vect b+\vect d-\vect{s'})/2}(D)$ and $\varphi\in \Ec(\Sc_{\Omega,L}(\Nc))$. Observe that $P_{\vect{s'}}f\in A^{2,2}_{(\vect b+\vect d-\vect{s'})/2}(D)=\widetilde A^{2,2}_{(\vect b+\vect d-\vect{s'})/2}(D)$, so that there is $u\in B^{(\vect{s'}-\vect b-\vect d)/2}_{2,2}(\Nc,\Omega)$ such that $P_{\vect{s'}}f=\Ec u$. Then, by Proposition~\ref{prop:10}, there is a constant $c\neq0$ such that
	\[
	\begin{split}
		\abs*{\langle u\vert \varphi\rangle }=&\abs*{c\int_D \Ec u \,\overline{\Ec (\varphi*I_\Omega^{\vect{s'}-\vect b-\vect d})} (\Delta^{-\vect{s'}}_\Omega\circ \rho)\,\dd \nu_D }\\
			&=\abs*{c\int_D f \overline{\Ec(\varphi*I_\Omega^{\vect{s'}-\vect b-\vect d})} (\Delta^{-\vect{s'}}_\Omega\circ \rho)\,\dd \nu_D}\\
			&\meg \abs{c} \norm{f}_{L^{p,q}_{\vect s}(D)} \norm{\Ec(\varphi*I_\Omega^{\vect{s'}-\vect b-\vect d})}_{A^{p',q'}_{\vect b+\vect d-\vect s-\vect{s'}}(D)}.
	\end{split}
	\]
	Now, by (1) and~\cite[Theorem 4.26]{CalziPeloso}, there are constants $C_1,C_1'>0$ such that 
	\[
	\norm{\Ec(\varphi*I_\Omega^{\vect{s'}-\vect b-\vect d})}_{A^{p',q'}_{\vect b+\vect d-\vect s-\vect{s'}}(D)} \meg C_1 \norm{\varphi*I_\Omega^{\vect{s'}-\vect b-\vect d}}_{B^{\vect s+\vect{s'}-\vect b-\vect d}_{p',q'}(\Nc,\Omega)}\meg C_1'\norm{\varphi}_{B_{p',q'}^{\vect s}(\Nc,\Omega)}
	\]
	for every $u\in \Sc_{\Omega,L}(\Nc)$, for some fixed norms on $B^{\vect s+\vect{s'}-\vect b-\vect d}_{p',q'}(\Nc,\Omega)$ and $B^{\vect s}_{p',q'}(\Nc,\Omega)$. Therefore, by means of~\cite[Theorem 4.23]{CalziPeloso} we see that $u\in B^{-\vect{s}}_{p,q}(\Nc,\Omega)$, that is, $P_{\vect{s'}} f \in \widetilde A^{p,q}_{\vect s}(D)$, and that there is a constant $C_2>0$ such that
	\[
	\norm{P_{\vect{s'}}f}_{\widetilde A^{p,q}_{\vect s}(D)}\meg C_2 \norm{f}_{L^{p,q}_{\vect s}(D)}
	\]
	for every  $f\in L^{p,q}_{\vect s}(D)\cap L^{2,2}_{(\vect b+\vect d-\vect{s'})/2}(D)$.  Then, (3$'$) follows.
	
	(3$'$) $\implies$ (2). Observe that $B^{\vect{s'}}_{(\zeta,z)}\in \widetilde A^{p,q}_{\vect s,0}(D)$ for every $(\zeta,z)\in D$, thanks to~\cite[Lemma 5.15]{CalziPeloso}. In addition, using~\cite[Theorem 2.47, Corollary 5.11, and Proposition 5.13]{CalziPeloso}, we see that the mapping $D\ni (\zeta,z)\mapsto B_{(\zeta,z)}\in \widetilde A^{p,q}_{\vect s,0}(D)$ is actually continuous. Therefore,~\cite[Proposition 8 of Chapter VI, \S\ 1, No.\ 2]{BourbakiInt1} implies that $P_{\vect{s'}}f\in \widetilde A^{p,q}_{\vect s,0}(D)$ for every $f\in C_c(D)$, so that $P_{\vect{s'}}$ induces a continuous linear mapping of $L^{p,q}_{\vect s,0}(D)$ into $\widetilde A^{p,q}_{\vect s,0}(D)$.
	Now, take $\vect{s''}\in \N_{\Omega'}$ so that $\vect s+\vect{s''}\succ \frac{1}{2 q}\vect m+\left(\frac{1}{2 \min(p,p')}-\frac{1}{2 q}\right)_+ \vect{m'}$, so that $A^{p,q}_{\vect s+\vect{s''},0}(D)=\widetilde A^{p,q}_{\vect s+\vect{s''},0}(D)$ by~\cite[Corollary 5.11]{CalziPeloso}. In addition, observe that, by~\cite[Proposition 2.29]{CalziPeloso}, there is a constant $c\neq 0$ such that
	\begin{equation}\label{eq:1}
		P_{\vect{s'}}(f (\Delta_\Omega^{\vect{s''}}\circ \rho))*I^{-\vect{s''}}_\Omega= c P_{\vect{s'}-\vect{s''}} f
	\end{equation}
	for every $f\in C_c(D)$. Therefore,~\cite[Proposition 5.13]{CalziPeloso} implies that $P_{\vect{s'}-\vect{s''}}$ induces a continuous linear mapping of $L^{p,q}_{\vect s+\vect{s''},0}(D)$ into $\widetilde A^{p,q}_{\vect s+\vect{s''},0}(D)=A^{p,q}_{\vect s+\vect{s''},0}(D)$. Then, $P_{\vect{s'}-\vect{s''}}(L^{p,q}_{\vect s+\vect{s''},0}(D))=\widetilde A^{p,q}_{\vect s+\vect{s''},0}(D)$, so that, by~\cite[Proposition 5.13]{CalziPeloso} again, $P_{\vect{s'}}(L^{p,q}_{\vect s,0}(D))=\widetilde A^{p,q}_{\vect s,0}(D)$.
	
	(2) $\implies$ (1). Take $f\in C_c(D) $ and $\varphi \in \Sc_{\Omega,L}(\Nc)$. Then, by Proposition~\ref{prop:10},
	\[
	\abs*{\int_D f \overline{\Ec \varphi} (\Delta^{-\vect{s'}}_\Omega\circ \rho)\,\dd \nu_D}=\abs*{\int_D P_{\vect{s'}}f \overline{\Ec \varphi} (\Delta^{-\vect{s'}}_\Omega\circ \rho)\,\dd \nu_D}\meg \norm{P_{\vect{s'}}}_{\Lin(L^{p,q}_{\vect s,0}(D);\widetilde A^{p,q}_{\vect s,0}(D))} \norm{f}_{L^{p,q}_{\vect s}(D)} \norm{\Ec \varphi}_{\widetilde A^{p',q'}_{\vect b+\vect d-\vect s-\vect{s'}}(D)}
	\]
	with suitable choices of norms on $\widetilde A^{p,q}_{\vect s,0}(D)$ and $\widetilde A^{p',q'}_{\vect b+\vect d-\vect s-\vect{s'}}(D)$. By the arbitrariness of $f$, this implies that
	\[
	\norm{\Ec \varphi}_{A^{p',q'}_{\vect b+\vect d-\vect s-\vect{s'}}(D)}\meg \norm{P_{\vect{s'}}}_{\Lin(L^{p,q}_{\vect s,0}(D);\widetilde A^{p,q}_{\vect s,0}(D))} \norm{\Ec \varphi}_{\widetilde A^{p',q'}_{\vect b+\vect d-\vect s-\vect{s'}}(D)}
	\]
	for every $\varphi \in \Sc_{\Omega,L}(\Nc)$. Hence, $\widetilde A^{p',q'}_{\vect b+\vect d-\vect s-\vect{s'},0}(D)\subseteq A^{p',q'}_{\vect b+\vect d-\vect s-\vect{s'}}(D)$ continuously, whence (1) by Proposition~\ref{prop:8}.

	(2) $\implies$ (3). This follows from formula~\eqref{eq:1} (arguing as in the proof of the implication (3$'$) $\implies$ (2)), which is readily extended to every $f\in L^{p,q}_{\vect s}(D)$ taking into account the fact that, by~\cite[Lemma 5.15]{CalziPeloso}, $B^{\vect{s'}}_{(\zeta,z)}\in \widetilde A^{p',q'}_{\vect b+\vect d-\vect s-\vect{s'}}(D)=A^{p',q'}_{\vect b+\vect d-\vect s-\vect{s'}}(D)$ for every $(\zeta,z)\in D$ (since we proved that conditions (1), (2), and (3$'$) are equivalent), which gives convergence of $P_{\vect{s'}}(f_j)$ in $\Hol(D)$ when $(f_j)$ converges almost everywhere to $f\in L^{p,q}_{\vect s}(D)$ and $\abs{f_j}\meg \abs{f}$ for every $j\in \N$. 
	
	(3) $\implies$ (3$'$). Obvious.

	(1) $\iff$ (4). This is a consequence of Proposition~\ref{prop:10} and of the continuous inclusion $A^{p',q'}_{\vect b+\vect d-\vect s-\vect{s'}}(D)\subseteq\widetilde A^{p',q'}_{\vect b+\vect d-\vect s-\vect{s'}}(D)$.
	
	(4) $\implies$ (5). This follows by transposition, using~\cite[Theorem 3.22]{CalziPeloso} and~\cite[Corollary 3 to Theorem 1 of Chapter IV, \S\ 4, No.\ 2, and Proposition 5 of Chapter IV, \S\ 1, No.\ 3]{BourbakiTVS} (cf.~Proposition~\ref{prop:12} and~\cite[Proposition 3.39]{CalziPeloso}).
	
	(1) $\implies$ (6). This follows by transposition since $\widetilde A^{p,q}_{\vect s}(D)$ canonically identifies with the dual of the closed subspace $\widetilde A^{p',q'}_{\vect b+\vect d-\vect s-\vect s',0}(D)$ of $A^{p',q'}_{\vect b+\vect d-\vect s-\vect s'}(D)$, using~\cite[Theorem 3.22]{CalziPeloso} and~\cite[Corollary 1 to Theorem 1 of Chapter IV, \S\ 4, No.\ 2]{BourbakiTVS} (cf.~Proposition~\ref{prop:12} and~\cite[Proposition 3.39]{CalziPeloso}).
	
	(5) $\implies$ (5$'$); (6) $\implies$ (6$'$). Obvious.
	
	(6$'$) $\iff$ (5$'$); (5$'$) $\implies$ (1). This is a consequence of Proposition~\ref{prop:12}. 
\end{proof}

\begin{cor}\label{cor:1}
	Take $p,q\in [1,\infty]$ and $\vect s,\vect{s'}\in \R^r$ such that the following conditions hold:
	\begin{itemize}
		\item $\vect s\succ \frac 1 p (\vect b+\vect d)+\frac{1}{2 q'}\vect{m'}  $; 
		
		\item $\vect{s'}\prec \vect b+\vect d-\frac 1 2 \vect m$;
		
		\item $\vect s+\vect{s'}\prec\frac 1 p (\vect b+\vect d)- \frac{1}{2 q}\vect{m'}$.
	\end{itemize}		
	Then, the following conditions are equivalent:
	\begin{enumerate}
		\item[\textnormal{(1)}] $A^{p,q}_{\vect s,0}(D)=\widetilde A^{p,q}_{\vect s,0}(D)$ (resp.\ $A^{p,q}_{\vect s}(D)=\widetilde A^{p,q}_{\vect s}(D)$) and $A^{p',q'}_{\vect b+\vect d-\vect s-\vect{s'}}(D)=\widetilde A^{p',q'}_{\vect b+\vect d-\vect s-\vect{s'}}(D)$;
		
		\item[\textnormal{(2)}] $P_{\vect{s'}}$ induces a continuous linear mapping of $L^{p,q}_{\vect s,0}(D)$ into $L^{p,q}_{\vect s}(D)$   and $\vect s\succ \vect 0$ (resp.\ $\vect s\Meg \vect 0$);\footnote{Notice that the condition $\vect s\Meg \vect 0$ is empty, since it is a consequence of the continuity of $P_{\vect s'}$ on $L^{p,q}_{\vect s,0}(D)$.}  
				
		\item[\textnormal{(3)}] $P_{\vect{s'}}$ induces a continuous linear projector of $L^{p,q}_{\vect s,0}(D)$ \emph{onto} $ A^{p,q}_{\vect s,0}(D)$ (resp.\ of $L^{p,q}_{\vect s}(D)$ \emph{onto} $ A^{p,q}_{\vect s}(D)$) and of $L^{p',q'}_{\vect b+\vect d-\vect s-\vect{s'}}(D)$ \emph{onto} $A^{p',q'}_{\vect b+\vect d-\vect s-\vect{s'}}(D)$;
				
		\item[\textnormal{(4)}]  $ \vect s\succ \frac{1}{2 q}\vect m$ (resp.\ $\vect s\Meg \vect 0$ if $q=\infty$) and  the sesquilinear mapping
		\[
		(f,g)\mapsto \int_D f \overline g (\Delta^{-\vect{s'}}_\Omega\circ\rho)\,\dd \nu_D,
		\]
		induces an antilinear isomorphism of $A^{p',q'}_{\vect b+\vect d-\vect s-\vect{s'}}(D)$ onto $A^{p,q}_{\vect s,0}(D)'$ (resp.\ onto the dual of the closed vector subspace of  $A^{p,q}_{\vect s}(D)$ generated by the $B^{\vect s'}_{(\zeta,z)}$, $(\zeta,z)\in D$);

		\item[\textnormal{(4$'$)}]  $ \vect s\succ \frac{1}{2 q}\vect m$ (resp.\ $\vect s\Meg \vect 0$ if $q=\infty$) and  the sesquilinear mapping
		\[
		(f,g)\mapsto \int_D f \overline g (\Delta^{-\vect{s'}}_\Omega\circ\rho)\,\dd \nu_D,
		\]
		induces an antilinear isomorphism of $A^{p',q'}_{\vect b+\vect d-\vect s-\vect{s'}}(D)$ onto the dual of the closed subspace $\widetilde A^{p,q}_{\vect s,0}(D)\cap A^{p,q}_{\vect s}(D)$ of $A^{p,q}_{\vect s}(D)$;
		
		\item[\textnormal{(5)}] strong properties $(L')^{p,q}_{\vect s,\vect{s'},0}$ (resp.\ $(L')^{p,q}_{\vect s,\vect{s'}}$) and $(L')^{p',q'}_{\vect b+\vect d-\vect s-\vect{s'},\vect{s'}}$  hold;
				
		\item[\textnormal{(6)}] weak property $(L)^{p,q}_{\vect s,\vect{s'},0}$  (resp.\ $(L)^{p,q}_{\vect s,\vect{s'}}$) holds.
	\end{enumerate}
\end{cor}

This extends~\cite[Theorem 1.6]{Bekolleetal}, where the equivalence of (2) and (3) is proved when $p=q\in (1,\infty)$, $\vect s\in \R \vect d$, and $D$ is an irreducible symmetric tube domain.

\begin{proof}
	Theorem~\ref{teo:1bis} shows that (1) implies (2), (4$'$), and (5). In addition, (2) is equivalent to (3) thanks to~\cite[Proposition 5.21]{CalziPeloso}, while (5) clearly implies (6). By means of Theorem~\ref{teo:1bis}, we also see that (3) implies (1). Let us prove that (6) implies (1). Observe that (6) implies weak property $(\widetilde L)^{p,q}_{\vect s,\vect s',0}$ (resp.\ $(\widetilde L)^{p,q}_{\vect s,\vect s'}$), so that Theorem~\ref{teo:1bis} implies that strong properties $(\widetilde L')^{p,q}_{\vect s,\vect s',0}$ and $(\widetilde L')^{p,q}_{\vect s,\vect s'}$ hold, and that $A^{p',q'}_{\vect b+\vect d-\vect s-\vect{s'}}(D)=\widetilde A^{p',q'}_{\vect b+\vect d-\vect s-\vect{s'}}(D)$. This, together with the assumption (6), implies (1).

	Next, observe that	Theorem~\ref{teo:1bis} again shows that (1) implies (4), since the closed vector subspace of $A^{p,q}_{\vect s}(D)=\widetilde A^{p,q}_{\vect s}(D)$ generated by the $B^{\vect s'}_{(\zeta,z)}$, $(\zeta,z)\in D$, is $\widetilde A^{p,q}_{\vect s,0}(D)$ as strong property $(\widetilde L')^{p,q}_{\vect s,\vect{s'},0}$ holds (cf.~Proposition~\ref{prop:8} and~\cite[Lemma 5.15]{CalziPeloso}). In order to conclude, it then suffices to prove that both (4) and (4$'$) imply (6). Arguing as in the proof of Proposition~\ref{prop:12} and using~\cite[Theorem 3.22]{CalziPeloso}, this is easily established.	
\end{proof}

\begin{cor}\label{prop:9}
	Take $p,q\in [1,\infty]$, $\vect s\in \R^r$, and $\vect{s'}\prec\vect b+\vect d-\frac{1}{2}\vect{m}$. Then, the following conditions are equivalent:
	\begin{enumerate}
		\item[\textnormal{(1)}] $P_{\vect{s'}}$ induces a continuous linear projector of $L^{p,q}_{\vect s,0}(D)$ onto $A^{p,q}_{\vect s,0}(D)$;
		
		\item[\textnormal{(2)}] $\vect s\succ\vect 0$ and $P_{\vect{s'}}$ induces a continuous linear projector of $L^{p,q}_{\vect s}(D)$ onto $A^{p,q}_{\vect s}(D)$.
	\end{enumerate}
\end{cor}

\begin{proof}
	This follows from Corollary~\ref{cor:1} and~\cite[Proposition 5.20]{CalziPeloso}.
\end{proof}

\begin{cor}\label{cor:11}
	Take $p,q\in [1,\infty]$ and $\vect s,\vect s'\in \R^r$. Then, the following conditions are equivalent:
	\begin{enumerate}
		\item[\textnormal{(1)}] strong properties $(L')^{p,q}_{\vect s,\vect s',0}$ and $(L')^{p',q'}_{\vect b+\vect d-\vect s-\vect s',\vect s'}$ hold;
		
		\item[\textnormal{(1$'$)}] weak property $(L)^{p,q}_{\vect s,\vect s',0}$ holds;
		
		\item[\textnormal{(2)}] $\vect s\succ\vect 0$ and strong properties $(L')^{p,q}_{\vect s,\vect s'}$ and $(L')^{p',q'}_{\vect b+\vect d-\vect s-\vect s',\vect s'}$ hold;
		
		\item[\textnormal{(2$'$)}] $\vect s\succ\vect 0$ and weak property $(L)^{p,q}_{\vect s,\vect s'}$ holds.		
	\end{enumerate}
\end{cor}

\begin{proof}
	This follows from Corollary~\ref{cor:1} and~\cite[Lemma 3.29]{CalziPeloso}.
\end{proof}

\begin{cor}\label{cor:3}
	Take $p,q\in [1,\infty]$, and $\vect s,\vect{s'}\in \R^r$ such that the following hold:
	\begin{itemize}
		\item $\vect s\succ \frac{1}{p}(\vect b+\vect d)+\frac{1}{2 q'}\vect {m'}$;
		
		\item $\vect{s'}\prec\vect b+\vect d-\frac 1 2 \vect m$;
		
		\item	$\vect s+\vect {s'}\prec \vect b+\vect d-\frac{1}{2 q'}\vect m,\frac 1 p (\vect b+\vect d)-\frac{1}{2 q}\vect{m'}$.
	\end{itemize}
	Then, the following conditions are equivalent:
	\begin{enumerate}
		\item[\textnormal{(1)}] $\widetilde A^{p,q}_{\vect s}(D)= A^{p,q}_{\vect s}(D)$ and $\widetilde A^{p',q'}_{\vect b+\vect d-\vect s-\vect{s'},0}(D)= A^{p',q'}_{\vect b+\vect d-\vect s-\vect{s'},0}(D)$;
		
		\item[\textnormal{(2)}]  the sesquilinear form
		\[
		A^{p',q'}_{\vect b+\vect d-\vect s-\vect{s'},0}(D)\times A^{p,q}_{\vect s}(D)\ni (f,g)\mapsto \int_D f \overline g (\Delta_\Omega^{-\vect{s'}}\circ \rho)\,\dd \nu_D\in \C
		\]
		induces an antilinear isomorphism of $A^{p,q}_{\vect s}(D)$ onto $A^{p',q'}_{\vect b+\vect d-\vect s-\vect{s'},0}(D)'$;
		
		\item[\textnormal{(3)}] weak property $\atomic^{p,q}_{\vect s,\vect{s'}}$ holds;
		
		\item[\textnormal{(3$'$)}] weak property $\atomic^{p',q'}_{\vect b+\vect d-\vect s-\vect{s'},\vect{s'}}$ holds;
		
		\item[\textnormal{(4)}]  strong properties $\atomics^{p,q}_{\vect s,\vect{s'}}$, $\atomics^{p',q'}_{\vect b+\vect d-\vect s-\vect{s'},\vect{s'},0}$, and $\atomics^{p',q'}_{\vect b+\vect d-\vect s-\vect{s'},\vect{s'}}$  hold;
		
		\item[\textnormal{(5)}]  $P_{\vect{s'}}$ induces a continuous linear projector of $L^{p,q}_{\vect s}(D)$ onto $A^{p,q}_{\vect s}(D)$.		
	\end{enumerate}
\end{cor}

Notice that the condition $\vect s+\vect{s'}\prec \vect b+\vect d-\frac{1}{2 q'}\vect m$ must be imposed to ensure that $A^{p',q'}_{\vect b+\vect d-\vect s-\vect{s'},0}(D)\neq \Set{0}$, for otherwise condition (2) could be trivial. One may also have imposed that $A^{p,q}_{\vect s}(D)\neq \Set{0}$, but this condition, expressed in terms of $\vect s$, would have required to treat separately the case $q=\infty$.

The condition $\vect{s'}\prec \vect b+\vect d-\frac 1 2 \vect m$ is necessary for the operator $P_{\vect{s'}}$ to be defined, and is implied by the other assumptions and each one of the conditions (1), (2), (3), (3$'$),  and (4).

\begin{proof}
	The equivalence of (1), (2), (3$'$), (4), and (5) follows from Corollary~\ref{cor:1}. In addition, (5) implies (3) by~\cite[Proposition 5.24]{CalziPeloso}. It will then suffice to prove that (3) implies (3$'$). Then, assume that weak property $(L)^{p,q}_{\vect s,\vect{s'}}$ holds, so that also weak property $(\widetilde L)^{p,q}_{\vect s,\vect{s'}}$ holds and the sesquilinear form of (2) induces an antilinear isomorphism of $A^{p',q'}_{\vect b+\vect d-\vect s-\vect{s'}}(D)$ onto $\widetilde A^{p,q}_{\vect s,0}(D)'$ by Proposition~\ref{prop:12}. In addition, Theorem~\ref{teo:1bis} implies that strong property $(\widetilde L')^{p,q}_{\vect s,\vect s'} $ holds. This latter fact, together with weak property $(L)^{p,q}_{\vect s,\vect{s'}}$ and the inclusion $A^{p,q}_{\vect s}(D)\subseteq \widetilde A^{p,q}_{\vect s}(D)$ (cf.~Proposition~\ref{prop:1}), implies that $A^{p,q}_{\vect s}(D)= \widetilde A^{p,q}_{\vect s}(D)$, so that Corollary~\ref{cor:1} implies that strong property $(L)^{p',q'}_{\vect b+\vect d-\vect s-\vect s'}$ holds, whence (3$'$).
\end{proof}

\begin{cor}
	Take $p,q\in [1,\infty]$ with $q'\meg \min(p,p')$, and $\vect s,\vect{s'}\in \R^r$ such that the following hold:
	\begin{itemize}
		\item$\vect s\succ\frac 1 p(\vect b+\vect d)+\frac{1}{2 q'}\vect {m'}$;
		
		\item $\vect{s'}\prec \vect b+\vect d-\frac 1 2 \vect m$;
		
		\item	$\vect s+\vect {s'}\prec\vect b+\vect d-\frac{1}{2 q'}\vect m, \frac1 p (\vect b+\vect d)-\frac{1}{2 q}\vect{m'}$.
	\end{itemize}
	Then, the following conditions are equivalent:
	\begin{enumerate}
		\item[\textnormal{(1)}] $\widetilde A^{p,q}_{\vect s}(D)= A^{p,q}_{\vect s}(D)$;
		
		\item[\textnormal{(2)}]  the sesquilinear form
		\[
		A^{p',q'}_{\vect b+\vect d-\vect s-\vect{s'},0}(D)\times A^{p,q}_{\vect s}(D)\ni (f,g)\mapsto \int_D f \overline g (\Delta_\Omega^{-\vect{s'}}\circ \rho)\,\dd \nu_D\in \C
		\]
		induces an antilinear isomorphism of $A^{p,q}_{\vect s}(D)$ onto $A^{p',q'}_{\vect b+\vect d-\vect s-\vect{s'},0}(D)'$;
		
		\item[\textnormal{(3)}] weak property $\atomic^{p,q}_{\vect s,\vect{s'}}$ holds;
		
		\item[\textnormal{(3$'$)}] weak property $\atomic^{p',q'}_{\vect b+\vect d-\vect s-\vect s',\vect{s'}}$ holds;
		
		\item[\textnormal{(4)}]  strong properties $\atomics^{p,q}_{\vect s,\vect{s'}}$, $\atomics^{p',q'}_{\vect b+\vect d-\vect s-\vect{s'},\vect{s'},0}$, and $\atomics^{p',q'}_{\vect b+\vect d-\vect s-\vect{s'},\vect{s'}}$  hold;
		
		\item[\textnormal{(5)}]  $P_{\vect{s'}}$ induces a continuous linear projector of $L^{p,q}_{\vect s}(D)$ onto $A^{p,q}_{\vect s}(D)$.		
	\end{enumerate}
\end{cor}

In particular, this applies to the case $p=q\Meg 2$. 

\begin{proof}
	Apply Proposition~\ref{prop:10} to show that $\widetilde A^{p',q'}_{\vect b+\vect d-\vect s-\vect{s'},0}(D)=A^{p',q'}_{\vect b+\vect d-\vect s-\vect{s'},0}(D)$, and then apply Corollary~\ref{cor:3}.
\end{proof}

\begin{cor}
	Take $\vect s,\vect{s'}\in \R^r$, and $\vect{s''}\in \N_{\Omega'}$, and take $p,q\in [1,\infty]$. Assume that the following conditions are satisfied:
	\begin{itemize}
		\item $\vect s+\vect{s''}\succ\frac{1}{p}(\vect b+\vect d)+\frac{1}{2 q'}\vect {m'}$;
		
		\item $\vect{s'}\prec \vect b+\vect d-\frac 1 2 \vect m$;
		
		\item $\vect s+\vect{s'}\prec \frac{1}{p}(\vect b+\vect d)-\frac{1}{2 q}\vect{m'}$.
	\end{itemize}
	Then, the following conditions are equivalent:
	\begin{itemize}
		\item[\textnormal{(1)}] $A^{p',q'}_{\vect b+\vect d-\vect s-\vect{s'}}(D)=\widetilde A^{p',q'}_{\vect b+\vect d-\vect s-\vect{s'}}(D)$;
		
		\item[\textnormal{(2)}]  $P_{\vect{s'}}$ induces a continuous linear mapping of $L^{p,q}_{\vect s,0}(D)$ onto $\widehat A^{p,q}_{\vect s,\vect{s''},0}(D)$;\footnote{We define $\widehat A^{p,q}_{\vect s,\vect s'',0}(D)$ as the space of $f\in \Hol(D)$ such that $f*I^{-\vect{s''}}_\Omega\in \widetilde A^{p,q}_{\vect s+\vect s'',0}(D)$. We define $\widehat A^{p,q}_{\vect s,\vect s''}(D)$ analogously.}
		
		\item[\textnormal{(3)}]  $P_{\vect{s'}}$ induces a continuous linear mapping of $L^{p,q}_{\vect s}(D)$ onto $\widehat A^{p,q}_{\vect s,\vect{s''}}(D)$.
	\end{itemize}
\end{cor}

This extends~\cite[Theorem 1.8 (1)]{Bekolleetal2}. 
Observe that the assumption $\vect{s'}\prec \vect b+\vect d-\frac 1 2 \vect m$ cannot be replaced by the seemingly more natural assumption $\vect{s'}-\vect{s''}\prec \vect b+\vect d-\frac 1 2 \vect m$, since otherwise $P_{\vect{s'}}$ would not be well defined.

\begin{proof}
	It suffices to observe that there is a constant $c_{\vect{s'},\vect{s''}}\neq0$ such that $(P_{\vect{s'}}f)*I^{-\vect{s''}}_\Omega=c_{\vect{s'},\vect{s''}} P_{\vect{s'}-\vect{s''}}(f (\Delta_\Omega^{-\vect{s''}}\circ \rho))$ for every $f\in C_c(D)$, thanks to~\cite[Proposition 2.29]{CalziPeloso}, and to apply Theorem~\ref{teo:1bis}.
\end{proof}

When $\min(p,q)<1$ we do not know if the preceding equivalences still hold. Nonetheless, we have the following partial results.

\begin{cor}
	Take $p,q\in(0,\infty]$ and $\vect s\succ \frac 1 p (\vect b+\vect d)+\frac{1}{2 q'}\vect{m'}$.
	If $A^{p,q}_{\vect s,0}(D)=\widetilde A^{p,q}_{\vect s,0}(D)$ (resp.\ $A^{p,q}_{\vect s}(D)=\widetilde A^{p,q}_{\vect s}(D)$), then $A^{p'',q''}_{\vect s-(1/p-1)_+(\vect b+\vect d),0}(D)=\widetilde A^{p'',q''}_{\vect s-(1/p-1)_+(\vect b+\vect d),0}(D)$ (resp.\ $A^{p'',q''}_{\vect s-(1/p-1)_+(\vect b+\vect d)}(D)=\widetilde A^{p'',q''}_{\vect s-(1/p-1)_+(\vect b+\vect d)}(D)$). 
\end{cor}

\begin{proof}
	Observe first that~\cite[Corollary 5.16]{CalziPeloso} implies that there is $\vect s'\prec \vect b+\vect d-\frac 1 2 \vect m, \frac{1}{p} (\vect b+\vect d)- \frac{1}{2 q}\vect{m'}-\vect s$ such that strong property $(L')^{p,q}_{\vect s,\vect{s'},0}$ (resp.\ $(L')^{p,q}_{\vect s,\vect{s'}}$) holds, so that Proposition~\ref{prop:12} implies that $A^{p',q'}_{(\vect b+\vect d)/\min(1,p)-\vect s-\vect s'}(D)=\widetilde A^{p',q'}_{(\vect b+\vect d)/\min(1,p)-\vect s-\vect s'}(D)$. 
	Now, consider the sesquilinear form
	\[
	(f,g)\mapsto \int_D f \overline g (\Delta_\Omega^{-\vect s'}\circ \rho)\,\dd \nu_D,
	\]
	and observe that it induces an antilinear isomorphism $\iota_1\colon A^{p',q'}_{(\vect b+\vect d)/\min(1,p)-\vect s-\vect s'}(D)\to \widetilde A^{p,q}_{\vect s,0}(D)'$ and a continuous linear mapping $\iota_2 \colon A^{p',q'}_{(\vect b+\vect d)/\min(1,p)-\vect s-\vect s'}(D)\to V'$, where $V$ is the closure of $\widetilde A^{p,q}_{\vect s,0}(D)$ in $ A^{p'',q''}_{\vect s-(1/p-1)_+(\vect b+\vect d)}(D)$, thanks to by~\cite[Proposition 3.37]{CalziPeloso}. In addition, the inclusion $\iota_3\colon \widetilde A^{p,q}_{\vect s,0}(D)\to V$ is continuous by~\cite[Propositions 3.2 and 3.7]{CalziPeloso}, so that
	\[
	\iota_1=\trasp \iota_3\circ \iota_2 .
	\]
	Therefore,  $\trasp \iota_3$ is onto, hence an isomorphism, so that also $\iota_2$ is an isomorphism. To conclude, observe that $B^{\vect s'}_{(\zeta,z)}\in \widetilde A^{p,q}_{\vect s,0}(D)\subseteq V$ for every $(\zeta,z)\in D$, and argue by duality to show that weak property $(L)^{p'',q''}_{\vect s-(1/p-1)_+(\vect b+\vect d),\vect s'}$ holds (cf.~Proposition~\ref{prop:12}). The conclusion follows from Corollary~\ref{cor:1}.	
\end{proof}

\begin{cor}\label{cor:4}
	Take $p,q\in (0,\infty]$ and $\vect s,\vect{s'}\in \R^r$ such that 	weak property $(L)^{p,q}_{\vect s,\vect s',0}$ holds.
	Then, 
	\[
	A^{p'',q''}_{\vect s-(1/p-1)_+(\vect b+\vect d),0}(D)=\widetilde A^{p'',q''}_{\vect s-(1/p-1)_+(\vect b+\vect d),0}(D), \quad A^{p',q'}_{(\vect b+\vect d)/\min(1,p)-\vect s-\vect s'}(D)=\widetilde A^{p',q'}_{(\vect b+\vect d)/\min(1,p)-\vect s-\vect s'}(D),
	\] and the sesquilinear form
	\[
	(f,g)\mapsto \int_D f \overline g (\Delta_\Omega^{-\vect s'}\circ \rho) \,\dd \nu_D
	\]
	induces an antilinear isomorphism of $A^{p',q'}_{(\vect b+\vect d)/\min(1,p)-\vect s-\vect s'}(D)$ onto $V'$, where $V$ denotes the closed vector subspace of $A^{p,q}_{\vect s,0}(D)$ generated by the $B^{\vect s'}_{(\zeta,z)}$, $(\zeta,z)\in D$.
\end{cor}

This result  improves also~\cite[Proposition 3.37]{CalziPeloso}. Notice that weak property $(L)^{p,q}_{\vect s,\vect s',0}$ implies that $\vect s\succ \frac 1 p (\vect b+\vect d)+\frac{1}{2 q'}\vect{m'}$, that $\vect s+\vect{s'}\prec\frac 1 p (\vect b+\vect d)- \frac{1}{2 q}\vect{m'}$, and that $\vect s'\prec \vect b+\vect d-\frac 1 2 \vect m$, thanks to~\cite[Lemma 3.29]{CalziPeloso}.

\begin{proof}
	Take $\delta_0>0$ so that for every $(\delta,4)$-lattice $(\zeta_{j,k},z_{j,k})_{j\in J,k\in K}$ on $D$ the mapping
	\[
	S \colon f \mapsto (\Delta_\Omega^{(\vect b+\vect d)/p-\vect s-\vect s'}(h_k) f(\zeta_{j,k},z_{j,k}))
	\]
	induces an isomorphism of $A^{p',q'}_{(\vect b+\vect d)/\min(1,p)-\vect s-\vect s'}(D)$ onto a closed subspace of $\ell^{p',q'}(J,K)$, and $S^{-1}(\ell^{p',q'}(J,K))\cap A^{\infty,\infty}_{(\vect b+\vect d)/p-\vect s-\vect s'}(D)=A^{p',q'}_{(\vect b+\vect d)/\min(1,p)-\vect s-\vect s'}(D)$, where $h_k=\rho(\zeta_{j,k},z_{j,k}) $ for every $j\in J$ and for every $k\in K$.
	Observe that we may assume that $(\zeta_{j,k},z_{j,k})_{j\in J,k\in K}$  is chosen so that the mapping
	\[
	\Psi\colon\ell^{p,q}_0(J,K)\ni \lambda \mapsto \sum_{j,k} \lambda_{j,k} B^{\vect s'}_{(\zeta_{j,k},z_{j,k})} \Delta_\Omega^{(\vect b+\vect d)/p-\vect s-\vect s'}(h_k)\in V
	\]
	is well defined and continuous. If we denote by $\langle\,\cdot\,\vert \,\cdot\,\rangle$ the  sesquilinear form on $\widetilde A^{p,q}_{\vect s,0}(D)\times \widetilde A^{p',q'}_{(\vect b+\vect d)/\min(1,p)-\vect s-\vect s'}(D)$ which extends to the sesquilinear form in the statement  as in Proposition~\ref{prop:10}, then
	\[
	c_{\vect s'}\langle \Psi(\lambda) \vert f \rangle=\langle \lambda \vert S f\rangle
	\]
	for every $\lambda\in \C^{(J\times K)}$ and for every $f\in \widetilde A^{p',q'}_{(\vect b+\vect d)/\min(1,p)-\vect s-\vect s'}(D)$ by the proof of Proposition~\ref{prop:12}, where $c_{\vect s'}\neq 0$ is a suitable constant. In particular, this shows that $S$ maps $\widetilde A^{p',q'}_{(\vect b+\vect d)/\min(1,p)-\vect s-\vect s'}(D)$ into $\ell^{p',q'}(J,K)$ continuously. Since $\widetilde A^{p',q'}_{(\vect b+\vect d)/\min(1,p)-\vect s-\vect s'}(D)\subseteq A^{\infty,\infty}_{(\vect b+\vect d)/p-\vect s-\vect s'}(D)$ by Proposition~\ref{prop:1}, this implies that $\widetilde A^{p',q'}_{(\vect b+\vect d)/\min(1,p)-\vect s-\vect s'}(D)\subseteq  A^{p',q'}_{(\vect b+\vect d)/\min(1,p)-\vect s-\vect s'}(D)$, so that $\widetilde A^{p',q'}_{(\vect b+\vect d)/\min(1,p)-\vect s-\vect s'}(D)=A^{p',q'}_{(\vect b+\vect d)/\min(1,p)-\vect s-\vect s'}(D)$ by Proposition~\ref{prop:1}.
	By Theorem~\ref{teo:1bis}, this implies that property $(\widetilde L')^{p'',q''}_{\vect s-(1/p-1)_+(\vect b+\vect d),0}$ holds, so that, in particular, $V$ is contained and dense in $\widetilde A^{p'',q''}_{\vect s-(1/p-1)_+(\vect b+\vect d),0}(D)$. Since $\langle\,\cdot\,\vert \,\cdot\,\rangle$ induces an antilinear isomorphism of $\widetilde A^{p',q'}_{(\vect b+\vect d)/\min(1,p)-\vect s-\vect s'}(D)$ onto $ \widetilde A^{p'',q''}_{\vect s-(1/p-1)_+(\vect b+\vect d),0}(D)'$ by Proposition~\ref{prop:10}, this implies that the continuous antilinear map $\iota\colon A^{p',q'}_{(\vect b+\vect d)/\min(1,p)-\vect s-\vect s'}(D)\to V'$, induced by the sesquilinear form in the statement, is one-to-one. To prove that $\iota$ is onto, take $\lambda\in V'$, and define
	\[
	f\colon D\ni (\zeta,z)\mapsto c_{\vect s'}\overline{\langle \lambda, B_{(\zeta,z)}^{\vect s'}\rangle}\in \C.
	\]
	By~\cite[Lemma 3.38]{CalziPeloso}, $f\in A^{p',q'}_{(\vect b+\vect d)/\min(1,p)-\vect s-\vect s'}(D)$. 
	Since 
	\[
	\langle \lambda, B_{(\zeta,z)}^{\vect s'}\rangle=\langle B_{(\zeta,z)}^{\vect s'}\vert f\rangle
	\]
	for every $(\zeta,z)\in D$ by~\cite[Lemma 3.29 and Proposition 3.13]{CalziPeloso}, we see that $\lambda=\iota(f)$.	
	Finally, using sampling in $V$ (inherited by sampling in $A^{p,q}_{\vect s,0}(D)$, cf.~\cite[Theorem 3.22]{CalziPeloso}), we see that strong property $(L)^{p',q'}_{(\vect b+\vect d)/\min(1,p)-\vect s-\vect s'}$ holds (cf.~Proposition~\ref{prop:12}), so that $A^{p'',q''}_{\vect s-(1/p-1)_+(\vect b+\vect d),0}(D)=\widetilde A^{p'',q''}_{\vect s-(1/p-1)_+(\vect b+\vect d),0}(D)$ by Corollary~\ref{cor:1} and Proposition~\ref{prop:8}.
\end{proof}

\section{Interpolation}\label{sec:5}

We now prove a useful result concerning \emph{complex} interpolation of Bergman
spaces and their boundary value spaces.    Cf.~\cite{BekolleGonessaNana2} for more information on \emph{real} interpolation of Bergman spaces.  
\begin{teo}\label{prop:6}
	Take $p_0,p_1,q_0,q_1\in (0,\infty]$ and $\vect s_j\succ\frac{1}{p_j}(\vect b+\vect d)+\frac{1}{2 q_j'}\vect{m'}$ for $j=0,1$. 
	If $\widetilde A^{p_j,q_j}_{\vect s_j,0}(D)=A^{p_j,q_j}_{\vect s_j,0}(D)$ (resp.\  $\widetilde A^{p_j,q_j}_{\vect s_j}(D)=A^{p_j,q_j}_{\vect s_j}(D)$) for $j=0,1$, then
	\[
	\widetilde A^{p_\theta,q_\theta}_{\vect s_\theta,0}(D)=A^{p_\theta,q_\theta}_{\vect s_\theta,0}(D)  \qquad \text{(resp.\ } \widetilde A^{p_\theta,q_\theta}_{\vect s_\theta}(D)=A^{p_\theta,q_\theta}_{\vect s_\theta}(D)\text{)}
	\]
	for every $\theta\in (0,1)$, where
	\[
	\frac{1}{p_\theta}=\frac{1-\theta}{p_0}+\frac{\theta}{p_1}, \qquad \frac{1}{q_\theta}=\frac{1-\theta}{q_0}+\frac{\theta}{q_1}, \qquad \text{and} \qquad \vect s_\theta=(1-\theta) \vect s_0+\theta \vect s_1.  
	\]
\end{teo}

\begin{proof}
	Define $S\coloneqq \Set{w\in \C\colon 0<\Re z<1}$. Take $\theta\in (0,1)$ and $u\in \mathring B^{-\vect s_\theta}_{p_\theta,q_\theta}(\Nc,\Omega)$ (resp.\ $u\in B^{-\vect s_\theta}_{p_\theta,q_\theta}(\Nc,\Omega)$).
	By~\cite[Theorem 6.6]{Besov}, there is a bounded continuous function 
	\[
	f\colon \overline S\to \Sc_{\Omega,L}'(\Nc)
	\]
	which is holomorphic in $S$, equals $u$ at $\theta$, and maps $j+i R$ boundedly into $\mathring B^{-\vect s_j}_{p_j,q_j}(\Nc,\Omega)$ (resp.\ $ B^{-\vect s_j}_{p_j,q_j}(\Nc,\Omega)$) for $j=0,1$; here, $\overline S$ denotes the closure of $S$ in $\C$.
	
	Take $(\Psi_j)$ as in Lemma~\ref{lem:3}, and observe that~\cite[2.4.6/2]{Triebel2} implies that there are two probability measures $\mi_0,\mi_1$ on $\R$ such that
	\[
	\abs{g(\theta)}^\ell\meg \left(\int_\R \abs{g(i t)}^\ell \,\dd \mi_0(t)\right)^{1-\theta} \left(\int_\R \abs{g(1+it)}^\ell\,\dd \mi_1(t) \right)^\theta
	\]
	for every bounded uniformly continuous function $g\colon \overline S\to \C$ which is holomorphic on $S$, where $\ell\coloneqq \min(p_0,p_1,q_0,q_1)$. Therefore,
	\[
	\begin{split}
	\ee^{\eps(\theta^2-\theta)}\abs{\langle u, \Psi_j (S_{(\zeta,z)})_0\rangle }&\meg \left(\int_\R \abs{\ee^{(i t)^2-(it)}\langle f(i t),  \Psi_j (S_{(\zeta,z)})_0\rangle}^\ell \,\dd \mi_0(t)\right)^{1-\theta}\\
		&\quad \times \left(\int_\R\abs{\ee^{(1+i t)^2-(1+it)}\langle f(1+i t),  \Psi_j (S_{(\zeta,z)})_0\rangle}^\ell\,\dd \mi_1(t) \right)^\theta
	\end{split}
	\]
	for every $(\zeta,z)\in D$, for every $j\in \N$, and for every $\eps>0$. Passing to the limit for $\eps\to 0^+$, and $j\to \infty$, this implies that
	\[
	\abs{(\Ec u)(\zeta,z) }\meg \left(\int_\R \abs{(\Ec f(i t))(\zeta,z)}^\ell \,\dd \mi_0(t)\right)^{1-\theta} \left(\int_\R\abs{(\Ec f(1+i t))(\zeta,z)}^\ell\,\dd \mi_1(t) \right)^\theta
	\]
	for every $(\zeta,z)\in D$. By repeated applications of H\"older's and Minkowski's integral inequalities (cf., e.g.,~the proof of~\cite[Theorem 6.6]{Besov}), this implies that
	\[
	\norm{\Ec u}_{A^{p_\theta,q_\theta}_{\vect s_\theta}(D)} \meg \left( \int_\R \norm{(\Ec f)(i t)}_{A^{p_0,q_0}_{\vect s_0}(D)}^\ell  \,\dd \mi_0(t) \right)^{1-\theta}\left( \int_\R \norm{(\Ec f)(1+i t)}_{A^{p_1,q_1}_{\vect s_1}(D)}^\ell  \,\dd \mi_1(t) \right)^{\theta}.
	\]
	Since, by assumption, $\Ec$ induces an isomorphism of $\mathring B^{-\vect s_j}_{p_j,q_j}(\Nc,\Omega)$ onto $A^{p_j,q_j}_{\vect s_j,0}(D)$ (resp.\ of $B^{-\vect s_j}_{p_j,q_j}(\Nc,\Omega)$ onto $A^{p_j,q_j}_{\vect s_j}(D)$) for $j=0,1$, this implies that there is a constant $C>0$ such that
	\[
	\norm{\Ec u}_{A^{p_\theta,q_\theta}_{\vect s_\theta}(D)} \meg C \sup_{j=0,1}\sup_{t\in \R} \norm{f(j+it)}_{B^{-\vect s_j}_{p_j,q_j}(\Nc,\Omega)}
	\]
	once a quasi-norm on $B^{-\vect s_j}_{p_j,q_j}(\Nc,\Omega)$ is fixed for $j=0,1$. By~\cite[Theorem 6.6]{Besov} and the arbitrariness of $f$, this implies that $\Ec$ maps $\mathring B^{-\vect s_\theta}_{p_\theta,q_\theta}(\Nc,\Omega)$ (resp.\ $B^{-\vect s_\theta}_{p_\theta,q_\theta}(\Nc,\Omega)$) continuously into $A^{p_\theta,q_\theta}_{\vect s_\theta}(D)$. Since $\Sc_{\Omega,L}(\Nc)$ is dense in $\mathring B^{-\vect s_\theta}_{p_\theta,q_\theta}(\Nc,\Omega)$ and $\Ec(\Sc_{\Omega,L}(\Nc))\subseteq A^{p_\theta,q_\theta}_{\vect s_\theta,0}(D)$ by Proposition~\ref{prop:1}, the assertion follows.
\end{proof}

\begin{cor}\label{cor:2}
	Take $p,q\in (0,\infty]$ and $\vect s\succ \frac{1}{p}(\vect b+\vect d)+\frac{1}{2 q'}\vect{m'}$. If $\widetilde A^{p,q}_{\vect s,0}(D)=A^{p,q}_{\vect s,0}(D)$ (resp.\ $\widetilde A^{p,q}_{\vect s}(D)=A^{p,q}_{\vect s}(D)$), then  $\widetilde A^{p,q}_{\vect {s'},0}(D)=A^{p,q}_{\vect{s'},0}(D)$ (resp.\ $\widetilde A^{p,q}_{\vect {s'}}(D)=A^{p,q}_{\vect{s'}}(D)$) for every $\vect{s'}\succeq \vect s$.
\end{cor}

\begin{proof}
	This follows from Theorem~\ref{prop:6}, since  $\widetilde A^{p,q}_{\vect {s''},0}(D)=A^{p,q}_{\vect{s''},0}(D)$ (resp.\ $\widetilde A^{p,q}_{\vect {s''}}(D)=A^{p,q}_{\vect{s''}}(D)$) for $\vect{s''}=\vect s+\frac 1 \theta (\vect{s'}-\vect s)$ and $\theta\in (0,1)$ sufficiently close to $0$, thanks to~\cite[Corollary 5.11]{CalziPeloso}.
\end{proof}

With similar techniques, one may prove that also strong properties $\atomic^{p,q}_{\vect s,\vect{s'},0}$ and $\atomic^{p,q}_{\vect s,\vect{s'}}$ interpolate (\emph{for fixed $\vect s'$}). For $p,q\Meg 1$ and for Bergman projectors, this is a consequence of Corollary~\ref{cor:1}.

\begin{cor}
	Take $p_0,p_1,q_0,q_1\in [1,\infty]$, $\vect s_0,\vect s_1\in \R^r$, and $\vect{s}'_0,\vect{s}'_1\prec \vect b+\vect d-\frac 1 2 \vect{m'}$. If $P_{\vect{s}'_j}$ induces an endomorphism of $L^{p_j,q_j}_{\vect s_j,0}(D)$ (resp.\ $L^{p_j,q_j}_{\vect s_j}(D)$) for $j=0,1$, then $P_{\vect{s}'_\theta}$ induces an endomorphism of $L^{p_\theta,q_\theta}_{\vect s_\theta,0}(D)$ (resp.\ $L^{p_\theta,q_\theta}_{\vect s_\theta}(D)$) for every $\theta\in (0,1)$, where
	\[
	\frac{1}{p_\theta}=\frac{1-\theta}{p_0}+\frac{\theta}{p_1}, \qquad \frac{1}{q_\theta}=\frac{1-\theta}{q_0}+\frac{\theta}{q_1}, \qquad \vect s_\theta=(1-\theta) \vect s_0+\theta \vect s_1, \qquad \text{and} \qquad \vect {s}'_\theta=(1-\theta) \vect{s}'_0+\theta \vect{s}'_1.
	\]
\end{cor}

\begin{proof}
	This follows from Corollary~\ref{cor:1}, Theorem~\ref{prop:6}, and~\cite[Proposition 5.20]{CalziPeloso}.
\end{proof}

\begin{cor}
	Take $p,q\in (0,\infty]$ and $\vect s\in \R^r$ and $\vect s'\prec \vect b+\vect d-\frac{1}{2 }\vect{m}$.
	$P_{\vect s'}$ induces an endomorphism of $L^{p,q}_{\vect s,0}(D)$ (resp.\ $L^{p,q}_{\vect s}(D)$), then $P_{\vect s'''}$ induces an endomorphism of $L^{p,q}_{\vect s'',0}(D)$ (resp.\ $L^{p,q}_{\vect s''}(D)$) for every $\vect{s''}\succeq \vect s$ and for every $\vect s'''\preceq\vect s + \vect s'-\vect s''$.
\end{cor}

\begin{proof}
	This follows from  Corollaries~\ref{cor:1} and~\ref{cor:2}, and~\cite[Proposition 5.20]{CalziPeloso}.
\end{proof}

\section{Transference}\label{sec:6}

 In this section we  prove a transference result, Corollary~\ref{trasference:cor}, showing that if the Bergman
projector $P_{\vect s'}$ is bounded on the pure norm spaces
$L^{p,p}_{\vect s}(F+i\Omega)$ on the tube domain $F+i\Omega$,
where $\Omega$ is a homegeneous cone
then the  Bergman
projector $P_{\vect s'+\vect b}$ is bounded on the pure norm spaces
$L^{p,p}_{\vect s}(D)$ on the corresponding homogeneous Siegel domain
of Type II.   We also prove a converse result, 
Corollary~\ref{transference:cor2}, showing that
if $P_{\vect s'}$ is bounded on the mixed norm spaces
$L^{p,q}_{\vect s}(D)$, then $P_{\vect s'}$ is bounded on the mixed
norm spaces $L^{p,q}_{\vect s-\vect b/p}(F+i\Omega)$ on the tube
domain $F+i\Omega$. 

  Notice that the analogous transference result~\cite[Theorem 2.1]{BekolleGonessaNana} holds even for $p\neq q$ since the definition of $A^{p,q}_{\vect s}(D)$ considered therein is essentially different from ours when $p\neq q$ and $n>0$.  

\begin{teo}\label{prop:2}
	Take $p\in (0,\infty]$ and $\vect s\succ \frac1 p\vect d+\frac{1}{2 p'}\vect{m'}  +(\R_+^*)^r$. If $A^{p,p}_{\vect s,0}(F+i\Omega)=\widetilde A^{p,p}_{\vect s,0}(F+i\Omega)$ (resp.\  $A^{p,p}_{\vect s}(F+i\Omega)=\widetilde A^{p,p}_{\vect s}(F+i\Omega)$), then $A^{p,p}_{\vect s,0}(D)=\widetilde A^{p,p}_{\vect s,0}(D)$ ($A^{p,p}_{\vect s}(D)=\widetilde A^{p,p}_{\vect s}(D)$).
\end{teo}

\begin{proof}
	Take $u\in \mathring B^{-\vect s}_{p,q}(\Nc,\Omega)$ (resp.\ $u\in B^{-\vect s}_{p,q}(\Nc,\Omega)$). Take a $(\delta,R)$-lattice $(\lambda_k)_{k\in K}$ on $\Omega'$ for some $\delta>0$ and some $R>1$, and fix a positive $\varphi \in C^\infty_c(\Omega')$ such that 
	\[
	\sum_{k\in K} \varphi(\,\cdot\, t_k^{-1})\Meg 1
	\]
	on $\Omega'$, where $t_k\in T_+$ is chosen such that $\lambda_k=e_{\Omega'}\cdot t_k$. Define $\psi_k\coloneqq \Fc_{\Nc}^{-1}(\varphi(\,\cdot\, t_k^{-1}))$ and $\psi'_k\coloneqq \Fc_F^{-1}(\varphi(\,\cdot\, t_k^{-1}))$. Let us prove that
	\[
	[(\Ec u)_h*\psi_k](\zeta,\,\cdot\,)=[(\Ec u)_h(\zeta,\,\cdot\,)]*\psi'_k
	\]
	for every $\zeta\in E$ and for every $k\in K$. Assume first that $u\in \Sc_\Omega(\Nc)$, and observe that by~\eqref{eq:2} there is a constant $c>0$ such that
	\[
	\Fc_F( [(\Ec u)_h*\psi_k](\zeta,\,\cdot\,))= c \Delta^{-\vect b}_{\Omega'} \ee^{-\langle\,\cdot\,,\Phi(\zeta)\rangle} \Fc_\Nc (\Ec u)_h \varphi(\,\cdot\, t_k^{-1})=\Fc_F((\Ec u)_h(\zeta,\,\cdot\,))\varphi(\,\cdot\,t_k^{-1})= \Fc_F([(\Ec u)_h(\zeta,\,\cdot\,)]*\psi'_k)
	\]
	for every $\zeta\in E$ and for every $k\in K$. The assertion then follows when $u$ is a finite sum of left translates of elements of $\Sc_\Omega(\Nc)$. Since this set is dense in $\Sc_{\Omega,L}(\Nc)$ (cf.~\cite[Proposition 4.5]{CalziPeloso}), hence in $B^{-\vect s}_{p,p}(\Nc,\Omega)$ for the weak topology $\sigma^{-\vect s}_{p,p}$, the assertion follows by continuity.
	
	In addition, by assumption there is a constant $C>0$ such that
	\[
	\norm{\Delta_{\Omega'}^{\vect s}(\lambda_k) \norm{[(\Ec u)_h(\zeta,\,\cdot\,)]*\psi'_k}_{L^p(F)}}_{\ell^p(K)}\Meg C\norm{(\Ec u)(\zeta,\,\cdot\,+i\Phi(\zeta)+i h)}_{A^{p,p}_{\vect s}(F+i\Omega)}
	\]
	for every $u$ as above, for every $h\in \Omega$, and for every $\zeta\in E$. Taking the $L^p(E)$-norm on both sides, this gives
	\[
	\norm{\Delta_{\Omega'}^{\vect s}(\lambda_k) \norm{[(\Ec u)_h]*\psi_k}_{L^p(\Nc)}}_{\ell^p(K)}\Meg C \norm{(\Ec u)(\,\cdot\,+(0,ih))}_{L^{p,p}_{\vect s}(D)}
	\]
	for every $u$ as above, and for every $h\in \Omega$. Since 
	\[
	\norm{u*\psi_k}_{L^p(\Nc)}\Meg \norm{[\Ec(u*\psi_k)]_h}_{L^p(\Nc)} =\norm{(\Ec u)_h*\psi_k}_{L^p(\Nc)}
	\]
	for every $h\in \Omega$, for every $k\in K$, and for every $u$ as above by~\cite[Theorem 1.7]{CalziPeloso2}, passing to the limit for $h\to 0$, $h\in\Omega$, by lower semi-continuity we infer that
	\[
	\norm{\Delta_{\Omega'}^{\vect s}(\lambda_k) \norm{ u*\psi_k}_{L^p(\Nc)}}_{\ell^p(K)}\Meg C \norm{\Ec u}_{L^{p,p}_{\vect s}(D)},
	\] 
	so that  $\widetilde A^{p,p}_{\vect s,0}(D) $ (resp.\ $\widetilde A^{p,p}_{\vect s}(D)$) embeds continuously into $A^{p,p}_{\vect s}(D)$ by the arbitrariness of $u$. Since $\Ec(\Sc_{\Omega,L}(\Nc))$ is dense in  $\widetilde A^{p,p}_{\vect s,0}(D) $ and contained in $A^{p,p}_{\vect s,0}(D) $ by Proposition~\ref{prop:1}, the assertion follows.
\end{proof}

\begin{cor}\label{trasference:cor}
	Take $p\in [1,\infty]$, $\vect s\in \R^r$,  and $\vect {s'}\prec \vect d-\frac 1 2 \vect{m}$. If $P_{\vect{s'}}$ induces a continuous linear projector of $L^{p,p}_{\vect s,0}(F+i\Omega)$ onto $A^{p,p}_{\vect s,0}(F+i\Omega)$ (resp.\ of $L^{p,p}_{\vect s}(F+i\Omega)$ onto $A^{p,p}_{\vect s}(F+i\Omega)$), then $P_{\vect{s'}+\vect b}$ induces a continuous linear projector of $L^{p,p}_{\vect s,0}(D)$ onto $A^{p,p}_{\vect s,0}(D)$ (resp.\ of $L^{p,p}_{\vect s}(D)$ onto $A^{p,p}_{\vect s}(D)$).
\end{cor}

Analogous results hold for atomic decomposition, provided that $p\in [1,\infty]$. When $\vect s'=\vect b+\vect d-p \vect s$ and $p<\infty$, this is a consequence of~\cite[Theorem 2.1]{BekolleGonessaNana}.

\begin{proof}
	The assertion follows from Corollary~\ref{cor:1} and Theorem~\ref{prop:2}.
\end{proof}

\begin{teo}\label{prop:3}
	Take $p,q\in (0,\infty]$ and $\vect s\succ  \frac1 p (\vect b+\vect d)+\frac{1}{2 q'}\vect{m'}$. If $A^{p,q}_{\vect s,0}(D)=\widetilde A^{p,q}_{\vect s,0}(D)$ (resp.\ $A^{p,q}_{\vect s}(D)=\widetilde A^{p,q}_{\vect s}(D)$), then  $A^{p,q}_{\vect s-\vect b/p,0}(F+i\Omega)=\widetilde A^{p,q}_{\vect s-\vect b/p,0}(F+i\Omega)$ (resp.\  $A^{p,q}_{\vect s-\vect b/p}(F+i\Omega)=\widetilde A^{p,q}_{\vect s-\vect b/p}(F+i\Omega)$).
\end{teo}
  
We shall postpone the proof of Theorem~\ref{prop:3} until the end of this section. Here is the corollary on the associated transference result for Bergman projectors.
 
\begin{cor}\label{transference:cor2}
	Take $p,q\in [1,\infty]$, $\vect s\in \R^r$, and $\vect {s'}\prec\vect b+\vect d-\frac 1 2 \vect{m}$. If $P_{\vect{s'}}$ induces a continuous linear projector of $L^{p,q}_{\vect s,0}(D)$ onto $A^{p,q}_{\vect s,0}(D)$ (resp.\ of $L^{p,q}_{\vect s}(D)$ onto $A^{p,q}_{\vect s}(D)$), then $P_{\vect{s'}}$ induces a continuous linear projector of $L^{p,q}_{\vect s-\vect b/p,0}(F+i \Omega)$ onto $A^{p,q}_{\vect s-\vect b/p,0}(F+i \Omega)$ (resp.\ of $L^{p,q}_{\vect s-\vect b/p}(F+i \Omega)$ onto $A^{p,q}_{\vect s-\vect b/p}(F+i \Omega)$).
\end{cor}

\begin{proof}
	The assertion follows from Theorem~\ref{prop:3}, Corollary~\ref{cor:1}, and~\cite[Proposition 5.20]{CalziPeloso}.
\end{proof}

Before we pass to the proof  of Theorem~\ref{prop:3}, we need some auxiliary results.
 
\begin{prop}\label{prop:5}
	Take $p,q\in (0,\infty]$ and $\vect s\in \R^r$. Then, there is a constant $C>0$ such that
	\[
	\norm{f(\zeta,\,\cdot\,+i \Phi(\zeta))}_{A^{p,q}_{\vect s-\vect b/p}(F+i\Omega)}\meg C \norm{f}_{A^{p,q}_{\vect s}(D)}
	\]
	for every $\zeta\in E$ and for every  $f\in A^{p,q}_{\vect s}(D)$. If, in addition, $f\in A^{p,q}_{\vect s,0}(D)$, then $f(\zeta,\,\cdot\,+i \Phi(\zeta))\in A^{p,q}_{\vect s,0}(F+i \Omega)$ for every $\zeta\in E$.
	
	Conversely, define $\iota\colon \Hol(F+i \Omega)\to \Hol(D)$ so that $\iota(f)(\zeta,z)=f(z)$ for every $(\zeta,z)\in D$. Assume that $\vect s \succ \frac 1 {2p} \vect m$ (resp.\ $\vect s\Meg \vect 0$ if $p=\infty$). 
	Then, there is a constant $C'>0$ such that
	\[
	\norm{\iota(f)}_{A^{p,p}_{\vect s}(D)}= C'\norm{f}_{A^{p,p}_{\vect s-\vect b/p}(F+i\Omega)}
	\]
	for every $f\in A^{p,p}_{\vect s-\vect b/p,0}(F+i\Omega)$ (resp.\ $f\in A^{p,p}_{\vect s-\vect b/p}(F+i\Omega)$). In addition, $\iota(f)\in A^{p,p}_{\vect s,0}(D)$ if $f\in A^{p,p}_{\vect s-\vect b/p,0}(F+i \Omega)$.
\end{prop}

\begin{proof}
	By translation invariance, it will suffice to prove the assertion for $\zeta=0$.
	Define $\ell\coloneqq \min (1,p,q)$ to simplify the notation, and set
	\[
	\varphi\colon E\times F\times \Omega\ni (\zeta,x,h)\mapsto (\zeta, x+i\Phi(\zeta)+i h)\in D,
	\]
	and observe that $\varphi$ is a bijection of $E\times F\times \Omega$ onto $D$.
	Observe that there are $R_0>0$ and $C'>0$ such that, for every $R\in (0,R_0]$ and for every $h\in \Omega$,
	\[
	B_{E\times F_\C}((0,i h),R/C')\subseteq \varphi(B_E(0,R)\times B_F(0,R)\times B_F(h,R))\subseteq B_{E\times F_\C}((0,i h),C' R).
	\]
	Therefore,~\cite[Lemma 1.24]{CalziPeloso} implies that there is a constant $C_{R_0}>0$ such that
	\[
	\abs{f(0,i h)}^{\ell}\meg C_{R_0} \dashint_{B_F(h,R)}\dashint_{B_E(0,R)} \dashint_{B_F(0,R)} \abs{f_{h'}(\zeta',x')}^{\ell}\,\dd x' \,\dd\zeta'\,\dd h'
	\]
	for every $h\in \Omega$ and for every $R\in (0,R_0]$ such that $\overline B_{E\times F_\C}((0,i h),C'R)\subseteq D$.
	Hence, 
	\[
	\abs{f(0,x+i h)}^{\ell}\meg C_{R_0} \dashint_{B_F(h,R)}\dashint_{B_E(0,R)} \dashint_{B_F(x,R)} \abs{f_{h'}(\zeta',x')}^{\ell}\,\dd x' \,\dd\zeta'\,\dd h'
	\]
	for every $x+i h\in F+i \Omega$ and for every $R\in (0,R_0]$ such that $\overline B_{E\times F_\C}((0,i h),C' R)\subseteq D$.
	Thus, applying Minkowski's  integral inequality, the  invariance of Lebesgue measure and Jensen's inequality,
	\[
	\begin{split}
		\norm{f(0,\,\cdot\,+i h)}_{L^{p}(F)}^{\min(1,q)}&=\norm*{\abs{f(0,\,\cdot\,+i h)}^{\ell}  }_{L^{p/\ell}(F)}^{ \min(1,q)/\ell}\\
		&\meg \left(  C_{R_0}\dashint_{B_F(h,R)}\dashint_{B_E(0, R)}\norm{f_{h'}(\zeta',\,\cdot\,)}_{L^{p}(\Nc)}^{\ell}\,\dd \zeta'\,\dd h'\right)^{ \min(1,q)/\ell} \\
		&\meg  C_{R_0}^{ \min(1,q)/\ell}\dashint_{B_F(h,R)}\left( \dashint_{B_E(0, R)}\norm{f_{h'}(\zeta',\,\cdot\,)}_{L^{p}(\Nc)}^{p}\,\dd \zeta'\right)^{\min(1,q)/p} \,\dd h'\\
		&\meg \frac{C_{R_0}^{ \min(1,q)/\ell}}{\Hc^{2n}(B_E(0,R))^{\min(1,q)/p}} \dashint_{B_F(h,R)}\norm{f_{h'}}_{L^{p}(\Nc)}^{\min(1,q)}\,\dd h'.
	\end{split}
	\]
	
	Then, there are $R>0$ such that $B_F(e_\Omega, R')\subseteq B_\Omega(e_\Omega,R)$ for some $R'\in (0, R_0]$, and a constant  $C'_{p,q,R}>0$ such that
	\[
	\begin{split}
		\norm{f(0,\,\cdot\,+i e_\Omega)}_{L^{p}(F)}^{\min(1,q)}&\meg C'_{p,q,R}\dashint_{B_\Omega(e_\Omega,R)}\left( \Delta_\Omega^{\vect s}(h')\norm{f_{h'}}_{L^{p}(\Nc)}\right) ^{\min(1,q)}\,\dd \nu_\Omega(h'),
	\end{split}
	\]
	so that, by homogeneity,
	\[
	\left( \Delta_\Omega^{\vect s-\vect b/p}(h)\norm{f(0,\,\cdot\,+i h)}_{L^{p}(F)}\right) ^{\min(1,q)}\meg C'_{p,q,R}\dashint_{B_\Omega(h,R)}\left( \Delta_\Omega^{\vect s}(h')\norm{f_{h'}}_{L^{p}(\Nc)}\right) ^{\min(1,q)}\,\dd \nu_\Omega(h'),
	\]
	for every $h\in \Omega$.
	Hence, by Jensen's inequality
	\[
	\norm{f(0,\,\cdot\,)}_{A^{p,q}_{\vect s-\vect b/p}(F+i\Omega)}\meg C_{p,q,R}'^{1/\min(1,q)} \norm{f}_{A^{p,q}_{\vect s}(D)}
	\]
	for every $f\in A^{p,q}_{\vect s}(D)$. The first assertion follows.
	In particular, the mapping $f\mapsto f(0,\,\cdot\,)$ is continuous from $A^{p,q}_{\vect s}(D)$ into $A^{p,q}_{\vect s-\vect b/p}(F+i\Omega)$, and maps $A^{p,q}_{\vect s}(D)\cap A^{\min(1,p),\min(1,q)}_{\vect s}(D)$ into $A^{p,q}_{\vect s-\vect b/p}(F+i\Omega)\cap A^{\min(1,p),\min(1,q)}_{\vect s-\vect b/\min(1,p)}(F+i\Omega)\subseteq A^{p,q}_{\vect s-\vect b/p,0}(F+i\Omega)$ (cf.~\cite[Proposition 3.7]{CalziPeloso}). Since $A^{p,q}_{\vect s,0}(D)\cap A^{\min(1,p),\min(1,q)}_{\vect s}(D)$ is dense in $A^{p,q}_{\vect s,0}(D)$ by~\cite[Proposition 3.9]{CalziPeloso}, the second assertion follows.
	
	Now, take $f\in \Hol(F+i \Omega)$ and assume that $\vect s \succ \frac{1}{2 p} \vect m$ (resp.\ $\vect s\Meg \vect 0$ if $p=\infty$). Then,~\cite[Proposition 2.30]{CalziPeloso} implies that there is a constant $c>0$ such that
	\[
	\begin{split}
		\norm{\iota(f)_h}_{L^p(\Nc)}&=\norm*{\zeta \mapsto \norm{f_{h+ \Phi(\zeta)}}_{L^p(F)}}_{L^p(E)}\\
		&=c^{1/p} \norm*{h' \mapsto \norm{f_{h+h'}}_{L^p(F)}}_{L^p(I^{-\vect b}_\Omega)},
	\end{split}
	\]
	so that
	\[
	\norm{\iota(f)}_{A^{p,p}_{\vect s}(D)}= \left(c\Gamma_\Omega(p\vect s) \right)^{1/p}\norm*{h \mapsto \norm{f_{h}}_{L^p(F)}}_{L^p(I^{p \vect s}_\Omega*I^{-\vect b}_\Omega)}
	\]
	when $p<\infty$ and
	\[
	\norm{\iota(f)}_{A^{\infty,\infty}_{\vect s}(D)}= \norm*{h \mapsto \Delta_\Omega^{\vect s}(h)\norm{f_{h}}_{L^\infty(F)}}_{L^\infty(\Omega)}
	\]
	otherwise. The third assertion follows from~\cite[Proposition 2.28]{CalziPeloso}, while the last one need only be verified when $p=\infty$, in which case it is clear.
\end{proof}

\begin{lem}	\label{lem:4}
	Take $\psi\in\Sc_\Omega(F)$  and $p\in (0,\infty]$. Then, there is a constant $C>0$ such that
	\[
	\frac 1 C \norm{u}_{L^p(F)}\meg  \norm{ \iota'(u)}_{L^p(\Nc)}\meg C \norm{u}_{L^p(F)}
	\]
	for every $u\in \Sc'_{\Omega,L}(F)$ such that $u=u*\psi\in L^p(F)$, where $\iota'(u)(\zeta,x)\coloneqq (\Ec u)(x+i\Phi(\zeta))$ for every $(\zeta,x)\in \Nc$.
\end{lem}

\begin{proof}
	Take a compact convex subset $K$ of $\Omega'$ such that $\Supp{\Fc_\Nc \psi}\subseteq K$, and observe that, if $u\in \Sc'_{\Omega,L}(F)$ and $u=u*\psi$, then $\Ec u$ is a well-defined entire function on $F_\C$ and 
	\[
	\norm{(\Ec u)_h}_{L^p(F)}\meg \norm{(\Ec u)_{h'}}_{L^p(F)} \ee^{H_{K}(h-h')}
	\]
	for every $h,h'\in F$, thanks to~\cite[Theorems 1.7 and 1.10]{CalziPeloso2}, where
	\[
	H_K(h)\coloneqq \sup_{\lambda\in- K} \langle \lambda,h\rangle
	\]
	for every $h\in F$. Therefore, 
	\[
	\begin{split}
		\norm{\iota'(u)}_{L^p(\Nc)}&=\norm*{\zeta \mapsto \norm{(\Ec u)_{\Phi(\zeta)}}_{L^p(F)}}_{L^p(E)}\meg \norm{u}_{L^p(F)}\norm{\zeta\mapsto \ee^{H_K(\Phi(\zeta))}}_{L^p(E)}
	\end{split}
	\]
	and, analogously,
	\[
	\begin{split}
		\norm{\iota'(u)}_{L^p(\Nc)}&=\norm*{\zeta \mapsto \norm{(\Ec u)_{\Phi(\zeta)}}_{L^p(F)}}_{L^p(E)}\Meg \norm{u}_{L^p(F)}\norm{\zeta\mapsto \ee^{-H_K(-\Phi(\zeta))}}_{L^p(E)}.
	\end{split}
	\]
	Now, observe that, since $K$ is a compact subset of $\Omega'$, there is a constant $C_1>0$ such that
	\[
	\langle \lambda, h\rangle \Meg C_1\abs{h}
	\]
	for every $\lambda\in K$ and for every $h\in \overline\Omega$, so that
	\[
	-H_K(-\Phi(\zeta))\meg H_K(\Phi(\zeta))\meg -C_1\abs{\Phi(\zeta)} 
	\]
	for every $\zeta\in E$. Since $\Phi$ is proper and absolutely homogeneous of degree $2$, there is a constant $C_2>0$ such that $\abs{\Phi(\zeta)}\Meg C_2 \abs{\zeta}^2$, so that the assertion follows.
\end{proof}

\begin{prop}\label{prop:4}
	Take $p,q\in (0,\infty]$ and  $\vect s\succ \frac 1 p (\vect b+\vect d)+\frac{1}{2 q'}\vect{m'}$. Then, the mapping
	\[
	\iota\colon \Hol(F+i\Omega)\ni f \mapsto [(\zeta,z)\mapsto f(z+i\Phi(\zeta))  ]\in \Hol(D) 
	\]
	induces isomorphisms of $\widetilde A^{p,q}_{\vect s-\vect b/p,0}(F+i\Omega)$ and $\widetilde A^{p,q}_{\vect s-\vect b/p}(F+i\Omega)$ onto closed subspaces of $\widetilde A^{p,q}_{\vect s,0}(D)$ and $\widetilde A^{p,q}_{\vect s}(D)$, respectively.
\end{prop}

\begin{proof}
	Since $\iota(\Ec(\Sc_\Omega(F)))=\Ec(\Sc_\Omega(\Nc))$, it will suffice to prove that $\iota$ induces an isomorphism of $\widetilde A^{p,q}_{\vect s-\vect b/p}(F+i\Omega)$ onto a closed subspace of $\widetilde A^{p,q}_{\vect s}(D)$.
	Take a $(\delta,R)$-lattice $(\lambda_k)_{k\in K}$ on $\Omega'$ for some $\delta>0$ and some $R>1$, and fix a positive $\varphi \in C^\infty_c(\Omega')$ such that 
	\[
	\sum_{k\in K} \varphi(\,\cdot\, t_k^{-1})\Meg 1
	\]
	on $\Omega'$, where $t_k\in T_+$ is chosen such that $\lambda_k=e_{\Omega'}\cdot t_k$. Define $\psi_k\coloneqq \Fc_{F}^{-1}(\varphi(\,\cdot\, t_k^{-1}))$, and observe that Lemma~\ref{lem:4} and a homogeneity argument show that there is a constant $C>0$ such that
	\[
	\frac1C\Delta_{\Omega'}^{\vect b/p}(\lambda_k)\norm{u*\psi_k}_{L^p(F)}\meg \norm{\iota'(u*\psi_k)}_{L^p(\Nc)}\meg C\Delta_{\Omega'}^{\vect b/p}(\lambda_k) \norm{u*\psi_k}_{L^p(F)}
	\]
	for every $u\in B^{-\vect s+\vect b/p}_{p,q}(F,\Omega)$ and for every $k\in \N$, where $\iota'(u*\psi_k)(\zeta,x)=\Ec (u*\psi_k)(x+i\Phi(\zeta))$ for every $(\zeta,x)\in \Nc$. 
	Therefore, $\iota'$ induces an isomorphism of $B^{-\vect s+\vect b/p}_{p,q}(F,\Omega)$ onto a closed subspace of $B^{-\vect s}_{p,q}(\Nc,\Omega)$. Since clearly
	\[
	\Ec\circ \iota'=\iota\circ \Ec,
	\]
	this implies that $\iota$ induces an isomorphism of $\widetilde A^{p,q}_{\vect s-\vect b/p}(F+i\Omega)$ onto a closed subspace of $\widetilde A^{p,q}_{\vect s}(D)$.
\end{proof}

\begin{proof}[Proof of Theorem~\ref{prop:3}.]
	Take $f\in \widetilde A^{p,q}_{\vect s-\vect b/p,0}(F+i\Omega)$ (resp.\ $f\in \widetilde A^{p,q}_{\vect s-\vect b/p}(F+i\Omega)$). Then, $\iota(f)\in \widetilde A^{p,q}_{\vect s,0}(D)$ (resp.\ $\iota(f)\in \widetilde A^{p,q}_{\vect s}(D)$), with the notation of Proposition~\ref{prop:4}.  Hence, $\iota(f)\in A^{p,q}_{\vect s,0}(D)$ (resp.\ $\iota(f)\in A^{p,q}_{\vect s}(D)$), so that $f=\iota(f)(0,\,\cdot\,)\in A^{p,q}_{\vect s-\vect b/p,0}(F+i\Omega)$ (resp.\ $f=\iota(f)(0,\,\cdot\,)\in A^{p,q}_{\vect s-\vect b/p}(F+i\Omega)$) by Proposition~\ref{prop:5}. The proof is complete. 
\end{proof}

\section{Proof of Proposition~\ref{prop:13}}\label{Proof:sec}

We now present the proof of Proposition~\ref{prop:13}. We begin with
some useful lemmas. 

\begin{lem}\label{lem:2}
	Let $P$ be a (holomorphic) polynomial on $\C^n$ of degree $k$,
        and let $C$ be an open convex subset of $\C^n$ contained in
        $\Set{z\in \C^n\colon P(z)\neq 0}$. Then, 
	\[
	\osc_C(\Im \log P)\meg k\pi. 
	\]
\end{lem}

Here, $\osc_C f$ denotes the oscillation of a function $f$ on the set
$C$, that is, the diameter of $f(C)$. Notice that $\log P$ is
(ambigously) defined on $C$ since $C$ is convex, and that $\osc_C(\Im
\log P)$ is unambiguously defined. 

\begin{proof}
	We may assume that $C\neq \emptyset$, so that $P\neq 0$.
	
	\textsc{Step I.} Assume first that $n=1$. Then, there are $a, z_1,\dots, z_k\in \C$, $a\neq 0$, such that 
	\[
	P(z)=a \prod_{j=1}^k (z-z_j)
	\]
	for every $z\in \C$. In addition, $z_j\not \in C$ for every
        $j=1,\dots, k$. Choose $a'\in \C$ so that $\ee^{a'}=a$ and,
        for every $j=1,\dots, k$, choose $\alpha_j\in \R$ such that 
	\[
	C-z_j\subseteq \ell_j\coloneqq \Set{z\in \C\colon \Re(z \ee^{i\alpha_j})>0},
	\]  
	and define
	\[
	\log_j (z)\coloneqq \log\abs{z}+A_j(z),
	\]
	for every $z\in \ell_j$, where $A_j(z)$ is the unique element of $(-\pi/2,\pi/2)-\alpha_j$ such that $\ee^{\log_j(z)}=z$.
	Therefore, we may assume that 
	\[
	\log P(z)= a'+\sum_{j=1}^k \log_j(z-z_j)
	\]
	for every $z\in C$. Then,
	\[
	\osc_C (\Im \log P)\meg \sum_{j=1}^k \osc_C(\Im \log_j (\,\cdot\,-z_j))\meg k \pi
	\]
	whence the result in this case.
	
	\textsc{Step II.} It suffices to observe that
	\[
	\osc_C(\Im \log P)=\sup_\ell \osc_{C\cap \ell}(\Im \log P),
	\]
	where $\ell$ runs through the set of (complex) affine lines in $\C^n$ which meet $C$, and to apply~\textsc{step I}.
\end{proof}

For every $\vect s\in \C^r$, we  define $\log \Delta^{\vect s}_\Omega$ as the unique holomorphic function $f$ on $\Omega+i F$ such that $f(e_\Omega)=0$ and $\ee^f=\Delta^{\vect s}_\Omega$ on $\Omega+ i F$.  Notice that
\begin{equation}\label{eq:3}
\log \Delta^{\alpha \vect s+\alpha'\vect s'}_\Omega=\alpha\log \Delta^{\vect s}_\Omega+\alpha' \log \Delta^{\vect s'}_\Omega
\end{equation}
for every $\vect s,\vect s'\in \C^r$ and for every $\alpha,\alpha'\in \C$, since both sides of the asserted equality are holomorphic functions which vanish at $e_\Omega$ and whose exponential is $\Delta^{\alpha \vect s+\alpha'\vect s'}_\Omega$ on $\Omega$, hence on $\Omega+i F$ by holomorphy.

We define 
\[
\abs{\vect s}\coloneqq \sum_{j,j'} \abs{\alpha_j} s_{j,j'} 
\]
for every $\vect s\in \R^r$, where $\vect s_j=(s_{j,1},\dots,
s_{j,r})$, $j=1,\dots, r$, is a `basis' of $\N_\Omega$ over $\N$
(hence of $\C^r$ over $\C$), and $\vect s=\sum_j \alpha_j \vect
s_j$ (such a `basis' exists by~\cite[Theorem 2.2]{Ishi}). Notice that $(\vect s_j)$ is uniquely determined up to the
order, so that $\abs{\vect s}$ is well defined. 
In addition,
\[
\sum_j \abs{s_j}\meg \abs{\vect s},
\]
with equality if and only if $\vect s$ or $-\vect s$ belongs to the closed convex cone generated by $\N_\Omega$.

\begin{cor}\label{cor:12}
	Take $\vect s\in \R^r$. Then, $\Im \log \Delta_\Omega^{\vect s}$ is bounded on $\Omega+i F$. More precisely, $\osc_{\Omega+i F}(\Im \log \Delta_\Omega^{\vect s})\meg \abs{\vect s} \pi$.
\end{cor}

When $\Omega$ is (irreducible and) symmetric,~\cite[Lemma 7.3]{Garrigos} shows that $\Re \Delta_\Omega^{\vect e_j}(z)>0$ for every $z\in \Omega+i F$, where $\vect e_j=(\delta_{j,j'})_{j'=1,\dots, r}$.   In other words, one may replace $\abs{\vect s}$ with $\sum_j \abs{s_j}$ in the above result.   Nonetheless, the analogous estimates do \emph{not} extend to all homogeneous cones. Counterexamples arise already when $r=3$ and $m=8$.

\begin{proof}
	The assertion follows from Lemma~\ref{lem:2} for $\vect s\in \N_\Omega$. The assertion then follows from~\eqref{eq:3} and the definition of $\abs{\vect s}$.
\end{proof}

\begin{lem}\label{lem:6}
	Let $P$ be a polynomial on $\C$ of degree $k$ whose zeroes are contained in $\R_-$. Then,
	\[
	\frac{1}{2^{k/2}} \meg \frac{\abs{P(\pm i x)}}{P(x)}\meg 1
	\]
	for every $x>0$.
\end{lem}

\begin{proof}
	Let $x_1,\dots, x_k$ be the (not necessarily distinct) zeroes of $P$, so that there is a non-zero $a\in \C$ such that
	\[
	P(z)=a\prod_{j=1}^k(z-x_j)
	\]
	for every $z\in \C$. Observe that
	\[
	\frac{\abs{\pm i x-x_j}}{x-x_j}=\frac{\sqrt{x^2+x_j^2}}{x+\abs{x_j}}\in [2^{-1/2},1]
	\]
	for every $j=1,\dots, k$ and for every $x>0$, since $x_j\meg 0$.  The assertion follows.
\end{proof}

\begin{cor}\label{cor:13}
	Take $\vect s\in \R^r$. Then,
	\[
	2^{-\abs{\vect s}/2}\meg \frac{\abs{ \Delta_\Omega^{\vect
              s}  (x\pm i y)}}{ \Delta_\Omega ^{\vect s}(x+y)}\meg 2^{\abs{\vect s}/2}
	\]
	for every $x,y\in \Omega$.
\end{cor}

\begin{proof}
	Fix $x,y\in \Omega$ and assume first that $\vect s\in \N_\Omega$, that is, that $\Delta^{\vect s}_\Omega$ is polynomial. Define
	\[
	P\colon \C\ni w\mapsto \Delta^{\vect s}_\Omega(x+ w y)\in \C,
	\]
	so that $P$ is a holomorphic polynomial on $\C$. Since $\Delta^{\vect s}_\Omega$ is homogeneous of degree $\sum_j s_j$, the degree of $P$ is $\sum_j s_j$. Observe that $P(w)=\Delta^{\vect s}_\Omega(x+w y)\neq 0$ for every $w\in \C$ with $\Re w >0$, since clearly
	\[
	\Delta^{\vect s}_\Omega(x+ w y)\Delta^{-\vect s}_\Omega(x+w y)=1.
	\]
	In addition, observe that
	\[
	P(\pm i w)=(\pm i)^{\sum_j s_j}\Delta^{\vect s}_\Omega(w y\mp i x )\neq0
	\]
	for every $w\in \C$ with $\mp \Im w=\Re(\pm i w)>0$, for the same reason as above. Consequently, the zeroes of $P$ are contained in $\R_-$, so that Lemma~\ref{lem:6} implies that the assertion holds in this case.
	
	The general case follows from~\eqref{eq:3} and the definition of $\abs{\vect s}$.	
\end{proof}

\begin{lem}\label{lem:7}
	Take $\vect s_1,\vect s_2,\vect s_3\in \R^r $ such that $\vect s_2\succ \vect 0$, and $\alpha\in \R$. Then,
	\[
	\int_\Omega \Delta^{\vect s_1}_\Omega(h+ h')(1+\abs{\log \Delta^{\vect s_2}_\Omega(h+ h')})^\alpha \Delta^{\vect s_3}_\Omega(h')\,\dd \nu_\Omega(h')<\infty,
	\]
	for some (equivalently, every) $h\in \Omega$, if and only if $\vect s_3\succ \frac 1 2 \vect m$ and either $\vect s_1+\vect s_3\meg -\frac 1 2 \vect m'$ and $\alpha<-1$, or $\vect s_1+\vect s_3\prec -\frac 1 2 \vect m'$.
\end{lem}

\begin{proof}
	Observe that, by~\cite[Corollary 2.49]{CalziPeloso}, there is $\eps>0$ such that
	\[
	\abs*{\frac{\Delta_\Omega^{\vect s_j}(h+h')}{\Delta_\Omega^{\vect s_j}(h+h'')}-1}<\frac 1 4
	\]
	for every $h\in \overline \Omega$ and for every $h',h''\in \Omega$ with $d_\Omega(h',h'')\meg \eps$, and for $j=1,2,3$.
	Consequently, if we define 
	\[
	I(h)\coloneqq \int_\Omega \Delta_\Omega^{\vect s_1}(h+h')(1+\abs{\log\Delta_\Omega^{\vect s_2}(h+h')})^{\alpha}\Delta_\Omega^{\vect s_3}(h')\,\dd \nu_\Omega(h')
	\]
	for every $h\in \Omega$, we have
	\[
	\frac{3\min\left((1-\log(4/3))^\alpha, (1+\log(4/3))^\alpha\right)}{4} I(h')\meg I(h)\meg \frac{5\max\left( (1-\log(4/3))^\alpha, (1+\log(4/3))^\alpha  \right)}{4} I(h')
	\]
	for every $h,h'\in \Omega$ with $d_\Omega(h,h')\meg \eps$. Consequently, $I(h)$ is finite for every $h\in \Omega$ if and only if $I(h)$ is finite for \emph{some} $h\in \Omega$.
	Next, observe that the function
	\[
	(1+\abs{\log \Delta^{\vect s_2}_\Omega})\Delta^{-\eps \vect s_2}_\Omega
	\]
	is bounded above on $h+\Omega$ for every $h\in \Omega$ and for every $\eps>0$, since
	\[
	 \Delta^{\vect s_2}_\Omega(h+h')\Meg  \Delta_\Omega^{\vect s_2}(h)>0
	\]
	for every $h'\in \Omega$ (cf.~\cite[Corollary 2.36]{CalziPeloso}). Therefore, by means of~\cite[Corollary 2.22]{CalziPeloso}, we see that $I(e_\Omega)$ (say) is finite when $\vect s_3\succ \frac 1 2 \vect m$ and $\vect s_1+\vect s_3\prec -\frac 1 2 \vect m'$, and that $I(e_\Omega)$ is infinite when either $\vect s_3\not \succ \frac 1 2 \vect m$, or $\vect s_1+\vect s_3\not \meg -\frac 1 2 \vect m'$, or $\vect s_1+\vect s_3\not \prec -\frac 12 \vect m'$ and $\alpha \Meg 0$. Consequently, we may reduce to the case in which $\vect s_3\succ \frac 1 2 \vect m$, $\vect s_1+\vect s_3\not \prec -\frac 1 2 \vect m'$,  $\vect s_1+\vect s_3 \meg -\frac 1 2 \vect m'$, and $\alpha<0$, and to prove that $I(e_\Omega)$ is finite if and only if $\alpha<-1$.
	
	In order to simplify the proof, we shall now identify $T_+$ with $\Omega$ by means of the mapping $t\mapsto t\cdot e_\Omega$, so that $hh'$ and $h^{-1}$ are defined for every $h,h'\in \Omega$. Then, using~\cite[Lemma 2.18]{CalziPeloso}, we see that
	\[
	\begin{split}
		I(h)&=\int_{h+\Omega} \Delta_\Omega^{\vect s_1-\vect d}(h')(1+\abs{\log\Delta_\Omega^{\vect s_2}(h')})^{\alpha}\Delta_\Omega^{\vect s_3+\vect d}(h'-h)\,\dd \nu_\Omega(h')\\
		&=\int_{h'^{-1}h\in h+\Omega}  \Delta_\Omega^{\vect s_1-\vect d+(\vect m'-\vect m)/2}(h'^{-1}h)(1+\abs{\log\Delta_\Omega^{\vect s_2}(h'^{-1}h)})^{\alpha}\Delta_\Omega^{\vect s_3+\vect d}(h'^{-1}h-h)\,\dd \nu_\Omega(h')\\
		&=\Delta_\Omega^{\vect s_3+2 \vect s_1+\vect d+\vect m'-\vect m}(h)\int_{h' h\in  \Omega\cap (h-\Omega)} \Delta_\Omega^{-\vect s_1-\vect s_3+(\vect m-\vect m')/2}(h' h)(1+\abs{\log\Delta_\Omega^{\vect s_2}(h'h)-\log \Delta_\Omega^{2 \vect s_2}(h)})^{\alpha}\\
			&\qquad \times\Delta_\Omega^{\vect s_3+\vect d}(h-h' h)\,\dd \nu_\Omega(h')\\
		&=\Delta_\Omega^{\vect s_3+2 \vect s_1+\vect d+(\vect m'-\vect m)/2}(h)\int_{\Omega\cap (h-\Omega)} \Delta_\Omega^{-\vect s_1-\vect s_3+(\vect m-\vect m')/2}(h')(1+\abs{\log\Delta_\Omega^{\vect s_2}(h')-\log \Delta_\Omega^{2 \vect s_2}(h)})^{\alpha}\\
			&\qquad \times\Delta_\Omega^{\vect s_3+\vect d}(h-h' )\,\dd \nu_\Omega(h')
	\end{split}
	\]
	for every $h\in \Omega$. Now, fix $\eps'>0$ such that 
	\[
	Q(\eps')\coloneqq[(1-\eps')e_\Omega+\Omega]\cap[(1+\eps')e_\Omega-\Omega]\subseteq B_\Omega(e_\Omega,\eps)
	\]
	and such that 
	\[
	\abs{\Delta_\Omega^{2 \vect s_2}(h)}\meg \frac 12
	\] 
	for every $h\in Q(\eps')$,
	so that, in particular, 
	\[
	\frac 1 2+\abs{\log\Delta_\Omega^{\vect s_2}(h')}\meg
	1+\abs{\log\Delta_\Omega^{\vect s_2}(h')-\log \Delta_\Omega^{2 \vect s_2}(h)}\meg\frac 3 2+\abs{\log\Delta_\Omega^{\vect s_2}(h')}
	\] 
	for every $h\in Q(\eps')$ and for every $h'\in\Omega$. Then, setting $C_1\coloneqq 2^{-\alpha}\max_{Q(\eps')}\Delta_\Omega^{\vect s_3+2 \vect s_1+\vect d+(\vect m'-\vect m)/2} $,
	\[
	\begin{split}
		&\int_{Q(\eps')} I(h)\,\dd h\meg C_1 \int_{Q(\eps')}\int_{\Omega\cap (h-\Omega)} \Delta_\Omega^{\vect d-\vect s_1-\vect s_3+(\vect m-\vect m')/2}(h')(1+\abs{\log\Delta_\Omega^{\vect s_2}(h')})^{\alpha}\Delta_\Omega^{\vect s_3+\vect d}(h-h' )\,\dd h'\,\dd h\\
		&\meg C_1 \int_{\Omega\cap ((1+\eps')e_\Omega-\Omega)}\Delta_\Omega^{\vect s_3}(h )\,\dd \nu_\Omega(h)  \int_{\Omega\cap ((1+\eps')e_\Omega-\Omega)} \Delta_\Omega^{-\vect s_1-\vect s_3+(\vect m-\vect m')/2}(h')(1+\abs{\log\Delta_\Omega^{\vect s_2}(h')})^{\alpha}\,\dd\nu_\Omega(h').
	\end{split}
	\]
	Now,~\cite[Proposition 2.19]{CalziPeloso} shows that 
	\[
	\int_{\Omega\cap ((1+\eps')e_\Omega-\Omega)}\Delta_\Omega^{\vect s_3}(h )\,\dd \nu_\Omega(h) <\infty.
	\]
	A direct computation (cf.~the proof of~\cite[Proposition 2.19]{CalziPeloso}) then shows that 
	\[
	 \int_{\Omega\cap ((1+\eps')e_\Omega-\Omega)} \Delta_\Omega^{-\vect s_1-\vect s_3+(\vect m-\vect m')/2}(h')(1+\abs{\log\Delta_\Omega^{\vect s_2}(h')})^{\alpha}\,\dd\nu_\Omega(h')<\infty
	\] 
	for $\alpha<-1$. Therefore, $I(h)$ is finite for some $h\in Q(\eps')$, hence for every $h\in \Omega$ by the preceding remarks, provided that $\alpha<-1$.
	
	Conversely, if $\alpha\Meg -1$, then a direct computation (cf.~the proof of~\cite[Proposition 2.19]{CalziPeloso})  shows that
	\[
	\begin{split}
	I(e_\Omega)&\Meg \int_{\Omega\cap (e_\Omega/2-\Omega)} \Delta_\Omega^{-\vect s_1-\vect s_3+(\vect m-\vect m')/2}(h')(1+\abs{\log\Delta_\Omega^{\vect s_2}(h')})^{\alpha}\Delta_\Omega^{\vect s_3+\vect d}(e_\Omega-h' )\,\dd \nu_\Omega(h')\\
		&\Meg 2^{\sum_j (-d_j+s_{3,j})_+} \int_{\Omega\cap (e_\Omega/2-\Omega)} \Delta_\Omega^{-\vect s_1-\vect s_3+(\vect m-\vect m')/2}(h')(1+\abs{\log\Delta_\Omega^{\vect s_2}(h')})^{\alpha}\,\dd \nu_\Omega(h')=\infty
	\end{split}
	\]
	(cf.~\cite[Corollary 2.36]{CalziPeloso}), whence the conclusion by the preceding remarks.
\end{proof}

\begin{proof}[Proof of Proposition \ref{prop:13}]
	By~\cite[Proposition 5.18]{CalziPeloso} and the necessary condition $\vect s \succ \frac{1}{2 q}\vect m\Meg \vect 0$ (cf.~\cite[Proposition 3.5]{CalziPeloso}), it will suffice to prove the first assertion for $p<q$ and $\vect s \Meg \left(\frac{1}{2p}-\frac{1}{2 q}\right)\vect m'$. The second assertion is a consequence of~\cite[Proposition 5.18]{CalziPeloso}.
	Observe first that, if $\vect s'\Meg \vect 0$, then~\cite[Corollary 2.36 and Lemma 2.37]{CalziPeloso} imply that
	\[
	\abs{\Delta^{\vect s'}_\Omega(h+z)}\Meg 
        \Delta_\Omega^{\vect s'}  (h+\Re z)\Meg 
        \Delta_\Omega^{\vect s'} ~ (h) 
	\]
	for every $h\in \Omega$ and for every $z\in \Omega+i F$. Consequently, $\abs{B^{\vect s'}_{(0,i e_\Omega)}(\zeta,z)}\Meg 1$ for every $(\zeta,z)\in D$, so that there is a unique $f\in \Hol(D)$ such that $f(0,i e_\Omega)=0$ and 
	\[
	\exp(\exp(f))=\ee B^{\vect s'}_{(0, i e_\Omega)}
	\]
	on $D$. We denote this function by $\log(1+\log B^{\vect s'}_{(0, i e_\Omega)})$. Thus, we may also write $(1+\log B^{\vect s'}_{(0,i e_\Omega)})^{s''}$ instead of  $\ee^{s'' f}$, for every $s''\in \C$.
	Define $g^{\vect s_1,s_2}\coloneqq B^{\vect s_1}_{(0,i e_\Omega)} (1+\log B^{-\vect s_1}_{(0,i e_\Omega)})^{s_2}$ for every $\vect s_1\meg \vect 0$ and for every $s_2\in \R$. 
	
	Now, take $\vect s_3\in \N_{\Omega'}$, and define 
	\[
	p_{\vect s_1,\vect s_3}\colon \C\ni w \mapsto i^{\vect s_3} \left( w \vect s_1+\frac 1 2 \vect m'\right)_{\vect s_3}\coloneqq \prod_{j=1}^r i^{s_{3,j}} \Big(w s_{1,j}+\frac 1 2 m'_j\Big)\cdots \Big(w s_{1,j}-s_{3,j}+\frac 1 2 m'_j+1\Big)\in \C,
	\]
	so that $p_{\vect s_1,\vect s_3}$ is polynomial and
	\[
	B^{k \vect s_1}_{(0,i e_\Omega)}* I^{-\vect s_3}_\Omega= p_{\vect s_1,\vect s_3}(k) B^{k \vect s_1-\vect s_3}_{(0,i e_\Omega)}=p_{\vect s_1,\vect s_3}(k) B^{k \vect s_1}_{(0,i e_\Omega)} B^{-\vect s_3}_{(0,i e_\Omega)}
	\]
	for every $k\in \N$, thanks to~\cite[Proposition 2.29]{CalziPeloso}. Define $P_{\vect s_1,\vect s_3} $ as the differential operator $p_{\vect s_1,\vect s_3}(\Rc)$ on $\C$, where $\Rc f(w)=w f'(w)$ for every $w\in \C$ and for every $f\in \Hol(\C)$. Then, for every holomorphic polynomial $\varphi$ on $\C$ we see that
	\[
	(\varphi\circ B^{\vect s_1}_{(0,i e_\Omega)})* I^{-\vect s_3}_\Omega =[(P_{\vect s_1,\vect s_3} \varphi)\circ  B^{\vect s_1}_{(0,i e_\Omega)}] B^{-\vect s_3}_{(0,i e_\Omega)}.
	\]
	By approximation, the preceding formula holds also  with $\varphi^{(s_2)}\colon w \mapsto w (1-\log w)^{s_2}$, defined in a neighbourhood of $1$, in place of $\varphi$. Observe that $P_{\vect s_1,\vect s_3}\varphi^{(s_2)}(w)=\sum_{k\in\N} a_k w(1-\log w)^{s_2-k}$ for some $(a_k)\in \R^{(\N)}$ and for every $w$ in a neighbourhood of $1$. Therefore,  by holomorphy we see that
	\[
	g^{\vect s_1,s_2}*I^{-\vect s_3}_\Omega =B^{-\vect s_3}_{(0,i e_\Omega)} \sum_{k\in \N} a_k g^{\vect s_1, s_2-k}.
	\]
	Observe that $(1+\log B^{-\vect s_1}_{(0,i e_\Omega)})^{s_2}$ is bounded for every $s_2\meg 0$ by the preceding remarks, so that $g^{\vect s_1,s_2}*I^{-\vect s_3}_\Omega\in A^{p,q}_{\vect s_1+\vect s_3}(D)$ if $ B^{-\vect s_3}_{(0,i e_\Omega)} g^{\vect s_1, s_2}\in A^{p,q}_{\vect s_1+\vect s_3}(D)$.
	
	Now, observe that the preceding remarks show that
	\[
	\begin{split}
	\Re \log(1+\log B^{-\vect s_1}_{(0,i e_\Omega)})(\zeta,z)&\Meg \log (1+\log \abs{B^{-\vect s_1}_{(0, i e_\Omega)}(\zeta,z)})\\
		&=\log(1+\log \abs{\Delta_\Omega^{-\vect s_1}(e_\Omega+z/i)})\\
		&\Meg \log (1+\log \Delta^{-\vect s_1}_\Omega(e_\Omega+\Im z)),
	\end{split}
	\]
	so that
	\[
	\begin{split}
		\abs{(1+\log B^{-\vect s_1}_{(0,i e_\Omega)})^{s_2}(\zeta,z)}\meg (1+\log \Delta^{-\vect s_1}_\Omega(e_\Omega+\Im z))^{s_2}
	\end{split}
	\]
	for every $(\zeta,z)\in D$, when  $s_2\meg 0$. Then,
	\[
	\begin{split}
	\norm{(g^{\vect s_1,s_2} B^{-\vect s_3}_{(0,i e_\Omega)})_h}_{L^p(\Nc)}&\meg \norm{(B^{\vect s_1-\vect s_3}_{(0,i e_\Omega)})_h}_{L^p(\Nc)} (1+\log \Delta^{-\vect s_1}_\Omega(e_\Omega+h))^{s_2}\\
		&=C_{\vect s_1-\vect s_3,p} \Delta^{\vect s_1-\vect s_3-(\vect b+\vect d)/p}_\Omega(e_\Omega+h)(1+\log \Delta^{-\vect s_1}_\Omega(e_\Omega+h))^{s_2}
	\end{split}
	\]
	for a suitable constant $C_{\vect s_1-\vect s_3,p}>0$ and for every $h\in \Omega$, provided that $\vect s_1-\vect s_3\prec\frac 1p (\vect b+\vect d)-\frac{1}{2 p} \vect m'$ (cf.~\cite[Corollary 2.36 and Lemma 2.39]{CalziPeloso}). Then, let us prove that
	\[
	\norm{\Delta^{\vect s_1-\vect s_3-(\vect b+\vect d)/p}_\Omega(e_\Omega+\,\cdot\,)(1+\log \Delta^{-\vect s_1}_\Omega(e_\Omega+\,\cdot\,))^{s_2} \Delta_\Omega^{\vect s+\vect s_3}   }_{L^q(\nu_\Omega)}
	\]
	is finite if $ \vect s +\vect s_3\succ \frac{1}{2 q} \vect m$, $ \vect s+\vect s_1\meg \frac 1 p (\vect b+\vect d)-\frac{1}{2 q}\vect m'$, and $s_2<-1/q$.
	Observe that the assertion is trivial when $q=\infty $ (in which case it suffices to assume that $\vect s+\vect s_3\Meg \vect 0$, $\vect s+\vect s_1\meg \frac 1 p (\vect b+\vect d)$ and $s_2 \meg 0$). The case $q<\infty$ is a consequence of Lemma~\ref{lem:7}.
	
 Now, assume that $ \vect s+\vect s_1\meg \frac 1 p (\vect b+\vect
 d)-\frac{1}{2 q}\vect m'$, and $s_2<-1/q$, and let us prove that
 $g^{\vect s_1,s_2}\in \widetilde A^{p,q}_{\vect s}(D) $. By the
 preceding remarks and Proposition~\ref{prop:1}, if $\vect s_3 \in
 \N_{\Omega'}$ is sufficiently large, then
 $$
 g^{\vect s_1,s_2}* I^{-\vect s_3}_\Omega\in A^{p,q}_{\vect s+\vect
   s_3}(D)=\widetilde A^{p,q}_{\vect s+\vect s_3}(D).
 $$
 In addition, there is $p_1\in (p,\infty)$ such that $\vect s\succ
 \frac{1}{p_1}(\vect b+\vect d)+\frac{1}{2 q'}\vect m'$ and  $\vect
 s+\vect s_1\prec\frac{1}{p_1}(\vect b+\vect d)-\frac{1}{2 q}\vect
 m'$, so that $\vect s_1 \prec \frac{1}{p_1}(\vect b+\vect
 d)-\frac{1}{2p_1}\vect m'$ by the assumptions on $\vect s$, and
 $g^{\vect s_1,s_2}\in A^{p_1,q}_{\vect s}(D)\subseteq \widetilde
 A^{p_1,q}_{\vect s}(D)$ thanks to~\cite[Proposition
 2.41]{CalziPeloso} and the preceding remarks. Now, denote by $g^{\vect s_1,s_2}_0$ the boundary
 values of $g^{\vect s_1,s_2}$ in $B^{-\vect s}_{p_1,q}(\Nc,\Omega)$,
 so that $g^{\vect s_1,s_2}_0*I^{-\vect s_3}_\Omega=(g^{\vect
   s_1,s_2}*I^{-\vect s_3}_\Omega)_0\in B^{-\vect s-\vect
   s_3}_{p,q}(\Nc,\Omega)$ by~\cite[Proposition
 5.13]{CalziPeloso} and the preceding remarks. Since the mapping $u\mapsto u* I^{-\vect
   s_3}_\Omega$ induces an isomorphism of $\Sc'_{\Omega,L}(\Nc)$ which
 induces isomorphisms of $B^{-\vect s}_{p,q}(\Nc,\Omega)\to B^{-\vect
   s-\vect s_3}_{p,q}(\Nc,\Omega)$ and $B^{-\vect
   s}_{p_1,q}(\Nc,\Omega)\to B^{-\vect s-\vect
   s_3}_{p_1,q}(\Nc,\Omega)$ (cf.~\cite[Proposition 4.11, by
 transposition, and Theorem 4.26]{CalziPeloso}), this implies that
 $g^{\vect s_1,s_2}_0\in B^{-\vect s}_{p,q}(\Nc,\Omega)$. Since
 $g^{\vect s_1,s_2}=\Ec g^{\vect s_1,s_2}_0$ (and $\Ec$ is defined in
 the same way in $B^{-\vect s}_{p_1,q}(\Nc,\Omega) $ and in $B^{-\vect
   s}_{p,q}(\Nc,\Omega)$), this implies that $g^{\vect s_1,s_2}\in
 \widetilde A^{p,q}_{\vect s}(D)$.

	Finally, let us prove that, if $g^{\vect s_1,s_2}\in A^{p,q}_{\vect s}(D)$ and $-1/p\meg s_2<-1/q$, then $\vect s_1\prec \frac{1}{p}(\vect b+\vect d)-\frac{1}{2 p}\vect m'$. This will lead to the conclusion.
	In order to prove our claim, it will suffice to prove that, if $g^{\vect s_1,s_2}_{e_\Omega}\in L^p(\Nc)$ and $-1/p\meg s_2<-1/q$, then  $\vect s_1\prec \frac{1}{p}(\vect b+\vect d)-\frac{1}{2 p}\vect m'$. 
	By Corollaries~\ref{cor:12} and~\ref{cor:13}, there are constants $C_2, C_3,C_4>0$ such that 
	\[
	\begin{split}
	(1+\log B^{-\vect s_1}_{(0,i e_\Omega)})^{s_2} (\zeta,h+i e_\Omega+i\Phi(\zeta))&\Meg C_2 (1+\abs{\log \abs{ \Delta_\Omega^{-\vect s_1}(e_\Omega+(\Phi(\zeta)-i h)/2) }})^{s_2}\\
		&\Meg C_3 (1+\abs{\log \Delta_\Omega^{-\vect s_1}(e_\Omega+\Phi(\zeta)/2+h/2 ) })^{s_2}
	\end{split}
	\]
	and then
	\[
	\abs{g^{\vect s_1,s_2}(\zeta,h+i e_\Omega+i\Phi(\zeta))}\Meg C_4 \Delta^{\vect s_1}_\Omega(e_\Omega+\Phi(\zeta)/2+h/2 )(1+\abs{\log \Delta_\Omega^{-\vect s_1}(e_\Omega+\Phi(\zeta)/2+h/2 ) })^{s_2}
	\]
	for every $\zeta\in E$ and for every $h\in \Omega $. Therefore, it will suffice to prove that
	\[
	\int_{E\times\Omega} \Delta^{p\vect s_1}_\Omega(e_\Omega+\Phi(\zeta)/2+h/2 )(1+\abs{\log \Delta_\Omega^{-\vect s_1}(e_\Omega+\Phi(\zeta)/2+h/2 ) })^{p s_2}\,\dd (\zeta,h)=\infty
	\]
	if $\vect s_1 \not \prec \frac 1 p (\vect b+\vect d)-\frac{1}{2 p}\vect m'$.
	Observe first that, by homogeneity,
	\[
	\begin{split}
	&\int_{E} \Delta^{p\vect s_1}_\Omega(e_\Omega+\Phi(\zeta)/2+h/2 )(1+\abs{\log \Delta_\Omega^{-\vect s_1}(e_\Omega+\Phi(\zeta)/2+h/2 ) })^{p s_2}\,\dd \zeta\\
		&\qquad= \Delta^{p\vect s_1-\vect b}_\Omega(e_\Omega+h/2 )\int_{E} \Delta^{p\vect s_1}_\Omega(e_\Omega+\Phi(\zeta)/2)(1+\abs{\log \Delta_\Omega^{-\vect s_1}(e_\Omega+\Phi(\zeta)/2) +\log\Delta_\Omega^{-\vect s_1}(e_\Omega+h/2)} )^{p s_2}\,\dd \zeta\\
		&\qquad\Meg C_5 \Delta^{p\vect s_1-\vect b}_\Omega(e_\Omega+h/2 ) (1 +\abs{\log\Delta_\Omega^{-\vect s_1}(e_\Omega+h/2) })^{p s_2},
	\end{split}
	\]
	where
	\[
	C_5\coloneqq \frac 1 2\int_{B_E(0,\eps)} \Delta^{p\vect s_1}_\Omega(e_\Omega+\Phi(\zeta)/2)\,\dd \zeta
	\]
	and $\eps>0$ is chosen so that $\abs{\log \Delta^{-\vect s_1}_\Omega(e_\Omega+\Phi(\zeta)/2)}\meg 1$ for every $\zeta\in B_E(0,\eps)$. Then,
	\[
	\begin{split}
	&\int_{E\times\Omega} \Delta^{p\vect s_1}_\Omega(e_\Omega+\Phi(\zeta)/2+h/2 )(1+\abs{\log \Delta_\Omega^{-\vect s_1}(e_\Omega+\Phi(\zeta)/2+h/2 ) })^{p s_2}\,\dd (\zeta,h)\\
		&\qquad \Meg C_5\int_\Omega\Delta^{p\vect s_1-\vect b-\vect d}_\Omega(e_\Omega+h/2 ) (1 +\abs{\log\Delta_\Omega^{-\vect s_1}(e_\Omega+h/2) })^{p s_2}\,\dd \nu_\Omega(h)
	\end{split}
	\]
	whence the result by Lemma~\ref{lem:7}.	
\end{proof}

\end{document}